\documentclass[a4paper,11pt]{article}
\usepackage{okicmd}
\usepackage{amsthm}
\usepackage[top=80pt, bottom=90pt, left=70pt, right=70pt]{geometry}
\usepackage{comment}
\usepackage{color}
\usepackage{multirow}
\usepackage[colorlinks=true,linkcolor=blue,citecolor=blue]{hyperref}
\usepackage{cleveref}
\usepackage{kbordermatrix}
\usepackage{blkarray}
\usepackage{cases}
\usepackage{enumitem}

\def\qqed{\hfill $\blacksquare$}

\DeclareMathOperator{\ncrank}{nc-rank}
\DeclareMathOperator{\kerL}{\ker_{\rm L}}
\DeclareMathOperator{\kerR}{\ker_{\rm R}}
\newcommand{\calC}{\mathcal{C}}
\newcommand{\calE}{\mathcal{E}}
\newcommand{\calF}{\mathcal{F}}
\newcommand{\calG}{\mathcal{G}}
\newcommand{\calP}{\mathcal{P}}
\newcommand{\calQ}{\mathcal{Q}}
\newcommand{\calR}{\mathcal{R}}
\newcommand{\calS}{\mathcal{S}}
\newcommand{\calT}{\mathcal{T}}
\newcommand{\calM}{\mathcal{M}}
\newcommand{\calV}{\mathcal{V}}
\newcommand{\calU}{\mathcal{U}}
\newcommand{\Ni}{{\rm N}_{\rm inner}}
\newcommand{\No}{{\rm N}_{\rm outer}}
\newcommand{\labeling}{(\{U_\alpha^+, U_\alpha^-\}, \{V_\beta^+, V_\beta^-\})_{\alpha, \beta}}

\usepackage{tikz}
\usetikzlibrary{positioning}
\usetikzlibrary{intersections}
\usetikzlibrary{calc}
\usetikzlibrary{arrows.meta}
\usetikzlibrary{shapes.geometric}
\usetikzlibrary{shapes.misc}
\usetikzlibrary {decorations.pathmorphing}

\newcounter{nodecount}
\providecommand{\nodecounter}[1]{
	\setcounter{nodecount}{-1}
	\foreach \iter in #1{
		\stepcounter{nodecount}
	}
}

\newtheorem{theorem}{\bfseries Theorem}[section] 
\newtheorem{lemma}[theorem]{\bfseries Lemma}
\newtheorem{proposition}[theorem]{\bfseries Proposition}

\newtheorem*{claim*}{\bfseries Claim}

\theoremstyle{definition}
\newtheorem{remark}[theorem]{\bfseries Remark} 

\crefname{theorem}{Theorem}{Theorems}
\crefname{proposition}{Proposition}{Propositions}
\crefname{lemma}{Lemma}{Lemmas}
\crefname{exmp}{Example}{Examples}
\crefname{corollary}{Corollary}{Corollarys}
\crefname{claim}{Claim}{Claims}
\crefname{remark}{Remark}{Remarks}
\crefname{section}{Section}{Sections}

\makeatletter
\renewcommand{\theequation}{%
	\thesection.\arabic{equation}}
\@addtoreset{equation}{section}
\makeatother

\usepackage{bm}
\allowdisplaybreaks

\begin{document}
	\title{A combinatorial algorithm for computing the entire sequence of the maximum degree of minors of a generic partitioned polynomial matrix with $2 \times 2$ submatrices\thanks{A preliminary version of this paper~\cite{IPCO/I21} has appeared in the proceedings of the 22nd Conference on Integer Programming and Combinatorial Optimization (IPCO 2021).}}
	\author{Yuni Iwamasa\thanks{Department of Communications and Computer Engineering, Graduate School of Informatics, Kyoto University, Kyoto 606-8501, Japan.
			Email: \texttt{iwamasa@i.kyoto-u.ac.jp}}}
	\date{\today}
	\maketitle
	
\begin{abstract}
	In this paper, we consider the problem of computing the entire sequence of the maximum degree of minors of a block-structured symbolic matrix (a generic partitioned polynomial matrix)
	$A = (A_{\alpha\beta} x_{\alpha \beta} t^{d_{\alpha \beta}})$,
	where $A_{\alpha\beta}$ is a $2 \times 2$ matrix over a field $\mathbf{F}$, $x_{\alpha \beta}$ is an indeterminate,
	and $d_{\alpha \beta}$ is an integer for $\alpha = 1,2,\dots, \mu$ and $\beta = 1,2,\dots,\nu$,
	and $t$ is an additional indeterminate.
	This problem
	can be viewed as an algebraic generalization of
	the maximum weight bipartite matching problem.
	
	The main result of this paper is a combinatorial $O(\mu \nu \min\set{\mu, \nu}^2)$-time algorithm for computing the entire sequence of the maximum degree of minors
	of a $(2 \times 2)$-type generic partitioned polynomial matrix of size $2\mu \times 2\nu$.
	We also present a minimax theorem, which can be used as a good characterization (NP $\cap$ co-NP characterization) for the computation of the maximum degree of minors of order $k$.
	Our results generalize the classical primal-dual algorithm (the Hungarian method) and minimax formula (Egerv\'ary's theorem) for the maximum weight bipartite matching problem.
\end{abstract}
\begin{quote}
	{\bf Keywords: }
	Generic partitioned polynomial matrix, Degree of minor, Weighted Edmonds' problem, Weighted non-commutative Edmonds' problem
\end{quote}

\section{Introduction}\label{sec:intro}
The maximum weight bipartite matching problem is one of the most fundamental problems in combinatorial optimization,
which admits a minimax theorem, called Egerv\'ary's theorem~\cite{MFL/E31},
and a primal-dual augmenting path algorithm, called the Hungarian method~\cite{NRLQ/K55}.
In particular,
the Hungarian method outputs, for all possible values of $k$, a matching of size $k$ having maximum weight among all matchings with the same size.
This can be rephrased as:
the Hungarian method computes the entire sequence of the maximum degree of minors of a certain symbolic polynomial matrix.
Indeed,
for a bipartite graph $G = (\set{1,2,\dots, m}, \set{ 1,2,\dots, n }; E)$ with edge weights $d_{ij}$ for $ij \in E$,
we define the matrix $A(t)$
by $(A(t))_{ij} \defeq x_{ij} t^{d_{ij}}$ if $ij \in E$
and zero otherwise,
where $x_{ij}$ is a variable for each edge $ij$ and $t$ is another variable.
Then the maximum weight of a matching of size $k$ in $G$
is equal to the maximum degree $\delta_k(A(t))$ of the minors of order $k$,
i.e.,
\begin{align}\label{eq:delta_k}
    \delta_k(A(t)) \defeq \sup\set*{ \deg \det B(t) }[$B(t)$: $k \times k$ submatrix of $A(t)$],
\end{align}
where the determinant $\det B(t)$ of $B(t)$ is regarded as a polynomial in $t$ and $\delta_0(A(t)) \defeq 0$.
Thus, the entire sequence $\prn{\delta_0(A(t)), \delta_1(A(t)), \dots, \delta_{\min\set{m,n}}(A(t))}$
of the maximum degree of minors
equals the sequence of the maximum weights of a matching of size $k$ for $k = 0,1,\dots,\min \set{m, n}$;
the Hungarian method computes this.

The above algebraic interpretation is generalized to
\emph{weighted Edmonds' problem}\footnote[3]{The original definition of weighted Edmonds' problem is the computation of $\deg \det A(t)$ of a \emph{square} matrix $A(t)$ of the form~\eqref{eq:linear A(t)}. Using the valuated-bimatroid property, the problem in our definition is polynomially equivalent to the original problem; see e.g.,~\cite[Section~5.2.5]{book/Murota00}.}(see~\cite{SIAAG/H19}),
which
asks to compute the entire sequence of the maximum degree of minors of
\begin{align}\label{eq:linear A(t)}
A(t) = A_1(t) x_1 + A_2(t) x_2 + \cdots + A_l(t) x_l.
\end{align}
Here $A_k(t)$ is a polynomial matrix over a field $\setF$ with an indeterminate $t$,
i.e., each entry of $A_k(t)$ is a polynomial in $t$ over $\setF$,
and $x_k$ is a different variable from $t$
for each $k = 1,2,\dots,l$.
This problem is a weighted generalization of a well-studied algebraic problem called \emph{Edmonds' problem}~\cite{JRNBSS/E67}:
It asks 
to compute the rank of
\begin{align}\label{eq:linear A}
A = A_1 x_1 + A_2 x_2 + \cdots + A_l x_l,
\end{align} 
where $A_i$ is a matrix over $\setF$ and $x_i$ is a variable for $i = 1,2,\dots,l$.
(Weighted) Edmonds' problem can capture various matching-type tractable combinatorial optimization problems including not only the maximum (weight) bipartite matching problem but also the maximum (weight) nonbipartite matching, (weighted) linear matroid intersection, 
and (weighted) linear matroid parity problems; see~\cite{JLMS/T47,BSBM/L89,JOA/CCM92}.
Although a randomized polynomial-time algorithm for (weighted) Edmonds' problem
is known (if $\card{\setF}$ is large)~\cite{FCT/L79,JACM/S80},
a deterministic polynomial-time algorithm is not known even for Edmonds' problem,
which is a prominent open problem in theoretical computer science (see e.g.,~\cite{CC/KI04}).

In this paper,
we address the problem of computing the entire sequence 
of the maximum degree of minors (weighted Edmonds' problem)
of the following $(2 \times 2)$-block-structured matrix:
\begin{align}\label{eq:A(t)}
A(t) =
\begin{pmatrix}
A_{11} x_{11} t^{d_{11}} & A_{12} x_{12} t^{d_{12}} & \cdots & A_{1 \nu} x_{1 \nu} t^{d_{1 \nu}}\\
A_{21} x_{21} t^{d_{21}} & A_{22} x_{22} t^{d_{22}} & \cdots & A_{2\nu} x_{2 \nu} t^{d_{2 \nu}}\\
\vdots & \vdots & \ddots & \vdots \\
A_{\mu 1} x_{\mu 1} t^{d_{\mu 1}} & A_{\mu 2} x_{\mu 2} t^{d_{\mu 2}} & \cdots & A_{\mu \nu} x_{\mu \nu} t^{d_{\mu \nu}}
\end{pmatrix},
\end{align}
where $A_{\alpha\beta}$ is a $2 \times 2$ matrix over a field $\setF$,
$x_{\alpha \beta}$ is a variable,
and $d_{\alpha \beta}$ is an integer for $\alpha = 1,2,\dots, \mu$ and $\beta = 1,2,\dots, \nu$,
and $t$ is another variable.
A matrix $A(t)$ of the form~\eqref{eq:A(t)} is called a \emph{$(2 \times 2)$-type generic partitioned polynomial matrix}.

The main result of this paper is as follows,
where we define $\delta_k(A(t))$ as~\eqref{eq:delta_k} for a $(2 \times 2)$-type generic partitioned polynomial matrix $A(t)$.
\begin{theorem}\label{thm:main}
Let $A(t)$ be a $(2 \times 2)$-type generic partitioned polynomial matrix of the form~\eqref{eq:A(t)}.
	There exists a combinatorial $O(\mu\nu \min\set{\mu, \nu}^2)$-time
	algorithm for computing the entire sequence $\prn{\delta_0(A(t)), \delta_1(A(t)), \dots, \delta_{\min\set{2\mu, 2\nu}}(A(t))}$
	of the maximum degree of minors of $A(t)$.
\end{theorem}

Our problem and result are related to the noncommutative analog of weighted Edmonds' problem, called
\emph{weighted noncommutative Edmonds' problem}~\cite{SIAAG/H19}.
In this problem,
given a matrix $A(t)$ of the form~\eqref{eq:linear A(t)},
in which $x_i$ and $x_j$ are supposed to be noncommutative
but $t$ is commutative for any variable $x_i$,
we are asked to compute the entire sequence of the maximum degree of minors,
where the determinant $\deg B(t)$ is replaced with the \emph{Dieudonn\'e determinant}~\cite{BSMF/D43,book/Cohn95} (a determinant concept of a matrix over a skew field) of $B(t)$,
denoted by $\Det B(t)$.
Oki~\cite{ICALP/O20} developed a
pseudopolynomial-time algorithm for this problem.
Hirai~\cite{SIAAG/H19} established a minimax theorem on the degree of the Dieudonn\'e determinant,
and developed another
pseudopolynomial-time algorithm by solving the dual problem.
By combining the above Hirai's algorithm with cost scaling and perturbation techniques,
Hirai and Ikeda~\cite{arxiv/HI20} presented a strongly polynomial-time algorithm for
weighted noncommutative Edmonds' problem for $A(t)$ having the following special form
\begin{align}\label{eq:linear A(t) t^d}
A(t) = A_1 x_1 t^{d_1} + A_2 x_2 t^{d_2} + \cdots + A_l x_l t^{d_l},
\end{align}
where $A_i$ is a square matrix over $\setF$ and $d_i$ is an integer for $i = 1,2,\dots, l$.
A $(2 \times 2)$-type generic partitioned polynomial matrix (with noncommutative variables $x_{\alpha \beta}$)
can be represented as~\eqref{eq:linear A(t) t^d}.

Although the degree of the Dieudonn\'e determinant is an upper bound of that of the determinant,
i.e., $\deg \det A(t) \leq \deg \Det A(t)$ for a matrix $A(t)$ of the form~\eqref{eq:linear A(t)},
and in general the inequality is strict,
Hirai and Ikeda~\cite{arxiv/HI20} also showed that the equality $\deg \det A(t) = \deg \Det A(t)$ holds for a $(2 \times 2)$-type generic partitioned polynomial matrix $A(t)$.
Therefore, the strongly polynomial-time solvability of
our problem
follows from that of weighted noncommutative Edmonds' problem for a matrix of the form~\eqref{eq:linear A(t) t^d} mentioned above.
Hirai--Ikeda's algorithm is conceptually simple but is slow and not combinatorial.
Let $A(t)$ be a $(2 \times 2)$-type generic partitioned polynomial matrix of the form~\eqref{eq:A(t)}.
They present an $O(\min\set{\mu, \nu}^6 \log D)$-time algorithm
for the computation of $\deg \Det A(t)$
via a cost scaling technique,
where $D \defeq \log \max_{\alpha, \beta} |d_{\alpha \beta}|$.
Then,
by utilizing the perturbation technique in~\cite{Comb/FT87} for $d_{\alpha \beta}$
so that $\log D$ is bounded by $O(\mu^3\nu^3)$ in polynomial time,
they devise a strongly polynomial-time algorithm
for computing $\deg \Det A(t)$.
To compute the entire sequence of the maximum degree of minors,
we further need to call the above algorithm $O(\mu \nu\min\set{\mu, \nu})$ times (see e.g.,~\cite[Section~5.2.5]{book/Murota00}).
Moreover,
in case of $\setF = \setQ$,
their algorithm requires an additional procedure (used in~\cite{SICOMP/IK20})
for bounding the bit-complexity.
The minimax theorem on the degree of the Dieudonn\'e determinant provided in~\cite{SIAAG/H19} does \emph{not} provide a good characterization for the deg-det computation even if we restrict to the input as a $(2 \times 2)$-type generic partitioned polynomial matrix $A(t)$ (explicitly described in \cite{arxiv/HI20}).
That is, the formula does \emph{not} imply that the problem of deciding if $\delta_k(A(t)) \geq \theta$ for a given threshold $\theta$
belongs to both NP and co-NP.

In this article,
we establish a new duality theorem on the degree of the determinant of a $(2 \times 2)$-type generic partitioned polynomial matrix $A(t)$,
which is a refinement of the minimax formula provided in~\cite{SIAAG/H19,arxiv/HI20}.
This plays an important role in devising our algorithm.
The proposed theorem consists of the primal concept of \emph{matching-pair}
and the dual concept of \emph{potential}.
The former is a pair of edge subsets of a graph consisting of edges $\alpha \beta$ with nonzero $A_{\alpha \beta}$ in $A(t)$ satisfying some combinatorial and algebraic conditions,
and the latter is a function defined on vector spaces that satisfies some inequalities.
We show that the maximum weight of a matching-pair of size $k$
is equal to the minimum value of a potential with respect to $k$,
and that they coincide with $\delta_k(A(t))$;
this is an algebraic generalization of Egerv\'ary's theorem.
Our minimax formula can be used as a good characterization for the computation of $\delta_k(A(t))$.

The proposed algorithm is a combinatorial primal-dual augmenting path algorithm,
which is an algebraic generalization of the Hungarian method.
An optimal matching-pair of size $k$ and an optimal potential with respect to $k$
enable us to define the auxiliary graph.
If we find an augmenting path on it,
then we can compute $\delta_{k+1}(A(t))$, particularly, we can obtain an optimal matching-pair of size $k+1$ and an optimal potential with respect to $k+1$
by using the augmenting path.
Otherwise, we can verify $\delta_{k+1}(A(t)) = - \infty$ (or equivalently $\rank A(t) = k$).
By repeating the above augmentations,
we finally obtain the entire sequence $\prn{\delta_0(A(t)), \delta_1(A(t)), \dots, \delta_{\min\set{2\mu, 2\nu}}(A(t))}$ of the maximum degree of minors of $A(t)$.
The validity of the algorithm provides a constructive proof of
our minimax theorem.
Our algorithm is simpler and faster than Hirai--Ikeda's algorithm;
ours requires no perturbation of the weight and no additional care for bounding the bit size.

\paragraph{Related work.}
A line of research on the noncommutative setting of an algebraic formulation of combinatorial optimization problems
was
initiated by Ivanyos, Qiao, and Subrahmanyam~\cite{CC/IQS17}
who introduced
\emph{noncommutative Edmonds' problem}:
It asks to compute 
the rank of a matrix of the form~\eqref{eq:linear A},
where 
$x_i$ and $x_j$ are supposed to be noncommutative, 
i.e., $x_ix_j \neq x_jx_i$ for $i \neq j$.
Here the ``rank'' is defined
via the inner rank of a matrix over a free skew field
and is called the \emph{noncommutative rank} or \emph{nc-rank}.
The duality theorem on the nc-rank was established by Fortin and Reutenauer~\cite{SLC/FR04}.
The nc-rank is an upper bound of the rank, i.e., $\rank A \leq \ncrank A$
for a matrix $A$ of the form~\eqref{eq:linear A},
and the inequality is generally strict.
Garg, Gurvits, Oliveira, and Wigderson~\cite{FCT/GGOW20},
Ivanyos, Qiao, and Subrahmanyam~\cite{CC/IQS18},
and Hamada and Hirai~\cite{arxiv/HH17,SIAAG/HH21} independently developed
deterministic polynomial-time algorithms for 
noncommutative Edmonds' problem.
Their algorithms are conceptually different.
Garg, Gurvits, Oliveira, and Wigderson showed that Gurvits' operator scaling algorithm~\cite{JCSS/G04},
which is also known as the flip-flop algorithm in statistics~\cite{JSCS/D99,SPL/LZ05} (see also \cite[Section~4.5]{SIAAG/AKRS21}),
can be used as the nc-rank-computation.
This
works
for the case of $\setF = \setQ$ or $\setC$.
The algorithm of Ivanyos, Qiao, and Subrahmanyam
is an algebraic generalization of an augmenting-path algorithm for the maximum bipartite matching problem,
which works
for an arbitrary field.
Hamada and Hirai reduced the nc-rank-computation to a geodesically-convex optimization on a CAT(0)-space;
the algorithm proposed
in~\cite{arxiv/HH17} works for an arbitrary field $\setF$
provided the arithmetic operations on $\setF$ can be performed in constant time,
while the bit-length may be unbounded if $\setF = \setQ$;
in~\cite{SIAAG/HH21},
the above bit-length issue is resolved.

The block-structured matrix (without an additional indeterminate $t$) was introduced by Ito, Iwata, and Murota~\cite{SIMAA/IIM94} for representing and analyzing a physical system.
In particular, its $(2 \times 2)$-restriction, called a \emph{$(2 \times 2)$-type generic partitioned matrices}, was considered in detail by Iwata and Murota~\cite{SIMAA/IM95}.
They established the minimax theorem on the rank of a $(2 \times 2)$-type generic partitioned matrix,
which is essentially the same as the duality theorem on the nc-rank proposed by Fortin and Reutenauer.
This implies that,
for a $(2 \times 2)$-type generic partitioned matrix, its rank and nc-rank coincide.
Therefore,
we can compute the rank of a $(2 \times 2)$-type generic partitioned matrix in polynomial time
by solving noncommutative Edmonds' problem.
In the previous paper~\cite{MPA/HI21},
Hirai and the author devised a simpler and faster combinatorial algorithm for the rank-computation of a $(2 \times 2)$-type generic partitioned matrix,
which is a combinatorial enhancement of Ivanyos--Qiao--Subrahmanyam's algorithm.
The proposed algorithm in this study is a weighted generalization of this previous algorithm.
We note that, in~\cite{SIMAA/IM95},
Iwata and Murota gave a block-structured matrix consisting only of $2 \times 2$ and $3 \times 2$ blocks
such that its rank and nc-rank are different.
It is known~\cite{SIAAG/H19} that the rank-computation of a general block-structured matrix is equivalent to Edmonds' problem;
its polynomial-time solvability is still open.

The entire sequence of the maximum degree of minors
plays an important role in engineering.
Such a sequence of a rational matrix determines its Smith--McMillan form at infinity,
which is used in control theory~\cite{IEEETAC/VK81},
and that of a matrix pencil determines its Kronecker form,
which is used in analyzing DAEs~\cite{book/KunkelMehrmann06}.
In this literature,
many combinatorial algorithms for computing (the entire sequence of) the maximum degree of minors
has been proposed for rational matrices~\cite{AAECC/M95,SISC/IMS96,Algo/S17},
for matrix pencils~\cite{Algo/I03}, and mixed polynomial matrices~\cite{Algo/IT13,JSIAMLett/S15};
see also \cite[Chapters 5 and 6]{book/Murota00}.

\paragraph{Organization.}
The remainder of this paper is organized as follows.
In \cref{sec:matching and potential},
we introduce the primal concepts called \emph{pseudo-matching} and \emph{matching-pair}
and the dual concept called \emph{potential}.
Then
we provide a minimax theorem between the weight of a matching-pair and a potential,
which leads to a good characterization for the computation of the maximum degree of minors of $A(t)$.
In \cref{sec:augmenting path},
we introduce an augmenting path for a matching-pair and a potential,
and develop an algorithm for finding an augmenting path for the current matching-pair and potential.
The rest of sections
(Sections~\ref{sec:preliminaries augmentation}--\ref{sec:non violate})
are devoted to devising an augmenting algorithm.

\paragraph{Notations.}
For a positive integer $k$,
we denote $\set{1,2,\dots, k}$ by $\intset{k}$.
Let $A(t)$ be
a $(2 \times 2)$-type generic partitioned polynomial matrix of the form~\eqref{eq:A(t)}.
The matrix $A(t)$ is regarded as
a matrix over the field $\setF(x, t)$ of rational functions with variables $t$ and $x_{\alpha\beta}$ for $\alpha \in \intset{\mu}$
and $\beta \in \intset{\nu}$.
The symbols $\alpha$, $\beta$, and $\gamma$ 
are used to represent a row-block index in $\intset{\mu}$, column-block index in $\intset{\nu}$, and row- or column-block index in $\intset{\mu} \sqcup \intset{\nu}$ of $A(t)$, respectively,
where $\sqcup$ denotes the direct sum.
We often drop ``$\in \intset{\mu}$" from the notation of ``$\alpha \in \intset{\mu}$" if it is clear from the context.
Each $\alpha$ and $\beta$
is endowed with the 2-dimensional $\setF$-vector space $\setF^2$, denoted by $U_{\alpha}$ and $V_{\beta}$,
respectively. 
Each submatrix $A_{\alpha \beta}$ is considered 
as the bilinear map $\doms{U_{\alpha} \times V_{\beta}}{\setF}$
defined by
$A_{\alpha \beta}(u, v) \defeq u^\top A_{\alpha \beta}v$
for $u \in U_\alpha$ and $v \in V_\beta$.
We denote by $\kerL(A_{\alpha \beta})$ and $\kerR(A_{\alpha \beta})$
the left and right kernels of $A_{\alpha \beta}$,
respectively.
Let us denote by $\calM_\alpha$ and $\calM_\beta$
the sets of 1-dimensional vector subspaces of $U_\alpha$ and $V_\beta$,
respectively.

We define the (undirected) bipartite graph $G \defeq (\intset{\mu}, \intset{\nu}; E)$
by $E \defeq \set{\alpha \beta}[A_{\alpha\beta} \neq O]$.
For $M \subseteq E$,
let $A_M(t)$ denote
the matrix obtained from $A(t)$ by replacing  
each submatrix $A_{\alpha\beta}$ with $\alpha \beta \not\in M$ by the $2 \times 2$ zero matrix.
An edge $\alpha \beta \in E$ is said to be \emph{rank-$k$} $(k=1,2)$
if $\rank A_{\alpha\beta} = k$.
For notational simplicity, 
the subgraph $(\intset{\mu}, \intset{\nu}; M)$ for $M \subseteq E$
is also denoted by $M$. 
For a node $\gamma$, let $\deg_M(\gamma)$ denote the degree of $\gamma$ in $M$, i.e.,
the number of edges in $M$ incident to $\gamma$. 
An edge $\alpha \beta \in M$ is said to be \emph{isolated}
if $\deg_M(\alpha) = \deg_M(\beta) = 1$.

\section{Duality theorem}\label{sec:matching and potential}
In this section,
we introduce a matching concept and a potential concept suitable
for a $(2 \times 2)$-type generic partitioned polynomial matrix $A(t)$ of the form~\eqref{eq:A(t)}.
They play a central role in devising our algorithm.
We also present a minimax theorem
between the weight of a ``matching'' and the value related to a ``potential'' in our setting,
which leads to a good characterization for the computation of the maximum degree of minors of $A(t)$.

\subsection{Matching concept}

We introduce a matching concept named {\it pseudo-matching}.
This is a weaker concept than {\it matching of a $(2 \times 2)$-type generic partitioned (not polynomial) matrix}
that
introduced in the previous work~\cite{MPA/HI21},
because of which, it is prefixed with
``pseudo.''
An edge subset $M \subseteq E$ is called a \emph{pseudo-matching}
if it satisfies
the following
combinatorial and algebraic conditions (Deg), (Cycle), and (VL):
\begin{description}
	\item[{\rm (Deg)}] $\deg_M(\gamma) \leq 2$ for each node $\gamma$ of $G$.
\end{description}
Suppose that $M$ satisfies (Deg). Then each connected component of $M$ forms a path or a cycle.
Thus $M$ is 2-edge-colorable; i.e.,
there are two edge classes such that any two incident edges are in different classes.
An edge in one color class is called a {\it $+$-edge},
and an edge in the other color class is called a {\it $-$-edge}.
\begin{description}
	\item[{\rm (Cycle)}]
	Each cycle component of $M$
	has at least one rank-1 edge.
\end{description}
A {\em labeling} $\calV = \labeling$ is a node-labeling that assigns
two distinct 1-dimensional subspaces to each node,
$U_\alpha^+, U_\alpha^- \in \calM_{\alpha}$ with $U_\alpha^+ \neq U_\alpha^-$
for $\alpha$
and $V_\beta^+, V_\beta^- \in \calM_{\beta}$ with $V_\beta^+ \neq V_\beta^-$ for $\beta$.
A labeling $\calV$
is said to be {\em valid} for $M$ if,
for each edge $\alpha \beta \in M$,
\begin{align}
&A_{\alpha \beta}(U_{\alpha}^+, V_{\beta}^-) = A_{\alpha
	\beta}(U_{\alpha}^-, V_{\beta}^+) = \set{0},\label{eq:+-}\\
&(\kerL(A_{\alpha \beta}), \kerR(A_{\alpha \beta})) =
\begin{cases}
(U_{\alpha}^+, V_{\beta}^+) & \text{if $\alpha \beta$ is a rank-1 $+$-edge},\\
(U_{\alpha}^-, V_{\beta}^-) & \text{if $\alpha \beta$ is a rank-1 $-$-edge}.
\end{cases}\label{eq:++ --}
\end{align}
For $\alpha$,
we refer to $U_\alpha^+$ and $U_\alpha^-$ as the {\it $+$-space} and {\it $-$-space} of $\alpha$ with respect to $\calV$,
respectively.
The same terminology is also used for $\beta$.
\begin{description}
	\item[{\rm (VL)}]
	$M$ admits a valid labeling.
\end{description}

In the following sections,
we use the symbol $\sigma$ as one of the signs $+$ and $-$.
The opposite sign of $\sigma$ is denoted by $\overline{\sigma}$,
i.e., $\overline{\sigma} = -$ if $\sigma = +$,
and $\overline{\sigma} = +$ if $\sigma = -$.

\begin{remark}\label{rmk:relabeling}
	Suppose that $M$ satisfies (Deg)
	and that $\alpha \beta$ is a rank-1 $\sigma$-edge in $M$.
	The condition~\eqref{eq:++ --} determines
	$U_{\alpha}^\sigma$ and $V_{\beta}^\sigma$,
	and the condition~\eqref{eq:+-}
	determines
	$V_{\beta'}^{\overline{\sigma}}$ and $U_{\alpha'}^\sigma$
	(resp. $U_{\alpha'}^{\overline{\sigma}}$ and $V_{\beta'}^\sigma$)
	for $\alpha'$ and $\beta'$ belonging to the path in $M$ which starts with $\alpha$ (resp. $\beta$) and consists of rank-2 edges.

	Suppose further that $M$ satisfies (Cycle).
	For each node in some cycle component of $M$,
	its $+$-space and $-$-space are uniquely determined
	by the above argument,
	since every cycle component has a rank-1 edge by (Cycle).
	Let $C$ be a path component of $M$,
	which has the end nodes $\gamma$ and $\gamma'$ incident to a $\sigma$-edge and a $\sigma'$-edge, respectively.
	When we set the $\overline{\sigma}$-space of $\gamma$
	and $\overline{\sigma'}$-space of $\gamma'$,
	the $+$-space and $-$-space of every node belonging to $C$
	are uniquely determined.
	\qqed
\end{remark}
By the argument in \cref{rmk:relabeling},
we can check if an edge subset $M$ is a pseudo-matching in polynomial time.

Let $M \subseteq E$ be a pseudo-matching, and $I$ a set of isolated rank-2 edges in $M$.
We refer to such a pair $(M, I)$ as a \emph{matching-pair}.
The \emph{size} of a matching-pair $(M, I)$ is $|M| + |I|$.
The \emph{weight} $w(M, I)$ of $(M, I)$
is defined
by
\begin{align*}
    w(M, I) \defeq \sum_{\alpha \beta \in M} d_{\alpha \beta} + \sum_{\alpha \beta \in I} d_{\alpha \beta}.
\end{align*}
Let $\labeling$ be a valid labeling for $M$.
We say that $U_\alpha^\sigma$ (resp. $V_\beta^\sigma$) is \emph{matched by $(M, I)$}
if
$\alpha$ (resp. $\beta$) is incident to a $\overline{\sigma}$-edge in $M$
or to a $\sigma$-edge in $I$.
That is, the set of all spaces matched by $(M, I)$ is representable as
\begin{align}\label{eq:covered}
\bigcup_{\sigma \in \set{+, -}}
    \set*{ U_\alpha^{\overline{\sigma}}, V_\beta^{\overline{\sigma}} }[$\alpha \beta \in M$: $\sigma$-edge] \cup \bigcup_{\sigma \in \set{+, -}}
    \set*{ U_\alpha^{\sigma}, V_\beta^{\sigma}}[$\alpha \beta \in I$: $\sigma$-edge].
\end{align}
Thus the number of $U_\alpha^\sigma$ that are matched by $(M, I)$ coincides with that of $V_\beta^\sigma$,
which are equal to the size $|M| + |I|$ of $(M, I)$.

\subsection{Minimax formula}
In this subsection,
we provide a minimax formula
between the maximum weight of a matching-pair of size $k$
and the minimum value corresponding to a potential
(defined below)
and $k$,
which coincides with $\delta_k(A(t))$.
This formula is an algebraic generalization of Egerv\'ary's theorem~\cite{MFL/E31} that is a minimax theorem for the maximum weight perfect bipartite matching problem.

For $c \in \setR$,
a function $\funcdoms{p}{\bigcup_\gamma \calM_\gamma}{\setR}$
is called a {\it $c$-potential}
if
\begin{itemize}
    \item $p$ is nonnegative, i.e., $p(Z) \geq 0$ for all $Z \in \bigcup_\gamma \calM_\gamma$, and
    \item $p(X) + p(Y) + c \geq d_{\alpha \beta}$
for all $\alpha \beta \in E$, $X \in \calM_\alpha$, and $Y \in \calM_\beta$
such that $A_{\alpha \beta}(X, Y) \neq \set{0}$.
\end{itemize}
We can omit the parameter $c$ from the notation if it is not important in the context.
For a potential $p$ and a labeling $\calV = (\{ U_\alpha^+, U_\alpha^- \}, \{ V_\beta^+, V_\beta^- \})_{\alpha, \beta}$,
we define
\begin{align*}
    p(\calV) \defeq \sum_\alpha \left(p(U_\alpha^+) + p(U_\alpha^-)\right) + \sum_\beta\left(p(V_\beta^+) + p(V_\beta^-)\right).
\end{align*}

The following minimax formula is a generalization of Egerv\'ary's theorem:
\begin{theorem}\label{thm:minmax}
    Let $k$ be a nonnegative integer.
    The following values {\rm (i)}--{\rm (iii)} are the same:
    \begin{description}
        \item[{\rm (i)}] $\delta_k(A(t))$.
        \item[{\rm (ii)}] $\sup \set*{ w(M, I)}[$(M, I)$: matching-pair of size $k$]$.
        \item[{\rm (iii)}] $\inf \set*{ p(\calV) + kc }[$\calV$: labeling,\ $c \in \setR$,\ $p$: $c$-potential]$.
    \end{description}
\end{theorem}
\begin{proof}
    We only show the weak duality $\text{(ii)} \leq \text{(i)} \leq \text{(iii)}$.
    The strong duality $\text{(ii)} = \text{(iii)}$ follows from the validity of our proposed algorithm.
    
    In the proof, we perform the following basis transformation with respect to a labeling $\calV = \labeling$.
	Take nonzero vectors $u_\alpha^+ \in U_\alpha^+$, $u_\alpha^- \in U_\alpha^-$, $v_\beta^+ \in V_\beta^+$, and $v_\beta^- \in V_\beta^-$ for each $\alpha$ and $\beta$.
	By $U_\alpha^+ \neq U_\alpha^-$ and $V_\beta^+ \neq V_\beta^-$,
	the $2 \times 2$ matrices
	$S_\alpha \defeq
	\left[
	\begin{array}{c}
	u_\alpha^+\\
	u_\alpha^-
	\end{array}
	\right]$
	and
	$T_\beta \defeq
	\begin{bmatrix}
	v_\beta^+ &
	v_\beta^-
	\end{bmatrix}$
	are both nonsingular.
	Let $S$ and $T$ be the block-diagonal matrices with diagonal blocks $S_{\alpha}$
	and $T_{\beta}$,
	respectively.
	Then, via the basis transformation with respect to $S$ and $T$,
	we obtain a $(2 \times 2)$-type generic partitioned polynomial matrix $S A(t) T = 
	(S_\alpha A_{\alpha \beta} x_{\alpha \beta} t^{d_{\alpha \beta}} T_\beta)$,
	in which the $\alpha^{\sigma} \beta^{\sigma'}$-th entry of $S A(t) T$ is $A_{\alpha \beta}(u_\alpha^\sigma, v_\beta^{\sigma'}) x_{\alpha \beta} t^{d_{\alpha \beta}}$.
	Note that the $\alpha^{\sigma} \beta^{\sigma'}$-th entry of $S A(t) T$ is of the form $a x_{\alpha \beta} t^{d_{\alpha \beta}}$ with some $a \in \setF$,
	and it is nonzero if and only if $A_{\alpha \beta}(U_\alpha^{\sigma}, V_\beta^{\sigma'}) \neq \set{0}$.

    $\text{(ii)} \leq \text{(i)}$.
    Take any matching-pair $(M, I)$ of size $k$
    and valid labeling $\labeling$ for $M$.
    We consider the basis transformation with respect to the valid labeling.
	By conditions~\eqref{eq:+-} and~\eqref{eq:++ --},
	we have
	\begin{numcases}
	{S_\alpha (A_{\alpha \beta} x_{\alpha \beta} t^{d_{\alpha \beta}}) T_\beta =}
	\kbordermatrix{
		& \beta^+ & \beta^- \\
		\alpha^+ & \bullet & 0 \\
		\alpha^- & 0 & \bullet
	}
	& if $\alpha\beta$ is rank-2,\label{eq:rank 2}\\
	\kbordermatrix{
		& \beta^+ & \beta^- \\
		\alpha^+ & 0 & 0 \\
		\alpha^- & 0 & \bullet
	}
	& if $\alpha\beta$ is a rank-1 $+$-edge\label{eq:rank 1 +},\\
	\kbordermatrix{
		& \beta^+ & \beta^- \\
		\alpha^+ & \bullet & 0 \\
		\alpha^- & 0 & 0
	}
	& if $\alpha\beta$ is a rank-1 $-$-edge\label{eq:rank 1 -}
	\end{numcases}
	for each $\alpha\beta \in M$,
	where $\bullet$ represents some nonzero element in $\setF(x, t)$.
	
	Define $\tilde{A}_M(t) \defeq SA_M(t)T$.
	Note that $\delta_k(\tilde{A}_M(t)) = \delta_k(A_M(t))$.
	Moreover, by $\delta_k(A_M(t)) \leq \delta_k(A(t))$,
    it suffices to show that $w(M, I) \leq \delta_k(\tilde{A}_M(t))$.
	Let $\mathcal{X}$ (resp. $\mathcal{Y}$) denote the set of $\alpha^\sigma$ (resp. $\beta^{\sigma}$) such that $U_\alpha^\sigma$ (resp. $V_\beta^\sigma$) is matched by $(M, I)$.
	By $|M| + |I| = k$, we have $|\mathcal{X}| = |\mathcal{Y}| = k$.
	Let $\tilde{A}_M(t)[\mathcal{X}, \mathcal{Y}]$ denote the submatrix of $\tilde{A}_M(t)$ with row set $\mathcal{X}$ and column set $\mathcal{Y}$.
	Furthermore, let $\calC$ be the set of connected components of $M \setminus I$.
	For each $C \in \calC$,
	we denote $\mathcal{X}_C$ and $\mathcal{Y}_C$ by the restrictions of $\mathcal{X}$ and $\mathcal{Y}$
	to $C$, respectively.
	Then we have
	\begin{align}\label{eq:det tilde AM}
	    \det \tilde{A}_M(t)\sqbr{\mathcal{X}, \mathcal{Y}} = \prod_{\alpha \beta \in I} \det \tilde{A}_M(t)\sqbr{\set{\alpha^+, \alpha^-}, \set{\beta^+, \beta^-}} \cdot \prod_{C \in \calC} \det \tilde{A}_M(t)\sqbr{\mathcal{X}_C, \mathcal{Y}_C}.
	\end{align}
	
	In the following, we prove that
    $\deg \det\tilde{A}_M(t)\sqbr{\set{\alpha^+, \alpha^-}, \set{\beta^+, \beta^-}} = 2 d_{\alpha \beta}$ for $\alpha \beta \in I$, and that $\deg \det \tilde{A}_M(t)[\mathcal{X}_C, \mathcal{Y}_C] \geq \sum_{\alpha \beta \in C} d_{\alpha \beta}$ for $C \in \calC$;
    these imply $w(M, I) \leq \deg \det \tilde{A}_M(t)[\mathcal{X}, \mathcal{Y}]$ by~\eqref{eq:det tilde AM},
	and hence, we obtain $w(M, I) \leq \delta_k(\tilde{A}_M(t))$, as required.
	The former immediately follows from \eqref{eq:rank 2} and the fact that $\alpha \beta \in I$ is rank-2.
	For the latter, we only consider the case where $C \in \calC$ is a cycle component;
    the argument for a path component is simpler and we omit it.
    Suppose that $C$ consists of $+$-edges $\alpha_1 \beta_1, \alpha_2 \beta_2, \dots, \alpha_k \beta_k$
	and $-$-edges $\beta_1 \alpha_2, \beta_2 \alpha_3, \dots, \beta_{k} \alpha_1$.
	By~\eqref{eq:rank 2}--\eqref{eq:rank 1 -},
	we obtain
	\begin{align}\label{eq:EAF}
	\tilde{A}_M[\mathcal{X}_C, \mathcal{Y}_C] =
	\begin{blockarray}{ccccccccccc}
	& \beta_1^+ & \beta_2^+ & \cdots & \beta_{k-1}^+ & \beta_k^+  & \beta_1^- & \beta_2^- & \cdots & \beta_{k-1}^- & \beta_k^- \\
	\begin{block}{c[ccccc|ccccc]}
	\alpha_1^+ & \ast & & & & \bullet & & & & & \\
	\alpha_2^+ & \bullet & \ast & & & & & & & &\\
	\vdots &  & \ddots & \ddots & & & & & & &\\
	\alpha_{k-1}^+ & & & \bullet & \ast & & & & & &\\
	\alpha_k^+ & & & & \bullet & \ast & & & & &\\
	\cline{2-11}
	\alpha_1^- & & & & & & \bullet & & & & \ast \\
	\alpha_2^- & & & & & & \ast & \bullet & & &\\
	\vdots & & & & & & & \ddots & \ddots & &\\
	\alpha_{k-1}^- & & & & & & & & \ast & \bullet & \\
	\alpha_k^- & & & & & & & & & \ast & \bullet \\
	\end{block}
	\end{blockarray},
	\end{align}
	where $\bullet$ represents some nonzero element in $\setF(x, t)$
	and $\ast$ can be a zero/nonzero element;
	$\ast$ is nonzero if and only if the corresponding edge is rank-2.
	Let $C^+$ be the set of $+$-edges $\alpha_1 \beta_1, \alpha_2 \beta_2, \dots, \alpha_k \beta_k$ in $C$,
	and $C^-$ the set of $-$-edges $\beta_1 \alpha_2, \beta_2 \alpha_3, \dots, \beta_k \alpha_1$ in $C$.
	By (Cycle),
	one of $\ast$ elements in~\eqref{eq:EAF} is zero;
	we may assume that the $\alpha_1^+ \beta_1^+$-th entry of $\tilde{A}_M(t)[X_C, Y_C]$ is zero.
	Hence we obtain that
	\begin{align*}
	\det \tilde{A}_M(t)[\mathcal{X}_C, \mathcal{Y}_C] = \prn{a' \prod_{\alpha \beta \in C^-} x_{\alpha \beta} t^{d_{\alpha \beta}}} \cdot \prn{ a'' \prod_{\alpha \beta \in C^+} x_{\alpha \beta} t^{d_{\alpha \beta}} + a''' \prod_{\alpha \beta \in C^-} x_{\alpha \beta} t^{d_{\alpha \beta}} }
	\end{align*}
	for some $a''' \in \setF$ and nonzero $a', a'' \in \setF$.
	Thus $\deg \det \tilde{A}_M(t)[\mathcal{X}_C, \mathcal{Y}_C] \geq \sum_{\alpha \beta \in C^-} d_{\alpha \beta} + \sum_{\alpha \beta \in C^+} d_{\alpha \beta} = \sum_{\alpha \beta \in C} d_{\alpha \beta}$.
	
	This completes the proof.

	$\text{(i)} \leq \text{(iii)}$.
	It suffices to see the case of $\delta_k(A(t)) > -\infty$, i.e., $\rank A(t) \geq k$.
    Take any $t \in \setR$, $c$-potential $p$, and (not necessarily valid) labeling $\calV = \labeling$.
	We consider the basis transformation $\tilde{A}(t) \defeq S A(t) T$ with respect to $\calV$.
	Then $\delta_k(\tilde{A}(t)) = \delta_k(A(t))$ holds.
	
	Suppose $\delta_k(\tilde{A}(t)) = \deg \det \tilde{A}(t)[\mathcal{X}, \mathcal{Y}]$ for some $|\mathcal{X}| = |\mathcal{Y}| = k$,
	where $\mathcal{X} = \set{ \alpha_1^{\sigma_1}, \alpha_2^{\sigma_2}, \dots, \alpha_k^{\sigma_k} }$
	and $\mathcal{Y} = \set{ \beta_1^{\sigma_1'}, \beta_2^{\sigma_2'}, \dots, \beta_k^{\sigma_k'} }$.
	We may assume that the $\alpha_i^{\sigma_i} \beta_i^{\sigma_i'}$-th entry of $\tilde{A}(t)[\mathcal{X}, \mathcal{Y}]$ is nonzero (or equivalently, $A_{\alpha_i \beta_i}(U_{\alpha_i}^{\sigma_i}, V_{\beta_i}^{\sigma_i'}) \neq \set{0}$) for each $i = 1,2,\dots,k$
	and that $\deg \det \tilde{A}(t)[\mathcal{X}, \mathcal{Y}] = \sum_{i = 1}^k d_{\alpha_i \beta_i}$.
	Then
	we have $\delta_k(\tilde{A}(t)) = \sum_{i = 1}^k d_{\alpha_i \beta_i}
	\leq \sum_{i = 1}^k \prn{ p(U_{\alpha_i}^{\sigma_i}) + p(V_{\beta_i}^{\sigma_i'}) + c } \leq p(\calV) + kc$.
	Here the first
	and second inequalities follow from the fact that $p$ is a $c$-potential.
	
	This completes the proof.
\end{proof}

\subsection{Good characterization}
The minimax formula (\cref{thm:minmax}) states that
the decision problem of whether $\delta_k(A(t))$ is at least a threshold $\theta \in \setR$
belongs to NP.
Indeed,
a matching-pair $(M, I)$ of size $k$ such that $w(M, I) \geq \theta$
can be used as a proof for $\delta_k(A(t)) \geq \theta$,
which is verifiable in polynomial time.
In the following,
by introducing the concept of \emph{compatibility} for a potential,
we see that the problem of whether $\delta_k(A(t)) \geq \theta$
is also in co-NP by using \cref{thm:minmax}.
This implies that the minimax theorem can be used as a good characterization for the computation of $\delta_k(A(t))$.

Let $(M, I)$ be a matching-pair of size $k$
and $\calV$ a valid labeling for $M$.
A $c$-potential $p$ is said to be \emph{$(M, I, \calV)$-compatible}
if $p$ satisfies the following conditions (Reg) and (Tight):
\begin{description}
    \item[{\rm (Reg)}]
	For each $\alpha$ and $\beta$,
	\begin{align*}
	p(X) &= \max \{ p(U_\alpha^+), p(U_\alpha^-) \} \qquad (X \in \calM_\alpha \setminus \{ U_\alpha^+, U_\alpha^- \}),\\
	p(Y) &= \max \{ p(V_\beta^+), p(V_\beta^-) \} \qquad (Y \in \calM_\beta \setminus \{ V_\beta^+, V_\beta^- \}).
	\end{align*}
	\item[{\rm (Tight)}]
	For each $\alpha \beta \in M$,
	\begin{align*}
	d_{\alpha \beta} =
	\begin{cases}
	p(U_\alpha^-) + p(V_\beta^-) + c & \text{if $\alpha \beta$ is a $+$-edge},\\
	p(U_\alpha^+) + p(V_\beta^+) + c & \text{if $\alpha \beta$ is a $-$-edge},
	\end{cases}
	\end{align*}
	and for each $\alpha \beta \in I$,
	\begin{align*}
	d_{\alpha \beta} =
	\begin{cases}
	p(U_\alpha^+) + p(V_\beta^+) + c & \text{if $\alpha \beta$ is a $+$-edge},\\
	p(U_\alpha^-) + p(V_\beta^-) + c & \text{if $\alpha \beta$ is a $-$-edge},
	\end{cases}
	\end{align*}
\end{description}
An $(M, I, \calV)$-compatible $c$-potential $p$ is said to be \emph{optimal}
if the equality $w(M, I) = p(\calV) + kc$ holds,
namely, $(M, I)$ and $(p, \calV)$ attain
the supremum of (ii) and the infimum of (iii) in \cref{thm:minmax},
respectively.
The following theorem (\cref{thm:good characterization}) states that
such an optimal potential always exists if $\delta_{k}(A(t))$ is bounded;
its proof is given by the validity of our algorithm.
\begin{theorem}\label{thm:good characterization}
    Let $k$ be a nonnegative integer.
    If $\delta_k(A(t))$ is bounded,
    then there are a matching-pair $(M, I)$ of size $k$, a valid labeling $\calV$ for $M$,
    and an optimal $(M, I, \calV)$-compatible $c$-potential $p$ for some $c \in \setR$.
    In particular, the above $p$ and $c$ can be chosen to be integer-valued.
\end{theorem}

By \cref{thm:good characterization},
a pair $(p, \calV)$ of
a $c$-potential $p$ satisfying (Reg) and a valid labeling $\calV$ satisfying $p(\calV) + kc < \theta$
can be used as a proof for $\delta_k(A(t)) < \theta$.
The following shows that
the condition (Reg) enables us to check if
a given nonnegative function $p$ on $\bigcup_\gamma \calM_\gamma$ is a $c$-potential
in polynomial time.
\begin{lemma}\label{lem:(Reg)}
	Suppose that a nonnegative function $p$ on $\bigcup_\gamma \calM_\gamma$ satisfies {\rm (Reg)} for a labeling $\labeling$.
	If $p(U_{\alpha}^\sigma) + p(V_{\beta}^{\sigma'}) + c \geq d_{\alpha \beta}$ for all $\alpha \beta \in E$ with $A_{\alpha \beta}(U_{\alpha}^\sigma, V_{\beta}^{\sigma'}) \neq \set{0}$,
	then $p$ is a $c$-potential.
\end{lemma}
\begin{proof}
	Take arbitrary $X \in \calM_\alpha$ and $Y \in \calM_\beta$
	with $p(X) + p(Y) + c < d_{\alpha\beta}$.
	It suffices to show that $A_{\alpha \beta}(X, Y) = \set{0}$.
	Here
	we may assume $X \notin \{ U_\alpha^+, U_\alpha^- \}$;
	the case of $Y \notin \{ V_\beta^+, V_\beta^- \}$ is similar and we omit it.
	We also assume $p(U_\alpha^+) \leq p(U_\alpha^-)$ and $p(V_\beta^+) \leq p(V_\beta^-)$.
	
	By (Reg),
	we obtain $p(X) = p(U_\alpha^-) \geq p(U_\alpha^+)$
	and $p(Y) \geq p(V_\beta^+)$,
	implying
	$p(U_\alpha^+) + p(V_\beta^+) + c \leq p(U_\alpha^-) + p(V_\beta^+) + c \leq p(X) + p(Y) + c < d_{\alpha \beta}$.
	Therefore $A_{\alpha \beta}(U_\alpha^+, V_\beta^+) = A_{\alpha \beta}(U_\alpha^-, V_\beta^+) = \set{0}$,
	i.e., $\kerR(A_{\alpha\beta}) \supseteq V_\beta^+$.
	If $Y = V_\beta^+$,
	then we have $A_{\alpha \beta}(X, Y) = \set{0}$.
	If $Y \neq V_\beta^+$,
	then $p(Y) = p(V_\beta^-)$ holds by (Reg),
	which implies
	the inequalities $p(U_\alpha^+) + p(V_\beta^-) + c \leq p(U_\alpha^-) + p(V_\beta^-) + c < d_{\alpha \beta}$.
	Hence we have $A_{\alpha \beta}(U_\alpha^+, V_\beta^-) = A_{\alpha \beta}(U_\alpha^-, V_\beta^-) = \set{0}$;
	$A_{\alpha \beta}$ is the zero matrix.
	Thus we obtain $A_{\alpha \beta}(X, Y) = \set{0}$.
\end{proof}
It is known~\cite[Section~5]{MPA/HI21} that
the bit-length required for representing a valid labeling is polynomially bounded
even if $\setF = \setQ$.
Thus
the proof $(p, \calV)$ for $\delta_k(A(t)) < \theta$
is verifiable in polynomial time,
implying
the problem of whether $\delta_k(A(t)) \geq \theta$
is in co-NP.

We conclude this section with the observation that
the optimality of an $(M, I, \calV)$-compatible potential $p$
can be rephrased as the condition (Zero):
\begin{description}
    \item[{\rm (Zero)}]
	For all $U_\alpha^\sigma$ and $V_\beta^{\sigma'}$ that are unmatched by $(M, I)$,
	\begin{align*}
	    p(U_\alpha^\sigma) = p(V_\beta^{\sigma'}) = 0.
	\end{align*}
\end{description}
The definitions of (Tight) and (Zero) immediately imply the following.
\begin{lemma}\label{lem:compatible}
    Let $(M, I)$ be a matching-pair of size $k$, $\calV$ a valid labeling for $M$, and $p$ an $(M, I, \calV)$-compatible $c$-potential.
    Then we have $w(M, I) = \sum \set*{ p(Z) }[$Z$: matched by $(M, I)$] + kc$.
    In particular, $p$ is optimal if and only if $p$ satisfies {\rm (Zero)}.
\end{lemma}

\section{Augmenting path}\label{sec:augmenting path}
Our proposed algorithm is a primal-dual one.
An outline of the algorithm is as follows;
the formal description is given in \cref{subsec:finding}.
Let $(M, I)$ be a matching-pair of size $k$,
$\calV$ a valid labeling for $M$,
and $p$ an optimal $(M, I, \calV)$-compatible $c$-potential.
We
\begin{itemize}
	\item
	verify $\delta_{k+1}(A(t))= - \infty$ (or equivalently $\rank A(t) = k$),
	\item
	find an optimal compatible potential so that a \emph{rearrangeable} component (introduced in \cref{subsec:rearrangement}) exists in $M \setminus I$, or
	\item
	find an \emph{augmenting path} (introduced in \cref{subsec:definition}).
\end{itemize}
In the first case, we output the entire sequence of the maximum degree of minors as
\begin{align*}
    (\delta_0(A(t)), \delta_1(A(t)), \dots, \delta_k(A(t)), -\infty, \dots, -\infty)
\end{align*}
and stop this procedure.
The others are cases where
we obtain a matching-pair $(M^*, I^*)$ of size $k+1$,
a valid labeling $\calV^*$ for $M^*$,
and an optimal $(M^*, I^*, \calV^*)$-compatible potential $p^*$,
which implies $\delta_{k+1}(A(t)) = w(M^*, I^*)$.
This is an augmentation in our setting.

Our augmentation is based on the \emph{auxiliary graph} $\calG(\calV, p)$:
The vertex set is $\set{ \alpha^+, \alpha^-}[\alpha] \cup \set{ \beta^+, \beta^- }[\beta]$,
and the edge set, denoted by $\calE(\calV, p)$, is $\set{ \alpha^{\sigma} \beta^{\sigma'} }[A_{\alpha \beta}(U_\alpha^{\sigma}, V_\beta^{\sigma'}) \neq \set{0},\ p(U_\alpha^{\sigma}) + p(V_\beta^{\sigma'}) + c = d_{\alpha\beta} ]$.
By~\eqref{eq:+-},
for each $\alpha \beta \in M$,
neither $\alpha^+ \beta^-$ nor $\alpha^- \beta^+$
belongs to $\calE(\calV, p)$.
The condition (Tight) implies that
for each $\sigma$-edge $\alpha \beta \in M$,
we have $\alpha^{\overline{\sigma}} \beta^{\overline{\sigma}} \in \calE(\calV, p)$.
In addition, if $\alpha \beta \in I$, then $\alpha^+ \beta^+, \alpha^- \beta^- \in \calE(\calV, p)$.
An edge $\alpha \beta \in M$ is said to be \emph{double-tight}
if $\alpha^+ \beta^+, \alpha^- \beta^- \in \calE(\calV, p)$,
i.e., $\alpha \beta$ is rank-2 and $p(U_\alpha^+) + p(V_\beta^+) + c = p(U_\alpha^-) + p(V_\beta^-) + c = d_{\alpha \beta}$.
A $\sigma$-edge $\alpha \beta \in M$ is said to be \emph{single-tight}
if it is not double-tight,
i.e.,
$\alpha^{\overline{\sigma}} \beta^{\overline{\sigma}} \in \calE(\calV, p)$ and $\alpha^\sigma \beta^\sigma \notin \calE(\calV, p)$.
Note that all edges in $I$ are double-tight
and all rank-1 edges in $M$ are single-tight.
We refer to $\alpha^\sigma$ and $\beta^\sigma$ as {\it $\sigma$-vertices}.
We denote by $\calG(\calV, p)|_M$
the subgraph of $\calG(\calV, p)$ such that its edge set $\calE(\calV, p)|_M$
is $\set{ \alpha^{\sigma} \beta^{\sigma'} \in \calE(\calV, p) }[\alpha \beta \in M ]$.
A \emph{$\sigma$-path}
is a path in $\calG(\calV, p)$
consisting of edges $\alpha^\sigma \beta^\sigma$.

As the initialization ($k = 0$),
we set both $M$ and $I$ as the empty sets,
and $\calV$
as any labeling.
We define $p$ and $c$ by
$p(Z) \defeq 0$ for every $Z \in \bigcup_\gamma \calM_\gamma$
and $c \defeq \max \set{ d_{\alpha \beta} }[ \alpha \beta \in E ]$,
respectively.

\subsection{Rearrangement}\label{subsec:rearrangement}
Let $C$ be a path component of $M \setminus I$ with odd length such that the end edges of $C$ are $\sigma$-edges.
We say that $C$ is \emph{rearrangeable with respect to $p$}
if every $\sigma$-edge in $C$
is double-tight.

Suppose that $C$ is rearrangeable with respect to $p$.
The \emph{rearrangement} of $(M, I)$ with respect to $C$
is an operation
of modifying $(M, I)$ to a matching-pair $(M^*, I^*)$ of size $k+1$ as follows:
\begin{align*}
    M^* &\defeq M \setminus \set{\text{all $\overline{\sigma}$-edges in $C$}},\\
    I^* &\defeq I \cup \set{\text{all $\sigma$-edges in $C$}}.
\end{align*}
Since $C$ has odd length,
we have $|M^*| = |M| - (|C|-1)/2$ and $|I^*| = |I| + (|C|+1)/2$,
implying that the size of $(M^*, I^*)$ is larger than that of $(M, I)$ by one.
Clearly $\calV$ is still a valid labeling for the resulting $M^*$.
Since every $\sigma$-edges in $C$ is double-tight,
$p$ satisfies (Tight) for any newly added edge to $I$.
Thus
$p$ is still an $(M^*, I^*, \calV)$-compatible $c$-potential.
Let $\alpha, \beta$ be the end nodes of $C$.
Then $U_\alpha^\sigma, V_\beta^\sigma$ are spaces that are newly matched by the resulting $(M^*, I^*)$.
This implies that, if $p$ is optimal for $(M, I, \calV)$,
then so is it for $(M^*, I^*, \calV)$ by \cref{lem:compatible}.
Therefore the following holds.
\begin{lemma}\label{lem:rearrangement}
	Let $(M, I)$ be a matching-pair of size $k$, $\calV$ a valid labeling for $M$,
	$p$ an optimal $(M, I, \calV)$-compatible $c$-potential,
	and $C$ a rearrangeable connected component with respect to $p$.
	Also let $(M^*, I^*)$ be the pair of edge subsets obtained from $(M, I)$ by the rearrangement
	with respect to $C$.
	Then $(M^*, I^*)$ is a matching-pair of size $k+1$,
	$\calV$ is a valid labeling for $M^*$,
	and $p$ is an optimal $(M^*, I^*, \calV)$-compatible $c$-potential.
\end{lemma}

\subsection{Definition of an augmenting path}\label{subsec:definition}

In this subsection, we introduce an augmenting path in our setting.
First, we define the {\it source set} and the {\it target set}
as follows,
in which nodes $\beta^\sigma$ in the former and $\alpha^{\sigma}$ in the latter can be the initial and the last nodes of an augmenting path,
respectively.
Let $\calU(M, I)$ denote the set of all nodes $\alpha^\sigma$ and $\beta^{\sigma'}$
such that $U_\alpha^\sigma$ and $V_\beta^{\sigma'}$ are unmatched by $(M, I)$.
For each $\gamma^\sigma \in \calU(M, I)$,
we denote by $\calC(\gamma^\sigma)$ the connected component of $\calG(\calV, p)|_M$ containing $\gamma^\sigma$.
The {\it source set} $\calS(M, I, \calV, p)$ and the {\it target set} $\calT(M, I, \calV, p)$ for $(M, I, \calV, p)$
are defined by
\begin{align*}
\calS(M, I, \calV, p) &\defeq \bigcup \set{\text{the nodes belonging to $\calC(\beta^\sigma)$}}[\beta^\sigma \in \calU(M, I)],\\
\calT(M, I, \calV, p) &\defeq \bigcup \set{\text{the nodes belonging to $\calC(\alpha^\sigma)$}}[ \alpha^\sigma \in \calU(M, I) ].
\end{align*}

We then define the components of an augmenting path.
An {\it outer path} $\calP$ for $(M, I, \calV, p)$ is a path in $\calG(\calV, p)$ of the form
\begin{align*}
(\beta_0^{\sigma_0} \alpha_1^{\sigma_1}, \alpha_1^{\sigma_1} \beta_1^{\sigma_1}, \dots, \beta_k^{\sigma_k} \alpha_{k+1}^{\sigma_{k+1}})
\end{align*}
such that
\begin{description}
	\item[{\rm (O1)}]
	$\beta_i \alpha_{i+1} \in E \setminus M$ for each $i = 0,1,\dots,k$
	and $\alpha_{i+1} \beta_{i+1} \in I$ for each $i = 0,1,\dots, k-1$,
	and
	\item[{\rm (O2)}]
	$A_{\alpha_{i+1} \beta_i} (U_{\alpha_{i+1}}^{\overline{\sigma_{i+1}}}, V_{\beta_i}^{\sigma_i}) = \{ 0 \}$
	for each $i = 0,1,\dots,k -1$.
\end{description}
Note that (O2) does not require $A_{\alpha_{k+1} \beta_k} (U_{\alpha_{k+1}}^{\overline{\sigma_{k+1}}}, V_{\beta_k}^{\sigma_k}) = \{ 0 \}$ on the last edge $\beta_k^{\sigma_k} \alpha_{k+1}^{\sigma_{k+1}}$.
The initial vertex $\beta_0^{\sigma_0}$
and last vertex $\alpha_{k+1}^{\sigma_{k+1}}$
are denoted by $\beta(\calP)$ and $\alpha(\calP)$,
respectively.

An {\it inner path} $\calQ$ for $(M, I, \calV, p)$ is a path in $\calG(\calV, p)$
of the form
\begin{align*}
(\alpha_0^{\sigma} \beta_1^{\sigma}, \beta_1^{\sigma} \alpha_1^{\sigma}, \dots, \alpha_k^{\sigma} \beta_{k+1}^{\sigma}),
\end{align*}
such that
\begin{description}
	\item[{\rm (I1)}]
	the underlying path $(\alpha_0 \beta_1, \beta_1 \alpha_1, \dots, \alpha_k \beta_{k+1})$ of $\calQ$ in $G$
	is included in a connected component of $M \setminus I$, and
	\item[{\rm (I2)}]
	$\alpha_0 \beta_1, \alpha_1 \beta_2, \dots, \alpha_k \beta_{k+1}$ are $\overline{\sigma}$-edges and
	$\beta_1 \alpha_1, \beta_2 \alpha_2, \dots, \beta_k \alpha_k$ are $\sigma$-edges.
\end{description}
The former condition implies that $\calQ$ can also be viewed as a $\sigma$-path in $\calG(\calV, p)|_M$,
and the latter implies that the $\sigma$-edges $\beta_1 \alpha_1, \beta_2 \alpha_2, \dots, \beta_k \alpha_k$ are double-tight (and hence rank-2).
The initial vertex $\alpha_0^{\sigma}$
and last vertex $\beta_{k+1}^{\sigma}$
are denoted by $\alpha(\calQ)$ and $\beta(\calQ)$,
respectively.

We are now ready to define an augmenting path.
Here, for paths $\calP$ and $\calQ$ in $\calG(\calV, p)$
such that the last node of $\calP$ coincides with the first node of $\calQ$,
we denote 
the concatenation of $\calP$ and $\calQ$ by $\calP \circ \calQ$.
An {\it augmenting path} $\calR$ for $(M, I, \calV, p)$ is a path in $\calG(\calV, p)$
such that
\begin{description}
	\item[{\rm (A1)}]
	$\calR$ is the concatenation $\calP_0 \circ \calQ_1 \circ \calP_1 \circ \cdots \circ \calQ_m \circ \calP_m$
	of outer paths $\calP_0, \calP_1, \dots, \calP_m$ and inner paths $\calQ_1, \dots, \calQ_m$ for $(M, I, \calV, p)$
	in which $\alpha(\calP_i) = \alpha(\calQ_{i+1})$ and $\beta(\calQ_{i+1}) = \beta(\calP_{i+1})$ for each $i$, and
	\item[{\rm (A2)}]
	$\beta(\calP_0) \in \calS(M, I, \calV, p)$ and $\alpha(\calP_m) \in \calT(M, I, \calV, p)$.
\end{description}

An augmenting path augments a matching-pair.
The following provides the validity of our augmenting procedure;
Sections~\ref{sec:preliminaries augmentation}--\ref{sec:non violate} are devoted to its proof.
\begin{theorem}\label{thm:augment}
	From a matching-pair $(M, I)$ of size $k$, a valid labeling $\calV$ for $M$,
	an optimal $(M, I, \calV)$-compatible $c$-potential $p$,
	and an augmenting path for $(M, I, \calV, p)$,
	we can obtain a matching-pair $(M^*, I^*)$ of size $k+1$ and a valid labeling $\calV^*$ for $M^*$
	such that $p$ is an optimal $(M^*, I^*, \calV^*)$-compatible $c$-potential
	in $O(\min \set{\mu, \nu}^3)$ time.
\end{theorem}

\subsection{Finding an augmenting path}\label{subsec:finding}
In this subsection,
we present an algorithm for verifying $\delta_{k+1}(A(t)) = - \infty$,
finding an optimal potential so that a rearrangeable component exists, or
finding an augmenting path.

Suppose that we are given a matching-pair $(M_0, I_0)$ of size $k < \min \set{2\mu, 2\nu}$, a valid labeling $\calV_0$ for $M_0$,
and an optimal $(M_0, I_0, \calV_0)$-compatible potential $p_0$ as the input.
Suppose further that there is no rearrangeable component with respect to $p_0$. (If it exists, we can argument $(M_0, I_0)$ by \cref{lem:rearrangement}.)
We initialize $(M, I) \leftarrow (M_0, I_0)$, $\calV \leftarrow \calV_0$, and $p \leftarrow p_0$.
In addition, during the algorithm,
we maintain a forest $\calF$ in $\calG(\calV, p)$
such that each connected component of $\calF$ has exactly one node in $\calS(M, I, \calV, p)$;
we initialize $\calF \leftarrow \calS(M, I, \calV, p)$,
which is nonempty by $k < \min \{2\mu, 2\nu\}$ and \cref{lem:source set}~(3) below.

The algorithm consists of the primal update and the dual update.
While there is an edge $\beta^{\sigma} \alpha^{\sigma'} \in \calE(\calV, p)$ such that $\beta^\sigma \in \calF$ and $\alpha^{\sigma'} \not\in \calF$,
we execute the primal update.
If there is no such edge,
then we execute the dual update.
Here, for a vector space $X \subseteq U_\alpha$,
let $X^{\perp_{\alpha \beta}}$ (or $X^{\perp_{\beta \alpha}}$) denote
the orthogonal vector space with respect to $A_{\alpha\beta}$:
\begin{align*}
X^{\perp_{\alpha \beta}} (= X^{\perp_{\beta \alpha}}) \defeq \set*{ y \in V_\beta }[$A_{\alpha\beta}(x, y) = 0$ for all $x \in X$].
\end{align*}
For a vector space $Y \subseteq V_\beta$,
$Y^{\perp_{\alpha \beta}}$ (or $Y^{\perp_{\beta \alpha}}$) is analogously defined.
\begin{description}
	\item[Primal update:]
	We first add to $\calF$ an edge $\beta^{\sigma} \alpha^{\sigma'} \in \calE(\calV, p)$ such that $\beta^\sigma \in \calF$ and $\alpha^{\sigma'} \not\in \calF$.
	\begin{description}
		\item[{\rm (P1)}] If $\alpha^{\sigma'} \in \calT(M, I, \calV, p)$,
		then
		output $(M, I, \calV, p)$ and the unique path $\calR$ in $\calF$ from a vertex in $\calS(M, I, \calV, p)$ to $\alpha^{\sigma'}$.
		Stop this procedure.
		\item[{\rm (P2)}]
		Suppose that $\alpha^{\sigma'} \not\in \calT(M, I, \calV, p)$ and $\alpha$ is incident to an edge $\alpha \beta'$ in $I$.
		Then update the valid labeling $\calV$ for $M$ as
		\begin{align*}
		U_\alpha^{\overline{\sigma'}} \leftarrow (V_\beta^{\sigma})^{\perp_{\alpha \beta}},
		\qquad V_{\beta'}^{\sigma'} \leftarrow (U_\alpha^{\overline{\sigma'}})^{\perp_{\alpha \beta'}},
		\end{align*}
		and add $\alpha^{\sigma'} {\beta'}^{\sigma'}$ to $\calF$.
		Also update $\calG(\calV, p)$ for the resulting $\calV$.
		(This case will be an expansion of an outer path.)
		\item[{\rm (P3)}]
		Suppose that $\alpha^{\sigma'} \not\in \calT(M, I, \calV, p)$ and $\alpha$ belongs to a connected component of $M \setminus I$.
		Let $\calQ$ be the longest inner path in $\calG(\calV, p)|_M$ starting with $\alpha^{\sigma'}$ such that
		$\calQ$ does not meet $\calF$.
		Then add $\calQ$ to $\calF$.
		(This case will be an addition of an inner path.)
	\end{description}
	
	\item[Dual update:]
	For each $\beta^{\sigma} \in \calF$,
	define
	\begin{align*}
	\epsilon_{\beta^\sigma} \defeq \inf \set{ p(U_\alpha^{\sigma'}) + p(V_\beta^\sigma) + c - d_{\alpha \beta} }[ \alpha^{\sigma'} \not\in \calF,~A_{\alpha \beta}(U_\alpha^{\sigma'}, V_\beta^{\sigma}) \neq \set{0} ].
	\end{align*}
	Let $\epsilon$ be the minimum value of $\epsilon_{\beta^\sigma}$ over $\beta^{\sigma} \in \calF$.
	\begin{description}
		\item[{\rm (D1)}]
		If $\epsilon = +\infty$,
		then output ``$\delta_{k+1}(A(t)) = - \infty$'' and stop this procedure.
		\item[{\rm (D2)}]
		If $\epsilon < +\infty$,
		then update $p$ and $c$ as
		\begin{align*}
		p(V_\beta^\sigma) &\leftarrow p(V_\beta^\sigma) + \epsilon \quad  \text{if $\beta^\sigma \notin \calF$},\\
		p(U_\alpha^\sigma) &\leftarrow p(U_\alpha^\sigma) + \epsilon \quad  \text{if $\alpha^\sigma \in \calF$},\\
		c &\leftarrow c - \epsilon,
		\end{align*}
		and adjust $p$ so that $p$ satisfies (Reg),
		that is,
		for each $\alpha$ and $\beta$,
		\begin{align*}
		p(X) &\leftarrow \max \{ p(U_\alpha^+), p(U_\alpha^-) \} \qquad (X \in \calM_\alpha \setminus \{ U_\alpha^+, U_\alpha^- \}),\\
		p(Y) &\leftarrow \max \{ p(V_\beta^+), p(V_\beta^-) \} \qquad (Y \in \calM_\beta \setminus \{ V_\beta^+, V_\beta^- \}).
		\end{align*}
	\begin{description}
		\item[{\rm (D2-1)}]
		Suppose that there is a rearrangeable connected component $C$ with respect to the resulting $p$.
		Then we apply the rearrangement to $(M, I)$.
		Output the resulting $(M, I, \calV, p)$ and stop this procedure.
		\item[{\rm (D2-2)}]
		Otherwise, suppose that
		the resulting target set $\calT(M, I, \calV, p)$ is enlarged.
		In this case, we have
		$\calF \cap \calT(M, I, \calV, p) \neq \emptyset$.
		Output $(M, I, \calV, p)$
		and a minimal path $\calR$ with respect to inclusion in $\calF$
		from a vertex in $\calS(M, I, \calV, p)$ to a vertex in $\calF \cap \calT(M, I, \calV, p)$.
		Stop this procedure.
		\item[{\rm (D2-3)}]
		Otherwise,
		update
		\begin{align*}
		\calF \leftarrow \calF \cup \calS(M, I, \calV, p)
		\end{align*}
		if the resulting $\calS(M, I, \calV, p)$ is enlarged.
		\qqed
		\end{description}
	\end{description}
\end{description}

The following theorem states that the above algorithm correctly works,
the proof of which is given later.
\begin{theorem}\label{thm:augmenting path}
	\begin{itemize}
		\item[{\rm (1)}]
		Suppose that the algorithm reaches {\rm (P1)} or {\rm (D2-2)}.
		Then the output $(M, I)$ is a matching-pair of size $k$,
		$\calV$ is a valid labeling for $M$,
		$p$ is an optimal $(M, I, \calV)$-compatible $c$-potential,
		and $\calR$ is an augmenting path for $(M, I, \calV, p)$.
		
		\item[{\rm (2)}]
		If the algorithm reaches {\rm (D1)},
		then $\delta_{k+1}(A(t)) = -\infty$.
		
		\item[{\rm (3)}]
		Suppose that the algorithm reaches {\rm (D2-1)}.
		Then the output $(M, I)$ is a matching-pair of size $k+1$,
		$\calV$ is a valid labeling for $M$,
		and
		$p$ is an optimal $(M, I, \calV)$-compatible $c$-potential.
		\item[{\rm (4)}]
		The running-time of the algorithm is $O(\mu \nu \min \set{\mu, \nu})$.
	\end{itemize}
\end{theorem}

\cref{thm:augment,thm:augmenting path} imply \cref{thm:main,,thm:minmax,,thm:good characterization}.
Indeed,
the above algorithm detects $\delta_{k+1}(A(t)) = -\infty$
or outputs a matching-pair $(M, I)$ of size $k+1$, a valid labeling $\calV$, and an optimal compatible potential $p$ in $O(\mu \nu \min \set{\mu, \nu})$ time by \cref{thm:augment,thm:augmenting path}.
In the latter case,
the supremum and the infimum in the minimax formula in \cref{thm:minmax}
are attained by such $(M, I)$ and $(\calV, p)$.
In particular,
since every $d_{\alpha \beta}$ is integer,
so is
$\epsilon$ in the dual update (D2),
which implies that $p$ and $c$ are integer-valued.
Thus \cref{thm:minmax,thm:good characterization} follow.
Moreover, since at most $\min\{2\mu, 2\nu\}$ augmentations occur in the algorithm,
we obtain \cref{thm:main}.

For the proof of \cref{thm:augmenting path},
we present the following three lemmas (\cref{lem:source set,,lem:edge,,lem:>}).
In particular,
\cref{lem:edge,lem:>} are frequently used for the proofs of the validity of the augmentation procedure in Sections~\ref{sec:preliminaries augmentation}--\ref{sec:non violate}.
\begin{lemma}\label{lem:source set}
Suppose that there is no rearrangeable component in $M \setminus I$ with respect to $p$.
Then the following hold:
	\begin{itemize}
		\item[{\rm (1)}]
		$\calS(M, I, \calV, p) \cap \calT(M, I, \calV, p) = \emptyset$.
		\item[{\rm (2)}]
		For each $\gamma^\sigma \in \calU(M, I)$,
		$\calC(\gamma^\sigma)$ forms an even length $\sigma$-path in $\calG(\calV, p)|_M$ which starts with $\gamma^\sigma$.
		\item[{\rm (3)}]
		$\card{\calS(M, I, \calV, p) \cap \set{ \beta^+, \beta^- }[ \beta ]} - \card{\calS(M, I, \calV, p) \cap \set{ \alpha^+, \alpha^- }[ \alpha ]} = 2\nu - k$.
	\end{itemize}
\end{lemma}
\begin{proof}
For $\gamma^\sigma \in \calU(M, I)$,
we denote by $V(\calC(\gamma^\sigma))$ the set of nodes belonging to $\calC(\gamma^\sigma)$.

(1).
Suppose, to the contrary, that $\calS(M, I, \calV, p) \cap \calT(M, I, \calV, p) \neq \emptyset$.
Then $\bigcup_{\alpha^\sigma \in \calU(M, I)} V(\calC(\alpha^\sigma)) \cap \bigcup_{\beta^\sigma \in \calU(M, I)} V(\calC(\beta^\sigma)) \neq \emptyset$.
That is,
there is a path connected component $C$ of $M \setminus I$
with odd length
of the form $(\beta_0 \alpha_1, \alpha_1 \beta_1, \dots, \beta_k \alpha_{k+1})$
such that $\beta_0 \alpha_1$ and $\beta_k \alpha_{k+1}$ are $\sigma$-edges and
a $\sigma$-path $(\beta_0^\sigma \alpha_1^\sigma, \alpha_1^\sigma \beta_1^\sigma, \dots, \beta_k^\sigma \alpha_{k+1}^\sigma)$
exists in $\calG(\calV, p)|_M$.
This implies that all $\sigma$-edges $\beta_0 \alpha_1, \beta_1 \alpha_2, \dots, \beta_k \alpha_{k+1}$
are double-tight,
i.e., $C$ is rearrangeable,
which
contradicts the assumption that there is no rearrangeable component.

(2).
We only show that, for each $\beta^\sigma \in \calU(M, I)$,
$\calC(\beta^\sigma)$ forms an even length $\sigma$-path in $\calG(\calV, p)|_M$ which starts with $\beta^\sigma$;
the case for $\alpha^\sigma \in \calU(M, I)$ is similar.
Note that $\beta$ is not incident to an edge in $I$.
By~\eqref{eq:+-}, $\calC(\beta^\sigma)$ consists of nodes labeled by $\sigma$.
Suppose, to the contrary, that $\calC(\beta^\sigma)$ forms an odd length $\sigma$-path,
i.e., $\calC(\beta^\sigma) = (\beta^\sigma = \beta_0^\sigma \alpha_1^\sigma, \alpha_1^\sigma \beta_1^\sigma, \dots, \beta_k^\sigma \alpha_{k+1}^\sigma)$.
Note that $\beta_0 \alpha_1$ and $\beta_k \alpha_{k+1}$ are $\sigma$-edges.
If $\deg_M(\alpha_{k+1}) = 2$,
then $\alpha_{k+1}$ is incident to a $\overline{\sigma}$-edge $\alpha_{k+1} \beta_{k+1}$ in $M$,
and
hence, the edge $\alpha_{k+1}^\sigma \beta_{k+1}^\sigma$ belongs to $\calE(\calV, p)$.
This implies that $\beta_{k+1}^\sigma$ also belongs to $\calC_\beta$;
a contradiction.
If $\deg_M(\alpha_{k+1}) = 1$,
then the connected component $C$ of $M \setminus I$ containing $\beta_0$ and $\alpha_{k+1}$
forms an odd length path component 
and the end edges $\beta_0 \alpha_1$ and $\beta_k \alpha_{k+1}$ are $\sigma$-edges.
Thus all $\sigma$-edges in $C$ are double-tight.
This contradicts the assumption.

(3).
By~(2),
we have $\card{V(\calC(\beta^\sigma)) \cap \set{ \beta^+, \beta^- }[ \beta ]} - \card{V(\calC(\beta^\sigma)) \cap \set{ \alpha^+, \alpha^- }[\mid \alpha ]} = 1$
for each $\beta^\sigma \in \calU(M, I)$.
Thus we obtain $\card{\calS(M, I, \calV, p) \cap \set{ \beta^+, \beta^- }[ \beta ]} - \card{\calS(M, I, \calV, p) \cap \set{ \alpha^+, \alpha^- }[ \alpha ]} = \card{\set{ \beta^\sigma }[ \beta \in \calU(M, I) ]}
= 2\nu - k$.
\end{proof}
\begin{lemma}\label{lem:edge}
	Suppose that there is an edge $\alpha^\sigma \beta^{\sigma'}$ in $\calE(\calV, p)$.
	Then
	the vector spaces $(V_{\beta}^{\sigma'})^{\perp_{\beta \alpha}}$ and $(U_{\alpha}^{\sigma})^{\perp_{\alpha \beta}}$ belong to $\calM_{\alpha}$ and $\calM_{\beta}$,
	and are different from $U_{\alpha}^{\sigma}$ and $V_{\beta}^{\sigma'}$,
	respectively.
	Moreover,
	$(V_{\beta}^{\sigma'})^{\perp_{\beta \alpha}} = U_{\alpha}^{\overline{\sigma}}$ if $p(U_{\alpha}^{\overline{\sigma}}) < p(U_{\alpha}^\sigma)$,
	and $(U_{\alpha}^\sigma)^{\perp_{\alpha \beta}} = V_{\beta}^{\overline{\sigma'}}$ if $p(V_{\beta}^{\overline{\sigma'}}) < p(V_{\beta}^{\sigma'})$.
\end{lemma}
\begin{proof}
	We only show the statements on $(V_{\beta}^{\sigma'})^{\perp_{\beta \alpha}}$.
	By $\alpha^\sigma \beta^{\sigma'} \in \calE(\calV, p)$,
	we have $A_{\alpha \beta}(U_{\alpha}^\sigma, V_{\beta}^{\sigma'}) \neq \set{0}$ and $p(U_{\alpha}^\sigma) + p(V_{\beta}^{\sigma'}) + c = d_{\alpha \beta}$.
	By the former,
	we have $V_{\beta}^{\sigma'} \not\subseteq \kerR(A_{\alpha \beta})$.
	Hence $(V_{\beta}^{\sigma'})^{\perp_{\beta \alpha}}$ is a 1-dimensional vector space; i.e., it belongs to $\calM_{\alpha}$,
	and $(V_{\beta}^{\sigma'})^{\perp_{\beta \alpha}} \neq U_{\alpha}^\sigma$.
	Moreover,
	if $p(U_{\alpha}^{\overline{\sigma}}) < p(U_{\alpha}^\sigma)$,
	the identity $p(U_{\alpha}^\sigma) + p(V_{\beta}^{\sigma'}) + c = d_{\alpha \beta}$ implies $p(U_{\alpha}^{\overline{\sigma}}) + p(V_{\beta}^{\sigma'}) + c < d_{\alpha \beta}$.
	Since $p$ is a $c$-potential,
	we have $A_{\alpha \beta}(U_{\alpha}^{\overline{\sigma}}, V_{\beta}^{\sigma'}) = \set{0}$.
	Thus we obtain $(V_{\beta}^{\sigma'})^{\perp_{\beta \alpha}} = U_{\alpha}^{\overline{\sigma}}$.
\end{proof}
Let $(M, I)$ be a matching-pair, $\calV = \labeling$ a valid labeling for $M$,
and
$p$ an $(M, I, \calV)$-compatible $c$-potential.
A labeling $\hat{\calV} = (\{\hat{U}_\alpha^+, \hat{U}_\alpha^+\}, \{\hat{V}_\beta^+, \hat{V}_\beta^-\} )_{\alpha, \beta}$
is said to be \emph{equivalent to $\calV$ with respect to $p$}, denoted by $\hat{\calV} \simeq_p \calV$,
if $\hat{\calV}$ is a valid labeling for $M$
and $(p(U_\alpha^+), p(U_\alpha^-)) = (p(\hat{U}_\alpha^+), p(\hat{U}_\alpha^-))$ and $(p(V_\beta^+), p(V_\beta^-)) = (p(\hat{V}_\beta^+), p(\hat{V}_\beta^-))$ for each $\alpha, \beta$.
\begin{lemma}\label{lem:>}
    Let $\hat{\calV} = (\{\hat{U}_\alpha^+, \hat{U}_\alpha^+\}, \{\hat{V}_\beta^+, \hat{V}_\beta^-\} )_{\alpha, \beta}$ be a labeling equivalent to $\calV$ with respect to $p$.
	Then the following hold:
	\begin{itemize}
	    \item[{\rm (1)}]
	    $p$ is an $(M, I, \hat{\calV})$-compatible $c$-potential.
	    In addition, if $p$ is optimal for $(M, I, \hat{\calV})$,
	    then $p$ is also optimal for $(M, I, \hat{\calV})$.
		\item[{\rm (2)}]
		$\calG(\hat{\calV}, p)|_M = \calG(\calV, p)|_M$.
		\item[{\rm (3)}]
		Suppose that $U_\alpha^\sigma = \hat{U}_{\alpha}^\sigma$ or $p(U_{\alpha}^{\sigma}) > p(U_{\alpha}^{\overline{\sigma}})$,
		and that $V_\beta^{\sigma'} = \hat{V}_{\beta}^{\sigma'}$ or $p(V_{\beta}^{\sigma'}) > p(V_{\beta}^{\overline{\sigma'}})$.
		Then the edge $\alpha^\sigma \beta^{\sigma'}$ exists in $\calE(\calV, p)$
		if and only if it exists in $\calE(\hat{\calV}, p)$.
	\end{itemize}
\end{lemma}
\begin{proof}
    (1) and (2) immediately follow from the definition of $\calV \simeq_p \hat{\calV}$.

	(3).
	There are four cases:
	(i) $U_\alpha^\sigma = \hat{U}_{\alpha}^\sigma$ and $V_\beta^{\sigma'} = \hat{V}_{\beta}^{\sigma'}$,
	(ii) $U_\alpha^\sigma = \hat{U}_{\alpha}^\sigma$ and $p(V_{\beta}^{\sigma'}) > p(V_{\beta}^{\overline{\sigma'}})$,
	(iii) $p(U_{\alpha}^{\sigma}) > p(U_{\alpha}^{\overline{\sigma}})$ and $V_\beta^{\sigma'} = \hat{V}_{\beta}^{\sigma'}$,
	and (iv) $p(U_{\alpha}^{\sigma}) > p(U_{\alpha}^{\overline{\sigma}})$ and $p(V_{\beta}^{\sigma'}) > p(V_{\beta}^{\overline{\sigma'}})$.
	We note that $p(U_{\alpha}^{\sigma}) > p(U_{\alpha}^{\overline{\sigma}}) = p(\hat{U}_{\alpha}^{\overline{\sigma}})$ (resp. $p(V_{\beta}^{\sigma'}) > p(V_{\beta}^{\overline{\sigma'}}) = p(\hat{V}_{\beta}^{\overline{\sigma'}})$)
	implies $U_{\alpha}^{\overline{\sigma}} = \hat{U}_{\alpha}^{\overline{\sigma}}$
	(resp. $V_{\beta}^{\overline{\sigma'}} = \hat{V}_{\beta}^{\overline{\sigma'}}$)
	by (Reg).
	By symmetry, it suffices to show that $\alpha^\sigma \beta^{\sigma'} \in \calE(\calV, p)$ implies $\alpha^\sigma \beta^{\sigma'} \in \calE(\hat{\calV}, p)$ for each case.
	
	(i).
	It is clear that $\alpha^\sigma \beta^{\sigma'} \in \calE(\calV, p)$ implies $\alpha^\sigma \beta^{\sigma'} \in \calE(\hat{\calV}, p)$.
	
	(ii) and (iii).
	We only consider (ii);
	(iii) follows from the same argument.
	Suppose $\alpha^\sigma \beta^{\sigma'} \in \calE(\calV, p)$.
	Then we have $p(U_\alpha^\sigma) + p(V_{\beta}^{\sigma'}) + c = d_{\alpha \beta}$ and $A_{\alpha \beta}(U_\alpha^\sigma, V_{\beta}^{\sigma'}) \neq \set{0}$.
	By $U_\alpha^\sigma = \hat{U}_{\alpha}^\sigma$ and $p(V_{\beta}^{\sigma'}) > p(V_{\beta}^{\overline{\sigma'}}) = p(\hat{V}_{\beta}^{\overline{\sigma'}})$,
	we have $p(\hat{U}_{\alpha}^\sigma) + p(\hat{V}_{\beta}^{\overline{\sigma'}}) + c < d_{\alpha \beta}$,
	which implies $A_{\alpha \beta}(\hat{U}_\alpha^{\sigma}, \hat{V}_\beta^{\overline{\sigma'}}) = \set{0}$
	since $p$ is a $c$-potential.
	By $\hat{V}_{\beta}^{\sigma'} \neq \hat{V}_{\beta}^{\overline{\sigma'}}$,
	we obtain $A_{\alpha \beta}(\hat{U}_\alpha^{\sigma}, \hat{V}_\beta^{\sigma'}) \neq \set{0}$.
	Therefore $\alpha^\sigma \beta^{\sigma'}$ belongs to $\calE(\hat{\calV}, p)$.
	
	(iv).
	Suppose $\alpha^\sigma \beta^{\sigma'} \in \calE(\calV, p)$.
	By the assumption and $p(U_\alpha^\sigma) + p(V_{\beta}^{\sigma'}) + c = d_{\alpha \beta}$,
	all of $p(U_{\alpha}^{\overline{\sigma}}) + p(V_{\beta}^{\sigma'})$, $p(U_{\alpha}^{\sigma}) + p(V_{\beta}^{\overline{\sigma'}})$,
	and $p(U_{\alpha}^{\overline{\sigma}}) + p(V_{\beta}^{\overline{\sigma'}})$
	are smaller than $d_{\alpha \beta}$.
	Hence
	we have $A_{\alpha \beta}(U_\alpha^{\overline{\sigma}}, V_{\beta}^{\sigma'}) = A_{\alpha \beta}(U_\alpha^{\sigma}, V_{\beta}^{\overline{\sigma'}}) = A_{\alpha \beta}(U_\alpha^{\overline{\sigma}}, V_{\beta}^{\overline{\sigma'}}) = \set{0}$.
	This implies that $\alpha \beta$ is rank-1,
	$U_\alpha^{\overline{\sigma}} = \kerL(A_{\alpha \beta})$,
	and $V_\beta^{\overline{\sigma'}} = \kerR(A_{\alpha \beta})$.
	It follows from $\hat{U}_\alpha^\sigma \neq \hat{U}_\alpha^{\overline{\sigma}} = U_\alpha^{\overline{\sigma}}$ and $\hat{V}_\beta^{\sigma'} \neq \hat{V}_\beta^{\overline{\sigma'}} = V_\beta^{\overline{\sigma'}}$ that $A_{\alpha \beta}(\hat{U}_\alpha^\sigma, \hat{V}_\beta^{\sigma'}) \neq \set{0}$.
	Thus $\alpha^\sigma \beta^{\sigma'}$ belongs to $\calE(\hat{\calV}, p)$.
\end{proof}

We are ready to prove \cref{thm:augmenting path}.
\begin{proof}[Proof of \cref{thm:augmenting path}]
A path $(\beta_0^{\sigma_0} \alpha_1^{\sigma_1}, \alpha_1^{\sigma_1} \beta_1^{\sigma_1}, \dots, \alpha_k^{\sigma_k} \beta_k^{\sigma_k})$ in
$\calG(\calV, p)$
is called a {\it truncated outer path} for $(M, I, \calV, p)$
if it satisfies (O1) and (O2);
note that the last node is not $\alpha_{k+1}^{\sigma_{k+1}}$ but $\beta_k^{\sigma_k}$.
Also a path $\calR$ in $\calG(\calV, p)$
is called a {\it truncated augmenting path} for $(M, I, \calV, p)$
if it is the concatenation $\calP_0 \circ \calQ_1 \circ \calP_1 \circ \cdots \circ \calQ_m \circ \calP$
of outer paths $\calP_0, \calP_1, \dots, \calP_m$, inner paths $\calQ_1, \dots, \calQ_m$,
and a truncated outer path $\calP$,
where $\calP$ can be empty,
such that $\beta(\calP_0) \in \calS(M, I, \calV, p)$.
We can easily see that, for a truncated augmenting path $\calR$ with the end node $\beta^\sigma$ and an edge $\beta^\sigma \alpha^{\sigma'}$ such that $\beta \alpha \in E \setminus M$ and $\alpha^{\sigma'} \in \calT(M, I, \calV, p)$,
the path $\calR \circ (\beta^\sigma \alpha^{\sigma'})$ forms an augmenting path for $(M, I, \calV, p)$.

We simultaneously show \cref{thm:augmenting path}~(1),~(2), and~(3).
In particular,
we prove that
at the beginning of the primal/dual update phase,
the following five conditions hold:
\begin{enumerate}
	\item
	$(M, I) = (M_0, I_0)$,
	$\calV$ is a valid labeling for $M$,
	and
	$p$ is an optimal $(M, I, \calV)$-compatible potential.
	\item
	$\calF$ is a forest in $\calG(\calV, p)$
	such that each connected component of $\calF$
	has exactly one node in $\calS(M, I, \calV, p)$.
	\item
	For each $\beta^\sigma \in \calF$,
	the unique path $\calR$ in $\calF$ from a vertex in $\calS$ to $\beta^\sigma$
	is a truncated augmenting path for $(M, I, \calV, p)$.
	\item
	For each $\alpha \beta \in M$ with $\alpha^\sigma \beta^\sigma \in \calE(\calV, p)$,
	$\alpha^{\sigma} \in \calF$
	if and only if $\beta^{\sigma} \in \calF$.
	\item
	$\card{\calF \cap \set{ \beta^+, \beta^- }[ \beta ]} - \card{\calF \cap \set{ \alpha^+, \alpha^- }[ \alpha ]} = 2\nu - k$.
\end{enumerate}
These are true for the initial phase $\calF = \calS(M, I, \calV, p)$ (particularly, condition~5 follows from \cref{lem:source set}~(3)).
We see that the primal and dual updates keep these conditions.

(Primal update).
Suppose that we choose an edge $\beta^{\sigma} \alpha^{\sigma'} \in \calE(\calV, p)$ such that $\beta^\sigma \in \calF$ and $\alpha^{\sigma'} \not\in \calF$ to $\calF$.
Then $\beta \alpha$ belongs to $E \setminus M$ by condition~4 and $\beta^\sigma \in \calF \not\ni \alpha^{\sigma'}$.
Let $\calR$ be the truncated augmenting path for $(M, I, \calV, p)$
that ends at $\beta^{\sigma}$ (which uniquely exists by condition~3).

(P1).
By $\alpha^{\sigma'} \in \calT(M, I, \calV, p)$,
the output $\calR \circ (\beta^\sigma \alpha^{\sigma'})$ forms an augmenting path for $(M, I, \calV, p)$.
Condition~1 implies that \cref{thm:augmenting path}~(1) for (P1) holds.

(P2).
Suppose that $\alpha$ is incident to an edge $\alpha \beta'$ in $I$.
By \cref{lem:edge},
we have
$(V_\beta^{\sigma})^{\perp_{\alpha \beta}} \in \calM_\alpha$ and $(V_\beta^{\sigma})^{\perp_{\alpha \beta}} \neq U_\alpha^{\sigma'}$.
Furthermore $p((V_\beta^{\sigma})^{\perp_{\alpha \beta}}) = p(U_\alpha^{\overline{\sigma'}})$ holds.
Indeed,
if $p(U_\alpha^{\sigma'}) > p(U_\alpha^{\overline{\sigma'}})$,
then \cref{lem:edge} asserts $(V_\beta^{\sigma})^{\perp_{\alpha \beta}} = U_\alpha^{\overline{\sigma'}}$.
If $p(U_\alpha^{\sigma'}) \leq p(U_\alpha^{\overline{\sigma'}})$,
then the inequality $(V_\beta^{\sigma})^{\perp_{\alpha \beta}} \neq U_\alpha^{\sigma'}$ implies $p((V_\beta^{\sigma})^{\perp_{\alpha \beta}}) = p(U_\alpha^{\overline{\sigma'}})$ by (Reg).
Therefore, by update $U_\alpha^{\overline{\sigma'}} \leftarrow (V_\beta^{\sigma})^{\perp_{\alpha \beta}}$
and $V_{\beta'}^{\sigma'} \leftarrow (U_\alpha^{\overline{\sigma'}})^{\perp_{\alpha \beta'}}$,
the resulting $\calV$ is a valid labeling for $M$ and is equivalent to the previous one with respect to $p$.
By \cref{lem:>}~(1), $p$ is an optimal $(M, I, \calV)$-compatible $c$-potential.
Thus condition~1 holds.

The update of $\calV$ can change $\calG(\calV, p)$,
particularly,
the set of edges incident to $\alpha^{\overline{\sigma'}}$ or ${\beta'}^{\sigma'}$.
We here show that $\calF \cup \{ \beta^\sigma \alpha^{\sigma'}, \alpha^{\sigma'} {\beta'}^{\sigma'} \} \subseteq \calE(\calV, p)$,
i.e., no edge in $\calF \cup \{ \beta^\sigma \alpha^{\sigma'}, \alpha^{\sigma'} {\beta'}^{\sigma'} \}$ is deleted from $\calE(\calV, p)$ by this update.
By $\alpha \beta' \in I$ and \cref{lem:>}~(2),
we have
$\alpha^+ {\beta'}^+, \alpha^- {\beta'}^- \in \calE(\calV, p)$,
implying $\alpha^{\sigma'} {\beta'}^{\sigma'} \in \calE(\calV, p)$.
By condition~4,
$\alpha^{\sigma'}$ and ${\beta'}^{\sigma'}$ do not belong to $\calF$,
particularly,
no edge incident to ${\beta'}^{\sigma'}$ is in $\calF$.
Hence it suffices to see that if there is ${\beta''}^{\sigma''} \alpha^{\overline{\sigma'}} \in \calF$ for some $\beta'' \neq \beta'$,
then ${\beta''}^{\sigma''} \alpha^{\overline{\sigma'}} \in \calE(\calV, p)$.
Suppose ${\beta''}^{\sigma''} \alpha^{\overline{\sigma'}} \in \calF$.
Then we have $\alpha^{\overline{\sigma'}} {\beta'}^{\overline{\sigma'}} \in \calF$ by (P2).
Since the path in $\calF$ from a vertex in $\calS$ to ${\beta'}^{\overline{\sigma'}}$,
whose last edge is $\alpha^{\overline{\sigma'}} {\beta'}^{\overline{\sigma'}}$,
is a truncated augmenting path by condition~3,
we have
$U_\alpha^{\sigma'} = (V_{\beta''}^{\sigma''})^{\perp_{\alpha \beta''}}$.
By $U_\alpha^{\sigma'} \neq U_\alpha^{\overline{\sigma'}}$ and the equivalence of $\calV$,
we have $A_{\alpha \beta''}(U_\alpha^{\overline{\sigma'}}, V_{\beta''}^{\sigma''}) \neq \set{0}$
and $p(U_\alpha^{\overline{\sigma'}}) + p(V_{\beta''}^{\sigma''}) + c = d_{\alpha \beta''}$.
Thus we obtain ${\beta''}^{\sigma''} \alpha^{\overline{\sigma'}} \in \calE(\calV, p)$.

We finally prove that $\calF \cup \{ \beta^\sigma \alpha^{\sigma'}, \alpha^{\sigma'} {\beta'}^{\sigma'} \}$ satisfies conditions~2--5 after the update.
By condition~4 for $\calF$ and $\alpha^{\sigma'} \not\in \calF$,
we have
${\beta'}^{\sigma'} \not\in \calF$.
Hence $\calF \cup \{ \beta^\sigma \alpha^{\sigma'}, \alpha^{\sigma'} {\beta'}^{\sigma'} \}$
satisfies condition~4.
By $\calF \subseteq \calE(\calV, p)$,
condition~2 still holds for $\calF$.
In addition, $\alpha^{\sigma'} \not\in \calF$ and
${\beta'}^{\sigma'} \not\in \calF$
immediately imply
that
$\calF \cup \{ \beta^\sigma \alpha^{\sigma'}, \alpha^{\sigma'} {\beta'}^{\sigma'} \}$
satisfies conditions~2 and~5.
For condition~3,
it suffices to see that the unique path $\calR \circ (\beta^{\sigma} \alpha^{\sigma'}, \alpha^{\sigma'} {\beta'}^{\sigma'})$ in $\calF \cup \{ \beta^\sigma \alpha^{\sigma'}, \alpha^{\sigma'} {\beta'}^{\sigma'} \}$ from a vertex in $\calS(M, I, \calV, p)$ to ${\beta'}^{\sigma'}$ is a truncated augmenting path.
This follows from $\alpha \beta' \in I$
and
$A_{\alpha \beta}(U_{\alpha}^{\overline{\sigma'}}, V_\beta^\sigma) = \set{0}$ by the update.

(P3).
Suppose that $\alpha^{\sigma'} \not\in \calT(M, I, \calV, p)$ and $\alpha$ belongs to a connected component of $M \setminus I$.
Let $\calQ$ be the longest inner path in $\calG(\calV, p)|_M$ starting with $\alpha^{\sigma'}$ such that
$\calQ$ does not meet $\calF$.

Since $\calV$ and $p$ do not change in this update,
condition~1 clearly holds.
Condition~2 follows from the fact that every node in $\calQ$ exits $\calF$ by the definition.
Since $\calQ$ is of the form $(\alpha^{\sigma'} = \alpha_0^{\sigma'} \beta_1^{\sigma'}, \beta_1^{\sigma'} \alpha_1^{\sigma'}, \dots, \alpha_k^{\sigma'} \beta_{k+1}^{\sigma'})$,
in which $\alpha_0 \beta_1, \alpha_1 \beta_2, \dots, \alpha_k \beta_{k+1}$ are $\overline{\sigma'}$-edges,
$\calF \cup \{ \beta^\sigma \alpha^{\sigma'} \} \cup \calQ$ satisfies condition~5.
For each $\beta_i^{\sigma'}$ belonging to $\calQ$,
the unique path from a vertex in $\calS$ to $\beta_i^{\sigma'}$
is a truncated augmenting path (which does not have a truncated outer path).
Thus condition~3 holds.
Finally, we show that $\calF \cup \{ \beta^\sigma \alpha^{\sigma'} \} \cup \calQ$ satisfies the condition~4.
Suppose, to the contrary, that there is a $\sigma'$-edge $\beta_{k+1} \alpha_{k+1} \in M$ such that $\beta_{k+1}^{\sigma'} \alpha_{k+1}^{\sigma'} \in \calE(\calV, p)$
and $\alpha_{k+1}^{\sigma'} \notin \calF$.
Then a $\overline{\sigma'}$-edge $\alpha_{k+1} \beta_{k+2}$ also exists in $M$.
Indeed, otherwise $U_{\alpha_{k+1}}^{\sigma'}$ is unmatched by $(M, I)$ by $\deg_M(\alpha_{k+1}) = 1$,
which implies $\alpha_0^{\sigma'} \in \calT(M, I, \calV, p)$.
This contradicts the assumption of $\alpha_0^{\sigma'} \notin \calT(M, I, \calV, p)$ on (P3).
Thus $\alpha_{k+1}^{\sigma'} \beta_{k+2}^{\sigma'}$ also exists in $\calE(\calV, p)$,
since $\alpha_{k+1} \beta_{k+2}$ is a $\overline{\sigma'}$-edge.
By condition~4 for $\calF$ and $\alpha_{k+1}^{\sigma'} \notin \calF$, we have $\beta_{k+2}^{\sigma'} \notin \calF$.
Hence $\calQ \circ (\beta_{k+1}^{\sigma'} \alpha_{k+1}^{\sigma'}, \alpha_{k+1}^{\sigma'} \beta_{k+2}^{\sigma'})$ also forms an inner path for $(M, I, \calV, p)$ not meeting $\calF$,
which contradicts the maximality of $\calQ$.
Therefore $\calF \cup \{ \beta^\sigma \alpha^{\sigma'} \} \cup \calQ$ satisfies condition~4.

(Dual update).
Since there is no edge between $\beta^{\sigma}$ and $\alpha^{\sigma'}$ in $\calE(\calV, p)$ such that $\beta^\sigma \in \calF$ and $\alpha^{\sigma'} \not\in \calF$,
the minimum value $\epsilon$ of $\epsilon_{\beta^\sigma}$ over $\beta^\sigma \in \calF$ is positive by definition.
For $\epsilon'$ with $0 \leq \epsilon' \leq \epsilon$ and $\epsilon' < +\infty$,
let us define $c_{\epsilon'}$ by
$c_{\epsilon'} \defeq c - \epsilon'$
and
$p_{\epsilon'}$ by
\begin{align*}
p_{\epsilon'}(V_\beta^\sigma) \defeq
\begin{cases}
p(V_\beta^\sigma) & \text{if $\beta^\sigma \in \calF$},\\
p(V_\beta^\sigma) + \epsilon' & \text{if $\beta^\sigma \not\in \calF$},
\end{cases} \qquad
p_{\epsilon'}(U_\alpha^\sigma) \defeq
\begin{cases}
p(U_\alpha^\sigma) + \epsilon' & \text{if $\alpha^\sigma \in \calF$},\\
p(U_\alpha^\sigma) & \text{if $\alpha^\sigma \not\in \calF$},
\end{cases}
\end{align*}
and
\begin{align*}
p_{\epsilon'}(X) &\defeq \max \{ p_{\epsilon'}(U_\alpha^+), p_{\epsilon'}(U_\alpha^-) \} \qquad (X \in \calM_\alpha \setminus \{ U_\alpha^+, U_\alpha^- \}),\\
p_{\epsilon'}(Y) &\defeq \max \{ p_{\epsilon'}(V_\beta^+), p_{\epsilon'}(V_\beta^-) \} \qquad (Y \in \calM_\beta \setminus \{ V_\beta^+, V_\beta^- \})
\end{align*}
for each $\alpha$ and $\beta$.
Then the following claim holds:
\begin{claim*}
	The function $\funcdoms{p_{\epsilon'}}{\bigcup_{\gamma} \calM_\gamma}{\setR}$ is an optimal $(M, I, \calV)$-compatible $c_{\epsilon'}$-potential
	and $\calF \subseteq \calE(\calV, p_{\epsilon'})$.
	In particular,
	if $\epsilon < + \infty$,
	then there is at least one edge $\beta^\sigma \alpha^{\sigma'} \in \calE(\calV, p_{\epsilon})$
	such that $\beta^\sigma \in \calF$ and $\alpha^{\sigma'} \not\in \calF$.
\end{claim*}
\begin{proof}[Proof of Claim]
	We first show that $p_{\epsilon'}$ is a $c_{\epsilon'}$-potential.
	Take arbitrary $U_{\alpha}^\sigma$ and $V_{\beta}^{\sigma'}$
	with
	$p_{\epsilon'}(U_{\alpha}^\sigma) + p_{\epsilon'}(V_{\beta}^{\sigma'}) + c_{\epsilon'} < d_{\alpha \beta}$.
	We see $A_{\alpha \beta}(U_\alpha^\sigma, V_\beta^{\sigma'}) = \set{0}$.
	If $\alpha^\sigma \in \calF$ or $\beta^{\sigma'} \notin \calF$,
	then $d_{\alpha \beta} > p_{\epsilon'}(U_\alpha^\sigma) + p_{\epsilon'}(V_\beta^{\sigma'}) + c_{\epsilon'} \geq p(U_\alpha^\sigma) + p(V_\beta^{\sigma'}) + c$.
	Since $p$ is a $c$-potential,
	we have $A_{\alpha \beta}(U_\alpha^\sigma, V_\beta^{\sigma'}) = \set{0}$.
	If $\alpha^\sigma \not\in \calF$ and $\beta^{\sigma'} \in \calF$,
	then $p_{\epsilon'}(U_\alpha^\sigma) + p_{\epsilon'}(V_\beta^{\sigma'}) + c_{\epsilon'} = p(U_\alpha^\sigma) + p(V_\beta^{\sigma'}) + c - \epsilon' < d_{\alpha \beta}$.
	By $\epsilon' \leq \epsilon \leq \epsilon_{\beta^\sigma}$ and the definition of $\epsilon_{\beta^\sigma}$,
	it must hold that $A_{\alpha \beta}(U_\alpha^\sigma, V_\beta^{\sigma'}) = \set{0}$.
	Since $p_{\epsilon'}$ satisfies (Reg) by definition,
	it is a $c_{\epsilon'}$-potential by \cref{lem:(Reg)}.

    By condition~4, for each $\alpha^\sigma \beta^\sigma \in \calE(\calV, p)|_M$,
    we have either $\alpha^\sigma, \beta^\sigma \in \calF$ or $\alpha^\sigma, \beta^\sigma \notin \calF$,
    i.e., $p(U_\alpha^\sigma) + p(V_\beta^\sigma) + c = p_{\epsilon'}(U_\alpha^\sigma) + p_{\epsilon'}(V_\beta^\sigma) + c_{\epsilon'}$.
    Thus we obtain $\calE(\calV, p)|_M \subseteq \calE(\calV, p_{\epsilon'})|_M$,
    which implies that $p_{\epsilon'}$ satisfies (Tight).
	For each $U_\alpha^\sigma$ and $V_\beta^{\sigma'}$
	unmatched by $(M, I)$,
	we have $\alpha^\sigma \in \calT(M, I, \calV, p)$ and $\beta^{\sigma'} \in \calS(M, I, \calV, p)$.
	It follows from $\calT(M, I, \calV, p) \cap \calF = \emptyset$ and $\calS(M, I, \calV, p) \subseteq \calF$ that $p_{\epsilon'}(U_\alpha^\sigma) = p(U_\alpha^\sigma) = 0$ and $p_{\epsilon'}(V_\beta^{\sigma'}) = p(V_\beta^{\sigma'}) = 0$,
	which implies $p_{\epsilon'}$ satisfies (Zero).
	Hence $p_{\epsilon'}$ is an optimal $(M, I, \calV)$-compatible $c_{\epsilon'}$-potential by \cref{lem:compatible}.
	
	For each $\alpha^\sigma \beta^{\sigma'} \in \calF$,
	we have $p_{\epsilon'}(U_\alpha^{\sigma}) + p_{\epsilon'}(V_\beta^{\sigma}) + c_{\epsilon'} = p(U_\alpha^{\sigma}) + p(V_\beta^{\sigma}) + c = d_{\alpha \beta}$,
	Therefore we obtain $\calF \subseteq \calE(\calV, p_{\epsilon'})$.
	
	If $\epsilon < +\infty$,
	then there are $\beta^\sigma \in \calF$ and $\alpha^{\sigma'} \not\in \calF$ such that
	$A_{\alpha \beta}(U_\alpha^{\sigma'}, V_\beta^\sigma) \neq \set{0}$
	and $p(U_\alpha^{\sigma'}) + p(V_\beta^\sigma) + c = d_{\alpha \beta} + \epsilon$.
	For such $\beta^\sigma$ and $\alpha^{\sigma'}$,
	we have $p_{\epsilon}(U_\alpha^{\sigma'}) + p_{\epsilon}(V_\beta^\sigma) + c_{\epsilon'} = p(U_\alpha^{\sigma'}) + p(V_\beta^\sigma) + c - \epsilon = d_{\alpha \beta}$
	by the definition of $p_{\epsilon}$.
	Hence $\beta^\sigma \alpha^{\sigma'}$ belongs to $\calE(\calV, p_{\epsilon})$.
\end{proof}

We consider the case of $\epsilon = +\infty$, corresponding to (D1).
By Claim, the function $p_{\epsilon'}$ is a $c_{\epsilon'}$-potential
for any $\epsilon' \geq 0$.
By condition~5,
we have $2\nu - |\calF \cap \{ \beta^+, \beta^- \mid \beta \}| + |\calF \cap \{ \alpha^+, \alpha^- \mid \alpha \}| = k$.
Hence
$p_{\epsilon'}(\calV) + (k+1)c_{\epsilon'} = p(\calV) + k\epsilon' + (k+1)(c - \epsilon') = p(\calV) + (k+1)c - \epsilon'$.
Therefore
\begin{align*}
\inf_{\epsilon' \geq 0} \left\{ p_{\epsilon'}(\calV) + (k+1)c_{\epsilon'} \right\} = -\infty.
\end{align*}
By the weak duality in \cref{thm:minmax},
we obtain $\delta_{k+1}(A(t)) = -\infty$,
which implies \cref{thm:augmenting path}~(2).

We then consider the case of $\epsilon < +\infty$, corresponding to (D2).
By Claim,
$p_{\epsilon}$ is an optimal $(M, I, \calV)$-compatible $c_{\epsilon}$-potential.

Suppose that (D2-1) occurs.
Then we apply the rearrangement to $(M, I)$ with respect to some rearrangeable component in the algorithm;
the resulting edge set pair is denoted by $(M^*, I^*)$.
By \cref{lem:rearrangement},
$(M^*, I^*)$ is a matching-pair of size $k+1$,
$\calV$ is a valid labeling for $M^*$,
and
$p_\epsilon$ is an optimal $(M^*, I^*, \calV)$-compatible $c$-potential.
By updating $(M, I) \leftarrow (M^*, I^*)$ and $p \leftarrow p_{\epsilon}$,
we obtain \cref{thm:augmenting path}~(3).

In the following, we assume that there is no rearrangeable component in $M \setminus I$.
By Claim, $p_{\epsilon}$ is an optimal $(M, I, \calV)$-compatible $c_{\epsilon}$-potential,
and hence, condition~1 holds.
By the definition of $p_{\epsilon}$,
if $\beta^\sigma \alpha^{\sigma'} \in \calE(\calV, p)$
and either $\beta^\sigma, \alpha^{\sigma'} \in \calF$
or $\beta^\sigma, \alpha^{\sigma'} \notin \calF$,
then $\beta^\sigma \alpha^{\sigma'} \in \calE(\calV, p_\epsilon)$.
By $\calS(M, I, \calV, p) \subseteq \calF$ and $\calT(M, I, \calV, p) \cap \calF = \emptyset$,
we obtain $\calS(M, I, \calV, p) \subseteq \calS(M, I, \calV, p_\epsilon)$ and $\calT(M, I, \calV, p) \subseteq \calT(M, I, \calV, p_\epsilon)$.
By \cref{lem:source set}~(2),
$\calC(\beta^\sigma)$ forms an even length path for $\beta^\sigma \in \calU(M, I)$.
If $\calS(M, I, \calV, p_\epsilon) \setminus \calS(M, I, \calV, p) \neq \emptyset$,
then there is a $\sigma$-edge $\beta \alpha \in M$ such that $\beta^\sigma \in \calS(M, I, \calV, p)$
and $\beta^\sigma \alpha^\sigma \in \calE(\calV, p_{\epsilon}) \setminus \calE(\calV, p)$.
This implies that $\alpha^\sigma \notin \calF$.
Thus, by condition~4, $\calF$ exits $\calS(M, I, \calV, p_\epsilon) \setminus \calS(M, I, \calV, p)$.
We consider additional two cases:
(i)
$\calT(M, I, \calV, p_\epsilon) \supsetneq \calT(M, I, \calV, p)$
and (ii) $\calT(M, I, \calV, p_\epsilon) = \calT(M, I, \calV, p)$.

(i).
This case corresponds to (D2-2).
By $\calT(M, I, \calV, p_\epsilon) \supsetneq \calT(M, I, \calV, p)$,
there is a $\sigma$-edge $\beta \alpha \in M$
such that $\alpha^\sigma \in \calT(M, I, \calV, p)$ and
$\alpha^\sigma \beta^\sigma \in \calE(\calV, p_{\epsilon}) \setminus \calE(\calV, p)$.
This implies that $\beta^\sigma \in \calF \cap \calT(M, I, \calV, p_\epsilon)$.
Thus there is a path in $\calF$ from a vertex in $\calS(M, I, \calV, p) (\subseteq \calS(M, I, \calV, p_\epsilon))$ to a vertex in $\calT(M, I, \calV, p_\epsilon)$.

Let $\calR$ be such a minimal path in $\calF$ with respect to inclusion,
i.e.,  the first node of $\calR$ belongs to $\calS(M, I, \calV, p)$,
the last node belongs to $\calT(M, I, \calV, p_\epsilon)$,
and all intermediate nodes exit $\calS(M, I, \calV, p) \cup \calT(M, I, \calV, p_\epsilon)$.
Furthermore, since $\calF$ exits $\calS(M, I, \calV, p_\epsilon) \setminus \calS(M, I, \calV, p)$,
all intermediate nodes also exit $\calS(M, I, \calV, p_\epsilon)$.
We show that $\calR$ is an augmenting path for $(M, I, \calV, p_\epsilon)$.
By \cref{lem:source set}~(2) and the minimality of $\calR$ with respect to the inclusion,
the last node of $\calR$ belongs to $\calT(M, I, \calV, p_{\epsilon}) \cap \{ \alpha^+, \alpha^- \mid \alpha \}$,
say, $\alpha^\sigma$,
and the last edge ${\beta'}^{\sigma'} \alpha^\sigma$ of $\calR$ exits $\calE(\calV, p_{\epsilon})|_M$.
Since $\calF$ is included in $\calE(\calV, p_\epsilon)$ by Claim,
$\calR$ forms an augmenting path for $(M, I, \calV, p_\epsilon)$;
\cref{thm:augmenting path}~(1) for (D2-2) holds.

(ii).
This case corresponds to (D2-3).
Since $\calF$ exits $\calS(M, I, \calV, p_\epsilon) \setminus \calS(M, I, \calV, p)$,
$\calF \cup \calS(M, I, \calV, p_\epsilon)$ satisfies conditions~2 and~3.
Let $\calC(\beta^\sigma)_\epsilon$ be the connected component of $\calG(\calV, p_\epsilon)|_M$
containing $\beta^\sigma$.
By \cref{lem:source set}~(2),
$\calC(\beta^\sigma)_\epsilon \setminus \calC(\beta^\sigma)$
forms an even length $\sigma$-path in $\calG(\calV, p_\epsilon)|_M$.
Hence conditions~4 and~5 hold.

(4).
One primal or dual update can be performed in $O(|\calE(\calV, p)|) = O(\mu \nu)$ time.
By condition~5 and $\card{\set{\beta^+, \beta^-}[\beta]} \leq 2\nu$,
the number of nodes in $\set{\alpha^+, \alpha^-}[\alpha]$ covered with $\calF$
is bounded by $k \leq \min\set{\mu, \nu}$.
In the primal update,
if the algorithm does not stop,
then the number of nodes in $\set{\alpha^+, \alpha^-}[\alpha]$ covered with $\calF$ increases by at least one.
In the dual update,
if the algorithm does not stop,
then
at least one edge $\beta^\sigma \alpha^{\sigma'}$
such that $\beta^\sigma \in \calF$ and $\alpha^{\sigma'} \not\in \calF$ appears in $\calE(\calV, p)$ by Claim;
the next phase is primal.
Thus
the number of iterations of the primal/dual updates in the algorithm is $O(\min\set{\mu, \nu})$.
Hence the running-time of the algorithm is $O(\mu \nu \min\set{\mu, \nu})$.
\end{proof}

\section{Preliminaries for the augmentation procedure}\label{sec:preliminaries augmentation}
In this section,
we introduce concepts of {\it pseudo augmenting path} (\cref{subsec:pseudo augmenting path}),
{\it front-propagation} (\cref{subsec:front-propagation}),
and {\it back-propagation} (\cref{subsec:back-propagation})
for describing an augmentation procedure and proving its validity.
We also introduce two no short-cut conditions---named ($\No$) and ($\Ni$)---on an augmenting path
in \cref{subsec:no short-cut conditions}.
\cref{subsec:theta} is devoted to introducing two quantities used to estimate the time complexity of the augmentation procedure.
During the augmentation,
although an $(M, I, \calV)$-compatible $c$-potential $p$ may not satisfy (Zero) (or may not be optimal) for $(M, I, \calV)$,
it satisfies (Zero)$'$ for $\calR$ that is a weaker condition than (Zero).
This is formally introduced in \cref{subsec:IH}.
In \cref{subsec:outline},
we provide an outline of our augmentation.

We here employ several notations.
For an outer path $\calP = (\beta_0^{\sigma_0} \alpha_1^{\sigma_1}, \alpha_1^{\sigma_1} \beta_1^{\sigma_1}, \dots, \beta_k^{\sigma_k} \alpha_{k+1}^{\sigma_{k+1}})$,
its italic font $P$ denotes
the underlying walk $(\beta_0 \alpha_1, \alpha_1 \beta_1, \dots, \beta_k \alpha_{k+1})$ in $G$.
Similarly, the underlying path of an inner path $\calQ$
is denoted by its italic font $Q$.
An outer path $\calP$ is said to be {\it simple}
if the underlying walk $P$ in $G$ is actually a path,
i.e.,
it does not use the same edge twice.
For an outer or inner path $\mathcal{X} = (\gamma_0^{\sigma_0} \gamma_1^{\sigma_1}, \gamma_1^{\sigma_1} \gamma_2^{\sigma_2}, \dots, \gamma_k^{\sigma_k} \gamma_{k+1}^{\sigma_{k+1}})$ and $i, j$ with $0 \leq i < j \leq k+1$,
we define $\mathcal{X}[\gamma_i^{\sigma_i}, \gamma_j^{\sigma_j}]$ by
the subpath of $\mathcal{X}$ from $\gamma_i^{\sigma_i}$ to $\gamma_j^{\sigma_j}$.
In particular,
if $\gamma_i^{\sigma_i}$ is the initial node $\gamma_0^{\sigma_0}$ of $\mathcal{X}$,
then we denote $\mathcal{X}[\gamma_0^{\sigma_0}, \gamma_j^{\sigma_j}]$ by
$\mathcal{X}(\gamma_j^{\sigma_j}]$.
If $\gamma_j^{\sigma_j}$ is the last node $\gamma_{k+1}^{\sigma_{k+1}}$ of $\mathcal{X}$,
then we denote $\mathcal{X}[\gamma_i^{\sigma_i}, \gamma_{k+1}^{\sigma_{k+1}}]$ by
$\mathcal{X}[\gamma_i^{\sigma_i})$.
For a component $C$ of $G$,
we denote the vertex set of $C$ by $V(C)$.

In the following,
let $(M, I)$ denote a matching-pair of size $k$, $\calV$ a valid labeling for $M$,
and
$p$ an $(M, I, \calV)$-compatible $c$-potential.

\subsection{A pseudo augmenting path}\label{subsec:pseudo augmenting path}
A {\it pseudo outer path} $\calP$ for $(M, I, \calV, p)$ is a path in $\calG(\calV, p)$ of the form
\begin{align*}
(\beta_0^{\sigma_0} \alpha_1^{\sigma_1}, \alpha_1^{\sigma_1} \beta_1^{\sigma_1}, \dots, \beta_k^{\sigma_k} \alpha_{k+1}^{\sigma_{k+1}})
\end{align*}
satisfying (O1) and the following (O2)$'$:
\begin{description}
	\item[{\rm (O2)$'$}]
	$\beta_i^{\sigma_j} \alpha_{i+1}^{\overline{\sigma_{i+1}}} \not\in \calE(\calV, p)$
	for each $i = 0,1,\dots,k -1$.
\end{description}
That is,
the condition (O2) that $A_{\alpha_{i+1} \beta_i} (U_{\alpha_{i+1}}^{\overline{\sigma_{i+1}}}, V_{\beta_i}^{\sigma_i}) = \{ 0 \}$
of an outer path
is weakened
to $\beta_i^{\sigma_j} \alpha_{i+1}^{\overline{\sigma_{i+1}}} \not\in \calE(\calV, p)$
in the condition (O2)$'$ of an pseudo outer path.

A {\it pseudo augmenting path} $\calR$ for $(M, I, \calV, p)$ is a path in $\calG(\calV, p)$
satisfying (A2) and the following (A1)$'$:
\begin{description}
	\item[{\rm (A1)$'$}]
	$\calR$ is the concatenation $\calP_0 \circ \calQ_1 \circ \calP_1 \circ \cdots \circ \calQ_m \circ \calP_m$
	of pseudo outer paths $\calP_0, \calP_1, \dots, \calP_m$ and inner paths $\calQ_1, \dots, \calQ_m$ for $(M, I, \calV, p)$
	in which $\alpha(\calP_i) = \alpha(\calQ_{i+1})$ and $\beta(\calQ_{i+1}) = \beta(\calP_{i+1})$ for each $i$.
\end{description}

\subsection{Front-propagation}\label{subsec:front-propagation}
Let $\calP = (\beta_0^{\sigma_0} \alpha_1^{\sigma_1}, \alpha_1^{\sigma_1} \beta_1^{\sigma_1}, \dots, \beta_k^{\sigma_k} \alpha_{k+1}^{\sigma_{k+1}})$ be a pseudo outer path for $(M, I, \calV, p)$.
The {\it front-propagation} of $\calP$ is a sequence $(Y_0, X_1, Y_1, \dots, Y_k, X_{k+1})$
of 1-dimensional vector spaces
such that $Y_0 \defeq V_{\beta_0}^{\sigma_0}$ and for each $i$,
\begin{align*}
X_i \defeq
(Y_{i-1})^{\perp_{\beta_{i-1} \alpha_i}}, \qquad
Y_i \defeq (X_i)^{\perp_{\alpha_i \beta_i}}.
\end{align*}
Note here that
if $\calP$ is an outer path,
then
the front-propagation of $\calP$ is
\begin{align*}
(V_{\beta_i}^{\sigma_0}, U_{\alpha_1}^{\overline{\sigma_1}}, V_{\beta_1}^{\sigma_1}, \dots, V_{\beta_k}^{\sigma_k}, (V_{\beta_k}^{\sigma_k})^{\perp_{\beta_k \alpha_{k+1}}})
\end{align*}
by (O2).

Suppose that we replace $V_{\beta_0}^{\sigma_0}, U_{\alpha_1}^{\overline{\sigma_1}}, V_{\beta_1}^{\sigma_1}, \dots, V_{\beta_i}^{\sigma_i}$ in $\calV$ with $Y_0, X_1, Y_1, \dots, Y_i$,
respectively;
we refer to the resulting as $\calV(\calP(\beta_i^{\sigma_i}])$.
The following holds:
\begin{lemma}\label{lem:pseudo outer path}
	For each $i = 0,1, \dots, k$,
	we have
	$\calV(\calP(\beta_i^{\sigma_i}]) \simeq_p \calV$
	and
	$\calG(\calV(\calP(\beta_i^{\sigma_i}]), p) = \calG(\calV, p)$.
	Moreover, $\calP$ forms an outer path for $(M, I, \calV(\calP(\beta_k^{\sigma_k}]), p)$.
\end{lemma}
\begin{proof}
	Let $(Y_0, X_1, Y_1, \dots, Y_k, X_{k+1})$ be the front-propagation of $\calP$.
	For $i = 0,1,\dots,k$, let $\calV_i \defeq \calV(\calP(\beta_i^{\sigma_i}])$.
	We show
	$\calV_i \simeq_p \calV$
	and
	$\calG(\calV_i, p) = \calG(\calV, p)$
	by induction on $i$.
	In addition,
	we also show that, if $p(U_{\alpha_{i}}^{\overline{\sigma_{i}}}) \leq p(U_{\alpha_{i}}^{\sigma_{i}})$
	then $X_i = U_{\alpha_{i}}^{\overline{\sigma_{i}}}$,
	and if $p(V_{\beta_{i}}^{\sigma_{i}}) \leq p(V_{\beta_{i}}^{\overline{\sigma_{i}}})$
	then $Y_i = V_{\beta_{i}}^{\sigma_{i}}$.

	The case $i = 0$ is clear by $\calV_0 = \calV$ and $Y_0 = V_{\beta_0}^{\sigma_0}$.
	Suppose $0 \leq i-1 < k$.
	By the induction hypothesis of $\calG(\calV_{i-1}, p) = \calG(\calV, p)$,
	we have
	$\beta_{i-1}^{\sigma_{i-1}} \alpha_{i}^{\sigma_{i}} \in \calE(\calV_{i-1}, p)$
	and $\beta_{i-1}^{\sigma_{i-1}} \alpha_{i}^{\overline{\sigma_{i}}} \notin \calE(\calV_{i-1}, p)$.
	There are two cases:
	(i) $p(U_{\alpha_{i}}^{\sigma_{i}}) \geq p(U_{\alpha_{i}}^{\overline{\sigma_{i}}})$
	and (ii) $p(U_{\alpha_{i}}^{\sigma_{i}}) < p(U_{\alpha_{i}}^{\overline{\sigma_{i}}})$.
	Note that $Y_{i-1}$ is the $\sigma_{i-1}$-space of $\beta_{i-1}$ with respect to $\calV_{i-1}$,
	and that
	$U_{\alpha_{i}}^{\overline{\sigma_{i}}}$ and $V_{\beta_{i}}^{\sigma_{i}}$ are the $\overline{\sigma_{i}}$-space of $\alpha_{i}$ and the $\sigma_{i}$-space of $\beta_{i}$ with respect to $\calV_{i-1}$,
	respectively.

	(i) $p(U_{\alpha_{i}}^{\sigma_{i}}) \geq p(U_{\alpha_{i}}^{\overline{\sigma_{i}}})$.
	By $\alpha_{i} \beta_{i} \in I$ and (Tight),
	we have $p(V_{\beta_{i}}^{\sigma_{i}}) \leq p(V_{\beta_{i}}^{\overline{\sigma_{i}}})$.
	It suffices to show that $X_{i} = U_{\alpha_{i}}^{\overline{\sigma_{i}}}$.
	Indeed,
	the identity $X_{i} = U_{\alpha_{i}}^{\overline{\sigma_{i}}}$ implies $Y_{i} = V_{\beta_{i}}^{\sigma_{i}}$,
	i.e., $\calV_{i} = \calV_{i-1}$.
	Thus
	$\calV_{i} \simeq_p \calV$ and
	$\calG(\calV_{i}, p) = \calG(\calV, p)$
	by the induction hypothesis.

	We denote by $X$ the $\sigma_{i}$-space of $\alpha_{i}$ with respect to $\calV_{i-1}$;
	note that $X$ may be different from $U_{\alpha_{i}}^{\sigma_{i}}$
	if $\alpha_{i}^{\overline{\sigma_{i}}} = \alpha_l^{\sigma_l}$ for some $l \leq i-1$.
	Since $p$ satisfies (Reg),
	we have $p(X) = p(U_{\alpha_{i}}^{\sigma_{i}}) \geq p(U_{\alpha_{i}}^{\overline{\sigma_{i}}})$.
	By $\beta_{i-1}^{\sigma_{i-1}} \alpha_{i}^{\sigma_{i}} \in \calE(\calV_{i-1}, p)$ and \cref{lem:edge},
	we obtain $X_{i} (= (Y_{i-1})^{\perp_{\beta_{i-1} \alpha_{i}}}) \in \calM_{\alpha_{i}}$.
	Furthermore,
	it follows from $\beta_{i-1}^{\sigma_{i-1}} \alpha_{i}^{\sigma_{i}} \in \calE(\calV_{i-1}, p)$
	that $d_{\alpha_{i} \beta_{i-1}} = p(X) + p(Y_{i-1}) + c \geq p(U_{\alpha_{i}}^{\overline{\sigma_{i}}}) + p(Y_{i-1}) + c$.
	By $\beta_{i-1}^{\sigma_{i-1}} \alpha_{i}^{\overline{\sigma_{i}}} \notin \calE(\calV_{i-1}, p)$,
	we have $A_{\alpha_{i} \beta_{i-1}}(U_{\alpha_{i}}^{\overline{\sigma_{i}}}, Y_{i-1}) = \set{0}$.
	Thus we obtain $X_{i} = (Y_{i-1})^{\perp_{\beta_{i-1} \alpha_{i}}} = U_{\alpha_{i}}^{\overline{\sigma_{i}}}$, as required.

	(ii) $p(U_{\alpha_{i}}^{\sigma_{i}}) < p(U_{\alpha_{i}}^{\overline{\sigma_{i}}})$.
	By $\alpha_{i} \beta_{i} \in I$ and (Tight),
	we have $p(V_{\beta_{i}}^{\sigma_{i}}) > p(V_{\beta_{i}}^{\overline{\sigma_{i}}})$.
	By the induction hypothesis,
	the $\sigma_{i}$-space of $\alpha_{i}$
	and the $\overline{\sigma_{i}}$-space of $\beta_{i}$
	with respect to $\calV_{i-1}$ are $U_{\alpha_{i}}^{\sigma_{i}}$
	and $V_{\beta_{i}}^{\overline{\sigma_{i}}}$,
	respectively.
	By $\beta_{i-1}^{\sigma_{i-1}} \alpha_{i}^{\sigma_{i}} \in \calE(\calV_{i-1}, p)$ and \cref{lem:edge},
	we obtain $X_{i} = (Y_{i-1})^{\perp_{\beta_{i-1} \alpha_{i}}} \in \calM_{\alpha_i}$ and
	$X_{i} \neq U_{\alpha_{i}}^{\sigma_{i}}$.
	Since $\alpha_{i} \beta_{i}$ is rank-2,
	we also obtain $Y_{i} = (X_{i})^{\perp_{\alpha_{i} \beta_{i}}} \in \calM_{\beta_{i}}$
	and $Y_{i} \neq V_{\beta_{i}}^{\sigma_{i}}$.
	Hence $\calV_{i}$ is a valid labeling for $M$.
	Since $p$ satisfies (Reg),
	we have $p(X_{i}) = p(U_{\alpha_{i}}^{\overline{\sigma_{i}}}) > p(U_{\alpha_{i}}^{\sigma_{i}})$
	and $p(Y_{i}) = p(V_{\beta_{i}}^{\sigma_{i}}) > p(V_{\beta_{i}}^{\overline{\sigma_{i}}})$.
	Thus $\calV_i \simeq_p \calV_{i-1}$.
	Furthermore \cref{lem:>}~(3) asserts $\calG(\calV_{i}, p) = \calG(\calV_{i-1}, p)$.
	
	In the case of $i = k$,
	$\calP$ is a path in $\calG(\calV_k, p) = \calG(\calV, p)$ that clearly satisfies (O1).
	Furthermore, $\calP$ satisfies (O2)
	by $A_{\alpha_{i+1} \beta_i}(X_{i+1}, Y_i) = \set{0}$ for $i = 0,1,\dots, k-1$.
	Therefore $\calP$ is an outer path for $(M, I, \calV, p)$.
	
	This completes the proof.
\end{proof}

Let $\calR = \calP_0 \circ \calQ_1 \circ \calP_1 \circ \cdots \circ \calQ_m \circ \calP_m$ be a pseudo augmenting path for $(M, I, \calV, p)$.
We define $\calV(\calR)$ by the labeling obtained from $\calV$ by executing the front-propagation of $\calP_i$
for each $i$. That is,
\begin{align*}
\calV(\calR) \defeq \calV(\calP_0(\beta'(\calP_0)])(\calP_1(\beta'(\calP_1)])) \cdots (\calP_m(\beta'(\calP_m)]),
\end{align*}
where $\beta'(\calP_i)$ denotes the node adjacent to $\alpha(\calP_i)$ in $\calP_i$,
i.e., $\beta'(\calP_i)\alpha(\calP_i)$ is the last edge of $\calP_i$.
The following proposition states that we can construct an augmenting path from a pseudo outer path via the front-propagation.
\begin{proposition}\label{prop:pseudo augmenting path}
	Let $\calR$ be a pseudo augmenting path for $(M, I, \calV, p)$.
    Then $\calV(\calR) \simeq_p \calV$
	and $\calR$
	forms an augmenting path for $(M, I, \calV(\calR), p)$.
\end{proposition}
\begin{proof}
	Suppose that $\calR = \calP_0 \circ \calQ_1 \circ \calP_1 \circ \cdots \circ \calQ_m \circ \calP_m$
	and, for each $i$, let
	\begin{align*}
	\calV(\calR)_i \defeq \calV(\calP_0(\beta'(\calP_0)])(\calP_1(\beta'(\calP_1)])) \cdots (\calP_i(\beta'(\calP_i)]).
	\end{align*}
	We show by induction on $i$ that
	$\calV(\calR)_i \simeq_p \calV$
	and $\calR$ is a pseudo augmenting path for $(M, I, \calV(\calR)_i, p)$
	such that $\calP_0, \calP_1, \dots, \calP_i$
	are outer paths for $(M, I, \calV(\calR)_i, p)$.

	Let $\calV(\calR)_{-1} \defeq \calV$.
	We consider $i \geq 0$.
	Since $\calP_{i}$ is a pseudo outer path for $(M, I, \calV(\calR)_{i-1}, p)$,
	\cref{lem:pseudo outer path} and the induction hypothesis assert that
	$\calV(\calR)_{i} \simeq_p \calV(\calR)_{i-1} \simeq_p \calV$,
	$\calG(\calV(\calR)_{i}, p) = \calG(\calV(\calR)_{i-1}, p) = \calG(\calV, p)$,
	and $\calP_{i}$ is an outer path for $(M, I, \calV(\calR)_{i}, p)$.
	By $\calG(\calV(\calR)_{i}, p) = \calG(\calV, p)$,
	all $\calP_1, \calP_2, \dots, \calP_{k}$ are still pseudo outer paths for $(M, I, \calV(\calR)_{i}, p)$.
	In particular,
	since $\calP_{i}$ does not meet $\calP_l$ with $l \leq i-1$,
	$\calP_1, \calP_2, \dots, \calP_{i-1}$ are still outer paths for $(M, I, \calV(\calR)_{i}, p)$.
	It follows from $\calG(\calV, p) = \calG(\calV(\calR)_i, p)$
	that
	$\calQ_j$ forms an inner path for $(M, I, \calV(\calR)_i, p)$ for each $j$,
	$\calS(M, I, \calV, p) = \calS(M, I, \calV(\calR)_i, p)$,
	and $\calT(M, I, \calV, p) = \calT(M, I, \calV(\calR)_i, p)$.
	Thus $\calR$ is a pseudo augmenting path for $(M, I, \calV(\calR)_i, p)$
	such that $\calP_0, \calP_1, \dots, \calP_i$
	are outer paths for $(M, I, \calV(\calR)_i, p)$.
\end{proof}

\subsection{Back-propagation}\label{subsec:back-propagation}
Let $\calP = (\beta_0^{\sigma_0} \alpha_1^{\sigma_1}, \alpha_1^{\sigma_1} \beta_1^{\sigma_1}, \dots, \beta_k^{\sigma_k} \alpha_{k+1}^{\sigma_{k+1}})$ be an outer path for $(M, I, \calV, p)$.
The {\it back-propagation} of $\calP$ is a sequence $(Y_0, X_1, Y_1, \dots, Y_k, X_{k+1})$
such that $X_{k+1} \defeq U_{\alpha_{k+1}}^{\sigma_{k+1}}$ and for each $i$,
\begin{align*}
Y_{i-1} \defeq
(X_i)^{\perp_{\alpha_i \beta_{i-1}}}, \qquad
X_i \defeq (Y_i)^{\perp_{\beta_i \alpha_i}}.
\end{align*}
\begin{lemma}\label{lem:outer path}
	Let $(Y_0, X_1, Y_1, \dots, X_{k+1})$ be the back-propagation of $\calP$.
	For each $i = 0, 1,\dots, k$,
	$X_{i+1}$
	belongs to $\calM_{\alpha_{i+1}}$ and
	is different from $U_{\alpha_{i+1}}^{\overline{\sigma_{i+1}}}$,
	and
	$Y_i$ belongs to
	$\calM_{\beta_i}$
	and
	is different from $V_{\beta_i}^{\sigma_i}$.
	Moreover,
	if $p(U_{\alpha_i}^{\sigma_i}) < p(U_{\alpha_i}^{\overline{\sigma_i}})$ (resp. $p(V_{\beta_i}^{\overline{\sigma_i}}) < p(V_{\beta_i}^{\sigma_i})$),
	then $X_i$ is equal to $U_{\alpha_i}^{\sigma_i}$ (resp. $Y_i$ is equal to $V_{\beta_i}^{\overline{\sigma_i}}$).
\end{lemma}
\begin{proof}
	It suffices to prove that
	\begin{itemize}
		\item
		if $X_{i+1} \in \calM_{\alpha_{i+1}}$ and $X_{i+1} \neq U_{\alpha_{i+1}}^{\overline{\sigma_{i+1}}}$,
		then $Y_i \in \calM_{\beta_i}$ and $Y_i \neq V_{\beta_i}^{\sigma_i}$, and
		\item
		if
		$p(V_{\beta_i}^{\sigma_i}) > p(V_{\beta_i}^{\overline{\sigma_i}})$,
		then
		$Y_i = V_{\beta_i}^{\overline{\sigma_i}}$.
	\end{itemize}
	Indeed,
	by $\beta_i \alpha_i \in I$ and (Tight),
	the edge $\beta_i \alpha_i$ is rank-2, $U_{\alpha_{i}}^{\overline{\sigma_{i}}} = (V_{\beta_i}^{\sigma_i})^{\perp_{\beta_i \alpha_i}}$,
	$U_{\alpha_{i}}^{\sigma_{i}} = (V_{\beta_i}^{\overline{\sigma_i}})^{\perp_{\beta_i \alpha_i}}$,
	and $p(U_{\alpha_i}^{\sigma_i}) + p(V_{\beta_i}^{\sigma_i}) + c = p(U_{\alpha_i}^{\overline{\sigma_i}}) + p(V_{\beta_i}^{\overline{\sigma_i}}) + c$.
	Hence
	$Y_i \in \calM_{\beta_i}$, $Y_i \neq V_{\beta_i}^{\sigma_i}$, $p(V_{\beta_i}^{\sigma_i}) > p(V_{\beta_i}^{\overline{\sigma_i}})$,
	and
	$Y_i = V_{\beta_i}^{\overline{\sigma_i}}$ immediately imply $X_{i} \in \calM_{\alpha_{i}}$, $X_{i} \neq U_{\alpha_{i}}^{\overline{\sigma_{i}}}$,
	$p(U_{\alpha_i}^{\sigma_i}) < p(U_{\alpha_i}^{\overline{\sigma_i}})$,
	and $X_i = U_{\alpha_i}^{\sigma_i}$,
	respectively.

	We show both the bullets by induction on $i = k, k-1, \dots, 0$.
	By definition,
	we have $X_{k+1} = U_{\alpha_{k+1}}^{\sigma_{k+1}}$,
	which clearly belongs to $\calM_{\alpha_{k+1}}$ and is different from $U_{\alpha_{k+1}}^{\overline{\sigma_{k+1}}}$.
	Since the edge $\beta_k^{\sigma_{k}} \alpha_{k+1}^{\sigma_{k+1}}$ exists in $\calE(\calV, p)$,
	\cref{lem:edge} asserts that
	$Y_k =
	(U_{\alpha_{k+1}}^{\overline{\sigma_{k+1}}})^{\perp_{\alpha_{k+1} \beta_k}}$
	belongs to $\calM_{\beta_k}$
	and is different from $V_{\beta_k}^{\sigma_{k}}$.
	In addition,
	if $p(V_{\beta_k}^{\sigma_k}) > p(V_{\beta_k}^{\overline{\sigma_k}})$,
	then we have $Y_k = V_{\beta_k}^{\overline{\sigma_k}}$ by $X_{k+1} = U_{\alpha_{k+1}}^{\sigma_{k+1}}$, $\beta_k^{\sigma_{k}} \alpha_{k+1}^{\sigma_{k+1}} \in \calE(\calV, p)$, and \cref{lem:edge}.
	
	Assume $X_{i+1} \in \calM_{\alpha_{i+1}}$ and $X_{i+1} \neq U_{\alpha_{i+1}}^{\overline{\sigma_{i+1}}}$ for some $i < k$.
	If $\alpha_{i+1} \beta_i$ is rank-2,
	then $V_{\beta_i}^{\sigma_i} = (U_{\alpha_{i+1}}^{\overline{\sigma_{i+1}}})^{\perp_{\alpha_{i+1} \beta_i}}$ by (O2)
	and $Y_i =
	(X_{i+1})^{\perp_{\alpha_{i+1} \beta_i}}$.
	Hence we have $Y_i \in \calM_{\beta_i}$ and $Y_i \neq V_{\beta_i}^{\sigma_i}$.
	If $\alpha_{i+1} \beta_i$ is rank-1,
	then $X_{i+1} \neq U_{\alpha_{i+1}}^{\overline{\sigma_{i+1}}} = \ker_{\rm L}(A_{\alpha_{i+1} \beta_i})$
	by the assumption and (O2),
	which implies $Y_i = \ker_{\rm R}(A_{\alpha_{i+1} \beta_i})$.
	Moreover,
	we have $V_{\beta_i}^{\sigma_i} \neq \ker_{\rm R}(A_{\alpha_{i+1} \beta_i})$ by $\beta_i^{\sigma_i} \alpha_{i+1}^{\sigma_{i+1}} \in \calE(\calV, p)$.
	Thus we obtain $Y_i \neq V_{\beta_i}^{\sigma_i}$.

	Assume $p(V_{\beta_i}^{\sigma_i}) > p(V_{\beta_i}^{\overline{\sigma_i}})$ for some $i$.
	If $p(U_{\alpha_{i+1}}^{\sigma_{i+1}}) < p(U_{\alpha_{i+1}}^{\overline{\sigma_{i+1}}})$,
	then $X_{i+1} = U_{\alpha_{i+1}}^{\sigma_{i+1}}$ by the induction hypothesis.
	In this case, the identity $Y_i = V_{\beta_i}^{\overline{\sigma_i}}$
	follows from $\beta_k^{\sigma_{k}} \alpha_{k+1}^{\sigma_{k+1}} \in \calE(\calV, p)$ and \cref{lem:edge}.
	Suppose that $p(U_{\alpha_{i+1}}^{\sigma_{i+1}}) \geq p(U_{\alpha_{i+1}}^{\overline{\sigma_{i+1}}})$.
	Then we have $d_{\alpha_{i+1} \beta_i} = p(U_{\alpha_{i+1}}^{\sigma_{i+1}}) + p(V_{\beta_i}^{\sigma_i}) + c
	> p(U_{\alpha_{i+1}}^{\sigma_{i+1}}) + p(V_{\beta_i}^{\overline{\sigma_i}}) + c \geq p(U_{\alpha_{i+1}}^{\overline{\sigma_{i+1}}}) + p(V_{\beta_i}^{\overline{\sigma_i}}) + c$;
	the equality follows from $\beta_i^{\sigma_i} \alpha_{i+1}^{\sigma_{i+1}} \in \calE(\calV, p)$.
	Since $p$ is a $c$-potential,
	it must hold that $A_{\alpha_{i+1} \beta_i}(U_{\alpha_{i+1}}^{\sigma_{i+1}}, V_{\beta_i}^{\overline{\sigma_i}}) = 
	A_{\alpha_{i+1} \beta_i}(U_{\alpha_{i+1}}^{\overline{\sigma_{i+1}}}, V_{\beta_i}^{\overline{\sigma_i}}) = \set{0}$.
	This implies that $\alpha_{i+1} \beta_i$ is rank-1 and $V_{\beta_i}^{\overline{\sigma_i}} = \kerR(A_{\alpha_{i+1} \beta_i})$.
	By $Y_i = (X_{i+1})^{\perp_{\alpha_{i+1} \beta_i}}$,
	$Y_i$ is equal to $V_{\beta_i}^{\overline{\sigma_i}}$.
	
	This completes the proof.
\end{proof}

Suppose that we replace $U_{\alpha_i}^{\sigma_i}, V_{\beta_i}^{\overline{\sigma_i}}, \dots, V_{\beta_k}^{\overline{\sigma_k}}, U_{\alpha_{k+1}}^{\sigma_{k+1}}$
with $X_i, Y_i, \dots, Y_k, X_{k+1}$,
respectively;
we refer to the resulting as $\calV(\calP[\alpha_i^{\sigma_i})^{-1})$.
\begin{lemma}\label{lem:simple outer path}
	If $\calP$ is a simple outer path for $(M, I, \calV, p)$,
    then $\calV(\calP[\alpha_1^{\sigma_1})^{-1}) \simeq_p \calV$
	and $\calP$ is an outer path for $(M, I, \calV(\calP[\alpha_1^{\sigma_1})^{-1}), p)$.
\end{lemma}
\begin{proof}
	Let $(Y_0, X_1, Y_1, \dots, X_{k+1})$ be the back-propagation of $\calP$
	and $\calV' \defeq \calV(\calP[\alpha_1^{\sigma_1})^{-1})$.
	We first show $\calV' \simeq_p \calV$.
	Since $\calP$ is simple,
	for each $i = 1,2,\dots,k$
	the $\sigma_i$-space and $\overline{\sigma_i}$-space of $\alpha_i$ with respect to $\calV'$ are $X_i$ and $U_{\alpha_i}^{\overline{\sigma_i}}$,
	respectively.
	By
	\cref{lem:outer path}, they are different.
	Similarly, the $\sigma_i$-space and $\overline{\sigma_i}$-space of $\beta_i$ with respect to $\calV'$ are $V_{\beta_i}^{\sigma_i}$ and $Y_i$,
	respectively,
	which are different.
	By $A_{\alpha_i \beta_i}(U_{\alpha_i}^{\overline{\sigma_i}}, V_{\beta_i}^{\sigma_i}) = A_{\alpha_i \beta_i}(X_i, Y_i) = \set{0}$,
	$\calV'$ is a valid labeling for $M$.
	\cref{lem:outer path} also asserts that, if $X_i \neq U_{\alpha_i}^{\sigma_i}$ then $p(U_{\alpha_i}^{\sigma_i}) \geq p(U_{\alpha_i}^{\overline{\sigma_i}})$,
	and if $Y_i \neq V_{\beta_i}^{\overline{\sigma_i}}$
	then $p(V_{\beta_i}^{\overline{\sigma_i}}) \geq p(V_{\beta_i}^{\sigma})$.
	These imply $p(X_i) = p(U_{\alpha_i}^{\sigma_i})$ and $p(Y_i) = p(V_{\beta_i}^{\overline{\sigma_i}})$ for each $i$ by (Reg).
	Hence we have $\calV' \simeq_p \calV$.

    We then show that $\calP$ is an outer path for $(M, I, \calV', p)$.
    Clearly, $\calP$ satisfies (O1) and (O2);
    in particular, (O2) follows from the fact that the $\sigma_i$-space of $\beta_i$ and the $\overline{\sigma_{i+1}}$-space of $\alpha_{i+1}$
    are $V_{\beta_i}^{\sigma_i}$ and $U_{\alpha_{i+1}}^{\overline{\sigma_{i+1}}}$,
    respectively.
    Thus it suffices to see that $\calP$ is actually a path in $\calG(\calV', p)$,
    or $\beta_0^{\sigma_0} \alpha_1^{\sigma_1}, \beta_1^{\sigma_1} \alpha_2^{\sigma_2}, \dots, \beta_k^{\sigma_k} \alpha_{k+1}^{\sigma_{k+1}} \in \calE(\calV', p)$.
    By $X_{i+1} \neq U_{\alpha_{i+1}}^{\overline{\sigma_{i+1}}}$,
	$A_{\alpha_{i+1} \beta_i}(U_{\alpha_{i+1}}^{\overline{\sigma_{i+1}}}, V_{\beta_i}^{\sigma_i}) = \{ 0 \}$,
	and $A_{\alpha_{i+1} \beta_i}(U_{\alpha_{i+1}}^{\sigma_{i+1}}, V_{\beta_i}^{\sigma_i}) \neq \{ 0 \}$,
	we have $A_{\alpha_{i+1} \beta_i}(X_{i+1}, V_{\beta_i}^{\sigma_i}) \neq \set{0}$,
	implying $\beta_0^{\sigma_0} \alpha_1^{\sigma_1}, \beta_1^{\sigma_1} \alpha_2^{\sigma_2}, \dots, \beta_k^{\sigma_k} \alpha_{k+1}^{\sigma_{k+1}} \in \calE(\calV', p)$,
	as required.
\end{proof}

\subsection{No short-cut conditions: {\rm ($\No$)} and {\rm ($\Ni$)}}\label{subsec:no short-cut conditions}
Let $\calR = \calP_0 \circ \calQ_1 \circ \calP_1 \circ \cdots \circ \calQ_m \circ \calP_m$ be an augmenting path for $(M, I, \calV, p)$,
where the last outer path $\calP_m$ is of the form $(\beta_0^{\sigma_0} \alpha_1^{\sigma_1}, \alpha_1^{\sigma_1} \beta_1^{\sigma_1}, \dots, \beta_k^{\sigma_k} \alpha_{k+1}^{\sigma_{k+1}})$.
A node $\beta_i^{\overline{\sigma_i}}$ is said to be \emph{BP-invariant with respect to $\calP_m$}
if $p(V_{\beta_i}^+) \neq p(V_{\beta_i}^-)$, or 
$p(V_{\beta_i}^+) = p(V_{\beta_i}^-)$ and $V_{\beta_i}^{\overline{\sigma_i}}$
coincides with the space at $\beta_i^{\sigma_i}$ of the back-propagation of $\calP_m$.
Through \cref{lem:outer path},
the above condition can be rephrased as:
$p(V_{\beta_i}^{\sigma_i}) < p(V_{\beta_i}^{\overline{\sigma_i}})$, or 
$p(V_{\beta_i}^{\sigma_i}) \geq p(V_{\beta_i}^{\overline{\sigma_i}})$ and $V_{\beta_i}^{\overline{\sigma_i}}$
coincides with the space of $\beta_i^{\sigma_i}$ of the back-propagation of $\calP_m$.
``BP'' is an abbreviation of ``back-propagation.''
An outer path for $(M, I, \calV, p)$ with the last edge $\beta^\sigma \alpha^{\sigma'}$
is said to be {\it proper}
if there is no edge between $\beta^\sigma$ and $\alpha^{\overline{\sigma'}}$ in $\calE(\calV, p)$.
Let $\calQ^{\overline{\sigma_0}}$ be the maximal inner $\overline{\sigma_0}$-path for $(M, I, \calV, p)$
such that $\beta(\calQ^{\overline{\sigma_0}}) = \beta_0^{\overline{\sigma_0}}$.

The following conditions are referred to as ($\No$) and as ($\Ni$):
\begin{description}
	\item[{\rm ($\No$)}]
	The last outer path $\calP_m$ is simple, and
	every $\beta_i^{\overline{\sigma_i}}$ belonging to $\calR$ is BP-invariant with respect to $\calP_m$.
	\item[{\rm ($\Ni$)}]
	If $\beta_0^{\overline{\sigma_0}}$ is not BP-invariant with respect to $\calP_m$ and the last vertex of $P_l$ with $l \leq m-2$ belongs to $Q^{\overline{\sigma_0}}$,
	then $\calP_l$
	is proper and
	$\alpha(\calP_l)$ is a $\sigma_0$-vertex.
\end{description}

\subsection{Decremental quantities $\theta$ and $\phi$}\label{subsec:theta}
To estimate the time complexity of the augmentation procedure,
we introduce two quantities $\theta$ and $\phi$ at least one of which decreases during the algorithm.
Let $\calR = \calP_0 \circ \calQ_1 \circ \calP_1 \circ \cdots \circ \calQ_m \circ \calP_m$ be an augmenting path for $(M, I, \calV, p)$.
Since $\calR$ is a path in the bipartite graph $\calG(\calV, p)$
and $M$ consists of a path or a cycle component in the bipartite graph $G$,
we have $|\calR| = O(\min \set{\mu, \nu})$ and $|M| = O(\min \set{\mu, \nu})$.

We define $\theta(M, I, \calR)$ by
\begin{align*}
\theta(M, I, \calR) \defeq \sum_{i = 0}^m |\calP_i| + N_{\rm S}(M,I, \calR),
\end{align*}
where $N_{\rm S}(M, I, \calR)$ denotes the number of edges in the union of all connected components of $M \setminus I$ intersecting with $\bigcup_i V(Q_i) \cup V(P_m)$.
Clearly, $\theta(I, \calR) \leq |\calR| + |M| = O(\min \set{\mu, \nu})$

We then define $\phi$.
Suppose that the initial node $\beta(\calP_0)$ of $\calR$ is $\beta_0^{\sigma_0}$.
Since $\beta_0^{\sigma_0}$ belongs to $\calS(M, I, \calV, p)$ by (A2),
there uniquely exists $\beta_*^{\sigma_0} \in \calU(M, I)$ such that $\beta_0^{\sigma_0} \in \calC(\beta_*^{\sigma_0})$.
Define $N_{\rm 0}(M, I, \calR)$ as the number of edges in the path from $\beta_0^{\sigma_0}$ to $\beta_*^{\sigma_0}$
in $\calC(\beta_*^{\sigma_0})$.
We define $\phi(M, I, \calR)$ by
\begin{align*}
\phi(M, I, \calR) \defeq |\calR| + N_{\rm 0}(M, I, \calR).
\end{align*}
Clearly $\phi(M, I, \calR) \leq |\calR| + |\calS(M, I, \calV, p)| = O(\min \set{\mu, \nu})$.

\subsection{Induction hypothesis}\label{subsec:IH}
During augmentation,
we modify a matching-pair $(M, I)$ and a valid labeling $\calV$ for $M$.
Then, for the resulting $(M, I, \calV)$, $p$ is an $(M, I, \calV)$-compatible $c$-potential
but may be no longer optimal, or may not satisfy (Zero).
We require $p$ to satisfy a weaker condition (Zero)$'$ than (Zero), which we introduce below.

Let $\calR = \calP_0 \circ \calQ_1 \circ \calP_1 \circ \cdots \circ \calQ_m \circ \calP_m$
be an augmenting path for $(M, I, \calV, p)$.
By (A2), we have $\alpha(\calP_m) \in \calT(M, I, \calV, p)$.
Hence there uniquely exists a node $\alpha_*^{\sigma_*} \in \calU(M, I)$ such that $\alpha(\calP_m) \in \calC(\alpha_*^{\sigma_*})$.
The condition (Zero)$'$ for $\calR$ is the following:
\begin{description}
    \item[{\rm (Zero)$'$}]
	For each $U_\alpha^{\sigma}, V_\beta^{\sigma'}$ unmatched by $(M, I)$ except $U_{\alpha_*}^{\sigma_*}$,
	\begin{align*}
	    p(U_\alpha^\sigma) = p(V_\beta^{\sigma'}) = 0.
	\end{align*}
\end{description}

In the initial phase,
since $p$ is an optimal $(M, I, \calV)$-compatible $c$-potential,
$p$ satisfies (Zero) by \cref{lem:compatible}, and hence (Zero)$'$.

\subsection{Outline}\label{subsec:outline}
Our augmentation procedure is outlined as follows.
We consider the following three cases:
\begin{itemize}
	\item
	$\calR = \calP_0$ and $\calP_0$ is simple; it is called the base case.
	\item
	$\calR$ violates ($\No$) or ($\Ni$).
	\item
	$\calR$ satisfies both ($\No$) and ($\Ni$) but is not in the base case.
\end{itemize}
In the base case, we can augment a matching-pair, i.e.,
we obtain a matching-pair $(M^*,I^*)$ of size $k+1$ and a valid labeling $\calV^*$ for $M^*$ such that $p$ is an optimal $(M^*, I^*, \calV^*)$-compatible $c$-potential in $O(\min \set{\mu, \nu})$ time;
we terminate the augmentation procedure.
This is dealt with in \cref{sec:base}.
In the second (\cref{sec:violate}) and third cases (\cref{sec:non violate}),
we modify $(M, I, \calV, \calR)$ in $O(\min \set{\mu, \nu})$ time so that $\theta$ strictly decreases, or $\theta$ does not change and $\phi$ strictly decreases.
At the beginning of each update,
we modify $(M, I, \calV, \calR)$ in $O(\min \set{\mu, \nu})$ time so that $\calR$ satisfies three additional conditions
and the last outer path of $\calR$ becomes as short as possible in some sense,
which is described in \cref{sec:initial}.
Neither $\theta$ nor $\phi$ increases by this modification.

By $\theta(M, I, \calR) = O(\min \set{\mu, \nu})$ and $\phi(M, I, \calR) = O(\min \set{\mu, \nu})$,
the number of updates is bounded by $O(\min \set{\mu, \nu}^2)$.
Furthermore, each update takes $O(\min \set{\mu, \nu})$ time.
Hence the running-time of the procedure is bounded by $O(\min \set{\mu, \nu}^3)$,
which implies \cref{thm:augment}.

\section{Initial stage}\label{sec:initial}
This section is devoted to describing the modifications of $\calR$ so that $\calR$ satisfies the additional three conditions
and the last outer path of $\calR$ becomes shorter,
which are executed at the beginning of each update in \cref{sec:base,,sec:violate,,sec:non violate}.
Let $\calR = \calP_0 \circ \calQ_1 \circ \calP_1 \circ \cdots \circ \calQ_m \circ \calP_m$ be an augmenting path for $(M, I, \calV, p)$,
where $\calP_m = (\beta_0^{\sigma_0} \alpha_1^{\sigma_1}, \alpha_1^{\sigma_1} \beta_1^{\sigma_1}, \dots, \beta_k^{\sigma_k} \alpha_{k+1}^{\sigma_{k+1}})$.
Let $C$ denote the connected component of $M \setminus I$ containing $\alpha_{k+1}$,
and $\alpha_*^{\sigma_{k+1}}$ denote the node in $\calU(M, I)$ such that $\calC(\alpha_*^{\sigma_{k+1}})$ contains $\alpha_{k+1}^{\sigma_{k+1}}$.

In the initial stage, we execute the following:
\begin{itemize}
    \item
    If $C$ is rearrangeable,
    then we apply the rearrangement to $(M, I)$ with respect to $C$.
    Output the resulting $(M, I, \calV, p)$
    and stop the augmentation procedure.
    This is described in \cref{subsec:rearrangeable case}.
    \item
    Suppose that $C$ is not rearrangeable.
    Then we modify $(M, I)$ and $\calR$ so that the conditions (A3)--(A5) introduced in \cref{subsec:modifications} are satisfied.
    In addition,
    we appropriately modify $M$ and the last outer path in $\calR$
    so that the last outer path in $\calR$ becomes shorter,
    which is described in \cref{subsec:preprocessing}.
    This can simplify the arguments in \cref{sec:base,,sec:violate,,sec:non violate}.
\end{itemize}

\subsection{Rearrangeable case}\label{subsec:rearrangeable case}
Suppose that $C$ is rearrangeable with respect to $p$.
Then we apply the rearrangement to $(M, I)$ with respect to $C$;
the resulting is denoted by $(M^*, I^*)$.
By the same argument as in \cref{subsec:rearrangement},
$(M^*, I^*)$ is a matching-pair of size $k+1$,
$\calV$ is a valid labeling for $M$,
and $p$ is an $(M^*, I^*, \calV)$-compatible $c$-potential.
Moreover,
we have $\calU(M^*, I^*) = \calU(M, I) \setminus \{ \alpha_*^{\sigma_{k+1}}, \beta_*^{\sigma_{k+1}} \}$,
where $\alpha_*$ and $\beta_*$ are the end nodes of $C$.
Since $p$ satisfies (Zero)$'$ for $(M, I, \calV, \calR)$,
i.e., $p(U_\alpha^\sigma) = p(V_\beta^{\sigma'}) = 0$ for each $U_\alpha^{\sigma}, V_\beta^{\sigma'}$ unmatched by $(M, I)$ except $U_{\alpha_*}^{\sigma_{k+1}}$,
$p$ satisfies (Zero) for $(M^*, I^*, \calV)$.
By \cref{lem:compatible}, $p$ is optimal.
We output $(M^*, I^*, \calV, p)$, and stop this procedure.

\subsection{Additional requirements}\label{subsec:modifications}
In the following sections,
we assume that $C$ is not rearrangeable.
We consider the three additional conditions (A3)--(A5) for $\calR$:
\begin{description}
    \item[{\rm (A3)}]
	No intermediate vertices in $\calR$ belong to $\calS(M, I, \calV, p)$.
	\item[{\rm (A4)}]
	For each $l = 0,1,\dots,m-1$,
	the outer path $\calP_l$ with $\alpha(\calP_l) = \alpha^\sigma$ satisfies
	$\alpha^\sigma \notin \calC(\alpha_*^{\sigma_{k+1}})$.
	In addition, if $\calP_l$ is not proper,
	then $\alpha^{\overline{\sigma}} \notin \calC(\alpha_*^{\sigma_{k+1}})$.
	\item[{\rm (A5)}]
	$\deg_M(\alpha_{k+1}) \leq 1$, and if $\deg_M(\alpha_{k+1}) = 1$ then $\alpha_{k+1}$ is incident to a $\sigma_{k+1}$-edge in $M$.
\end{description}
If $\calR$ satisfies (A1) and (A2),
then we can modify $I$, $\calV$, and $\calR$ so that (A3), (A4), and (A5) also hold as follows.

We first modify an augmenting path $\calR$ so that
it satisfies (A3) and (A4).
If $\calR$ violates (A3),
then $\calR$ is updated as the minimal suffix of $\calR$ satisfying (A2).
One can see that the resulting $\calR$ is an augmenting path for $(M, I, \calV, p)$ satisfying also (A3).
Suppose that $\calR$ violates (A4), i.e., there is an outer path $\calP_l$ with $\alpha(\calP_l) = \alpha^\sigma$ for some $l < m$ such that $\alpha^\sigma \in \calC(\alpha_*^{\sigma_{k+1}})$,
or $\calP_l$ is not proper and $\alpha^{\overline{\sigma}} \in \calC(\alpha_*^{\sigma_{k+1}})$.
Let $\calP_l$ be such an outer path with the minimum index $l$,
and $\beta^{\sigma'} \alpha^{\sigma}$ the last edge of $\calP_l$.
If $\alpha^\sigma \in \calC(\alpha_*^{\sigma_{k+1}})$,
then we update $\calR$ as
\begin{align*}
\calR \leftarrow \calP_0 \circ \calQ_1 \circ \calP_1 \circ \cdots \circ \calQ_i \circ \calP_l.
\end{align*}
If $\calP_l$ is not proper and $\alpha^{\overline{\sigma}} \in \calC(\alpha_*^{\sigma_{k+1}})$,
then 
we update $\calR$ as
\begin{align*}
\calR \leftarrow \calP_0 \circ \calQ_1 \circ \calP_1 \circ \cdots \circ \calQ_i \circ \calP_l',
\end{align*}
where $\calP_l'$ is the outer path obtained from $\calP_i$ by replacing the last edge $\beta^{\sigma'} \alpha^{\sigma}$ with $\beta^{\sigma'} \alpha^{\overline{\sigma}}$.
Clearly, the resulting $\calR$ is an augmenting path for $(M, I, \calV, p)$ satisfying (A3) and (A4),
and that $\theta$ strictly decreases.
Moreover, $p$ satisfies (Zero)$'$ for the resulting $\calR$,
since the last node of $\calR$ belongs to $\calC(\alpha_*^{\sigma_{k+1}})$.
Checking (A3) and (A4) and the update can be done in $O(|\calR|) = O(\min \set{\mu, \nu})$ time.
We let $\calR = \calP_0 \circ \calQ_1 \circ \calP_1 \circ \cdots \circ \calQ_m \circ \calP_m$ again by re-index.

We now assume that $\calR$ satisfies (A3) and (A4).
We then modify $(M, I)$ and $\calV$ so that $(M, I)$ is a matching-pair of size $k$,
$\calV$ is a valid labeling for $M$,
$p$ is an $(M, I, \calV)$-compatible $c$-potential,
and
$\calR$ is an augmenting path for $(M, I, \calV, p)$
satisfying (A3)--(A5).
By $\alpha_{k+1}^{\sigma_{k+1}} \in \calC(\alpha_*^{\sigma_{k+1}})$,
there is a $\sigma_{k+1}$-path $\calP$ in $\calG(\calV, p)|_M$ from $\alpha_{k+1}^{\sigma_{k+1}}$ to $\alpha_\ast^{\sigma_{k+1}}$.
This implies that every $\sigma_{k+1}$-edge $\alpha \beta$ in the underlying path $P$ of $\calP$
is double-tight.
By deleting all $\overline{\sigma_{k+1}}$-edges in $P$ from $M$
and adding all $\sigma_{k+1}$-edges in $P$ to $I$,
the resulting $(M, I)$ is a matching-pair of size $k$,
$\calV$ is a valid labeling for $M$,
$p$ is an $(M, I, \calV)$-compatible $c$-potential,
and $\deg_M(\alpha_{k+1}) \leq 1$.
Moreover, if $\deg_M(\alpha_{k+1}) = 1$,
then $\alpha_{k+1}$ is incident to a $\sigma_{k+1}$-edge,
i.e., (A5) holds.

We see that $\calR$ forms a pseudo augmenting path for $(M, I, \calV, p)$.
Since
this modification does not change $\calS(M, I, \calV, p)$
and does not increase $\calT(M, I, \calV, p)$,
and $\alpha_{k+1}^{\sigma_{k+1}}$ still belongs to $\calT(M, I, \calV, p)$,
$\calR$ satisfies (A2)--(A4).
If $\calR$ does not meet any newly appearing $\sigma_{k+1}$-edges in $I$,
then $\calR$ clearly satisfies (A1), as required.

Otherwise,
there is an outer path $\calP_i$
that meets a newly appearing $\sigma_{k+1}$-edge $\alpha \beta$ in $I$.
Since $\alpha^{\sigma_{k+1}}$ belongs to $\calT(M, I, \calV, p)$ before the modification,
the last edge of $\calP_i$ must be of the form ${\beta'}^\sigma \alpha^{\overline{\sigma_{k+1}}}$
and $\calP_i$ is proper, i.e., there is no edge between ${\beta'}^\sigma$ and $\alpha^{\overline{\sigma_{k+1}}}$ in $\calE(\calV, p)$
by (A4).
Hence $\calP_i$ satisfies (O2)$'$,
which implies that $\calR$ is a pseudo augmenting path for $(M, I, \calV, p)$.
We update $\calV$ as
\begin{align*}
\calV \leftarrow \calV(\calR).
\end{align*}
Then it follows from \cref{prop:pseudo augmenting path}
that $\calR$ is an augmenting path for $(M, I, \calV, p)$.
By \cref{lem:pseudo outer path},
$\calG(\calV, p)$ does not change,
which implies that $\calR$ still satisfies (A3)--(A5).

By this update, $\alpha_*^{\sigma_{k+1}}$ is deleted from $\calU(M, I)$
and $\alpha_{k+1}^{\sigma_{k+1}}$ is added to $\calU(M, I)$.
Since $\alpha_{k+1}^{\sigma_{k+1}}$ is the last node of $\calR$,
$p$ satisfies (Zero)$'$ for the resulting $\calR$.

Checking (A5) and the update can be done in $O(\min \set{\mu, \nu})$ time.
In particular, the update requires the front-propagation on $\calR$ which takes $O(|\calR|) = O(\min \set{\mu, \nu})$ time.
Clearly this update does not change $\phi$.
Furthermore $\theta$ does not increase.
Indeed,
$N_{\rm S}$ decreases by $|P|$.
If the underlying path of an inner path in the previous $\calR$ meets $P$,
then the corresponding inner path forms a part of an outer path in the resulting $\calR$.
Hence the number of edges in the union of outer paths in the resulting $\calR$ can increase,
but its increment is bounded by $|P|$.

In \cref{sec:base,,sec:violate,,sec:non violate},
we require that the augmenting path $\calR$ satisfies (A3)--(A5).

\subsection{Simplification of the last outer path}\label{subsec:preprocessing}
Suppose that $\calP_m[\beta_{i^*}^{\sigma_{i^*}})$ is simple for some $i^*$
and that every $\beta_i^{\overline{\sigma_i}}$ with $i > i^*$ belonging to $\calR$ is BP-invariant with respect to $\calP_m$.
Since $\calP_m[\beta_{i^*}^{\sigma_{i^*}})$ is simple,
$(\beta_{i^*} \alpha_{i^*+1}, \alpha_{i^*+1} \beta_{i^*+1}, \dots, \beta_k \alpha_{k+1})$ forms a path in $G$.
Therefore we can redefine $+$- and $-$-edges of $I$
and $+$- and $-$-spaces of $\calV$
so that
all $\alpha_{i^*+1} \beta_{i^*+1}, \alpha_{i^*+2} \beta_{i^*+2}, \dots, \alpha_k \beta_k$
are $+$-edges
and
\begin{align*}
\calP_m = (\beta_0^{\sigma_0} \alpha_1^{\sigma_1}, \dots, \beta_{i^*}^{\sigma_{i^*}} \alpha_{i^*+1}^+, \alpha_{i^*+1}^+ \beta_{i^*+1}^+, \dots, \beta_k^+ \alpha_{k+1}^+).
\end{align*}
The condition (A4) verifies that
$\deg_M(\alpha_{k+1}) \leq 1$ and
the edge incident to $\alpha_{k+1}$ in $M$ is a $+$-edge if $\deg_M(\alpha_{k+1}) = 1$.

We define $\hat{M}$, $\hat{I}$, $\hat{\calV}$, and $\hat{\calR}$ by
\begin{align*}
\hat{M} &\defeq M \cup P_m[\alpha_{i^*+1}),\\
\hat{I} &\defeq I \setminus \set{\alpha_{i^*+1} \beta_{i^*+1}, \alpha_{i^*+2} \beta_{i^*+2}, \dots, \alpha_k \beta_k},\\
\hat{\calV} &\defeq \calV(\calP_m[\alpha_{i^*+1}^+)^{-1}),\\
\hat{\calR} &\defeq \calP_0 \circ \calQ_1 \circ \calP_1 \circ \cdots \circ \calQ_m \circ \calP_m(\alpha_{i^*+1}^+];
\end{align*}
see Figure~\ref{fig:preprocessing}.
\begin{figure}
	\centering
	\begin{tikzpicture}[
	node/.style={
		fill=black, circle, minimum height=5pt, inner sep=0pt,
	},
	M/.style={
		line width = 3pt
	},
	AM/.style={
		line width = 3pt,
		red
	},
	DI/.style={
		line width = 3pt,
		blue
	},
	P/.style={
		very thick,
	},
	PD/.style={
		very thick, dashed,
	},
	calE/.style={
		thick
	}
	]
	
	\def\mu{a1, a2, a3, a4, a5, a6, a7}
	\def\nu{b0, b1, b2, b3, b4, b5, b6, b7}
	\def\size{1.7cm}
	\def\hight{4cm}
	\def\side{5pt}
	\def\up{12pt}

	\nodecounter{\nu}
	\coordinate (pos);
	\foreach \currentnode in \nu {
		\node[node, below=0 of pos, anchor=center] (\currentnode) {};
		\coordinate (pos) at ($(pos)+(-\size, 0)$);
	}
	\nodecounter{\mu}
	\coordinate (pos) at ($(b0) + (-\size, -\size)$);
	\foreach \currentnode in \mu {
		\node[node, below=0 of pos, anchor=center] (\currentnode) {};
		\coordinate (pos) at ($(pos)+(-\size, 0)$);
	}
	
	\foreach \i in {1, 2, 3, 7} {
		\draw[M] (b\i) -- (a\i);
	}
	
	\foreach \i in {4, 5, 6} {
		\draw[DI] (b\i) -- (a\i);
	}
	
	\foreach \i in \nu {
		\coordinate [left = \side] (l\i) at (\i);
		\coordinate [right = \side] (r\i) at (\i);
	}
	
	\foreach \i in \mu {
		\coordinate [left = \side] (l\i) at (\i);
		\coordinate [right = \side] (r\i) at (\i);
	}
	
	\draw (lb0) node[above left = 0 and -8pt] {$\sigma_0$};
	\draw (rb0) node[above right = 0 and -8pt] {$\overline{\sigma_0}$};
	
	\draw (lb3) node[above left = 0 and -9pt] {$\sigma_{i^*}$};
	\draw (rb3) node[above right = 0 and -9pt] {$\overline{\sigma_{i^*}}$};
	
	\draw[P, -{Latex[length=3mm]}] (lb0) -- node[above left = 0 and -6pt]{$\calP_m$} ($(lb0)!0.6!(la1)$);
	\draw[P] ($(lb0)!0.4!(la1)$) -- (la1);
	\foreach \i/\j in {1/1, 1/2, 2/2, 2/3, 3/3} {
		\draw[P] (lb\i) -- (la\j);
	}
	\draw[P, -{Latex[length=3mm]}] (lb3) -- ($(lb3)!0.6!(la4)$);
	\draw[P] ($(lb3)!0.4!(la4)$) -- (la4);
	\foreach \i/\j in {4/4, 4/5, 5/5, 5/6, 6/6, 6/7} {
		\draw[PD] (lb\i) -- (la\j);
	}
	
	\foreach \i in {1, 2, 3, 4, 5, 6, 7} {
		\draw [calE] (rb\i) -- (ra\i);
	}
	
	\foreach \i/\j in {4/5, 5/6, 6/7} {
		\draw[AM] (b\i) -- (a\j);
	}

	\coordinate [above = \up] (b0n) at (b0);
	\draw (b0n) node[above] {$\beta_0$};
	\coordinate [above = \up] (b3n) at (b3);
	\draw (b3n) node[above] {$\beta_{i^*}$};
	\coordinate [below = \up] (a4n) at (a4);
	\draw (a4n) node[below] {$\alpha_{i^*+1}$};
	\coordinate [below = \up] (a7n) at (a7);
	\draw (a7n) node[below] {$\alpha_{k+1}$};

	\coordinate (Pls) at ($(ra5)+(0, -1cm)$);
	\coordinate (Plt) at ($(rb6)+(0, 1cm)$);
	\draw [P, -{Latex[length=3mm]}] (Pls) -- (ra5) -- (rb5) -- (ra6) -- (rb6) -- (Plt);
    \coordinate [label=below:{$\calP_l$}] () at (Pls);

	\foreach \i in {4, 5, 6, 7} {
		\draw (lb\i) node[above] {$+$};
		\draw (rb\i) node[above=.3333em-1pt, fill=white, inner sep=1pt] {$-$};
		\draw (la\i) node[below] {$+$};
		\draw (ra\i) node[below=.3333em-1pt, fill=white, inner sep=1pt] {$-$};
	}
	
	\coordinate[label=left:$\cdots$] () at ($(b7)!0.5!(a7) + (-0.2cm, 0)$);
	
	\end{tikzpicture}
	\caption{
		Modification in \cref{subsec:preprocessing}.
		The thick and thin lines represent edges in $M$ and in $\calE(\calV, p)$,
		respectively.
		The red and blue lines represent the edges added to $M$, i.e., the edges in $\hat{M} \setminus M$,
		and the edges deleted from $I$, i.e., the edges in $I \setminus \hat{I}$,
		by the modification,
		respectively.
		The thin solid/dashed paths represent subpaths of $\calR$.
		In particular,
		the dashed thin path represents a part of $\calR \setminus \hat{\calR}$.
	}
	\label{fig:preprocessing}
\end{figure}
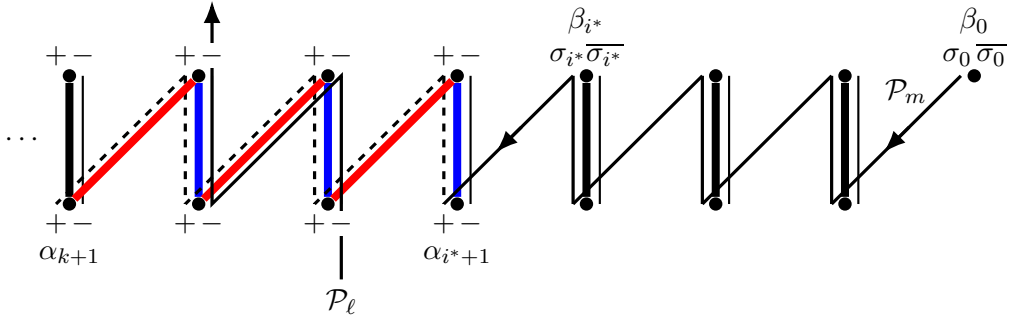
Then the following holds.
\begin{proposition}\label{prop:Nouter}
	$(\hat{M}, \hat{I})$ is a matching-pair of size $k$ such that $\calU(\hat{M}, \hat{I}) = \calU(M, I) \setminus \{ \alpha_{k+1}^+\} \cup \{ \alpha_{i^*+1}^+ \}$, $\hat{\calV}$ is a valid labeling for $\hat{M}$, $p$ is an $(\hat{M}, \hat{I}, \hat{\calV})$-compatible $c$-potential,
	and $\hat{\calR}$ is a pseudo augmenting path for $(\hat{M}, \hat{I}, \hat{\calV}, p)$
	satisfying {\rm (A3)}--{\rm (A5)}.
\end{proposition}
\begin{proof}
	If $\calP_m$ is of the form $(\beta_0^{\sigma_0} \alpha_1^{\sigma_1}, \dots, \beta_{i^*}^{\sigma_{i^*}} \alpha_{i^*+1}^+)$,
	then $(\hat{M}, \hat{I}) = (M, I)$, $\hat{\calV} = \calV$,
	and $\hat{\calR} = \calR$;
	the statement clearly holds.
	In the following, we assume $\calP_m = (\beta_0^{\sigma_0} \alpha_1^{\sigma_1}, \dots, \beta_{i^*}^{\sigma_{i^*}} \alpha_{i^*+1}^+, \alpha_{i^*+1}^+ \beta_{i^*+1}^+, \dots, \beta_k^+ \alpha_{k+1}^+)$ for some $k > i^*$.
	
	We first show that $(\hat{M}, \hat{I})$ is a matching-pair of size $k$ and $\hat{\calV}$ is a valid labeling for $\hat{M}$.
	Clearly, $\hat{M}$ satisfies (Deg) and (Cycle).
	Since $\calP_m[\beta_{i^*}^{\sigma_{i^*}})$ is a simple outer path,
	$\hat{\calV}$ is a valid labeling for $M$ by \cref{lem:simple outer path}.
	Let $(X_{i^*+1}, Y_{i^*+1}, X_{i^*+2}, \dots, X_{k+1})$ be the back-propagation of $\calP_m[\alpha_{i^*+1}^+)$.
	Then, for $i \geq i^*$,
	$X_i$ is the $+$-space of $\alpha_i$
	and $Y_i$ is the $-$-space of $\beta_i$
	with respect to $\hat{\calV}$.
	By the definition of the back-propagation,
	$A_{\alpha_{i+1} \beta_i}(X_{i+1}, Y_i) = \set{0}$ and $A_{\alpha_i \beta_i}(X_i, Y_i) = \set{0}$
	for each $i \geq i^*$.
	Thus $\hat{\calV}$ is also a valid labeling for $\hat{M}$,
	which implies that $\hat{M}$ is a pseudo-matching.
	Since $\hat{I}$ consists of isolated rank-2 edges in $\hat{M}$,
	$(\hat{M}, \hat{I})$ is a matching-pair.
	Since
	the number of edges in $\hat{M}$ increases by $k-i^*$
	and that of edges in $\hat{I}$ decreases by $k-i^*$,
	its size is equal to $k$.
	Since $\alpha_{i^*+1} \beta_{i^*+1}$ is a $+$-edge in $\hat{M}$,
	$U_{\alpha_{i^*+1}}^+$ is unmatched by $(\hat{M}, \hat{I})$.
	Thus we obtain $\calU(\hat{M}, \hat{I}) = \calU(M, I) \setminus \{\alpha_{k+1}^+\} \cup \{ \alpha_{i^*+1}^+ \}$.

    We then show that $p$ is an $(\hat{M}, \hat{I}, \hat{\calV})$-compatible potential.
	By \cref{lem:simple outer path},
	$\calV$ is equivalent to $\hat{\calV}$ with respect to $p$,
	and $\calP_m$ is an outer path for $(M, I, \hat{\calV}, p)$.
	By the former and \cref{lem:>}~(1) and~(2), $p$ satisfies (Reg) for $\hat{\calV}$
	and the edges $\alpha_{i^*+1}^- \beta_{i^*+1}^-, \dots, \alpha_k^- \beta_k^-$ exist in $\calE(\hat{\calV}, p)$,
	where $\alpha_{i^*+1} \beta_{i^*+1}, \alpha_{i^*+2} \beta_{i^*+2}, \dots, \alpha_k \beta_k$ are $+$-edges in $\hat{M}$.
	It follows from the latter that
	the edges $\beta_{i^*+1}^+ \alpha_{i^*+2}^+, \beta_{i^*+2}^+ \alpha_{i^*+3}^+, \dots, \beta_k^+ \alpha_{k+1}^+$ exist in $\calE(\hat{\calV}, p)$,
	where $\beta_{i^*+1} \alpha_{i^*+2}, \beta_{i^*+2} \alpha_{i^*+3}, \dots, \beta_k \alpha_{k+1}$ are $-$-edges in $\hat{M}$.
	Thus $p$ satisfies (Tight) for $(\hat{M}, \hat{I}, \hat{\calV})$,
	which implies that $p$ is $(\hat{M}, \hat{I}, \hat{\calV})$-compatible.
	
	We finally show that $\hat{\calR}$ is a pseudo augmenting path for $(\hat{M}, \hat{I} ,\hat{\calV}, p)$ satisfying (A3)--(A5).
	It
	clearly holds that $\calS(M, I, \calV, p) = \calS(\hat{M}, \hat{I}, \hat{\calV}, p)$.
	Thus $\calR$ satisfies (A3) and $\beta(\calP_0) \in \calS(\hat{M}, \hat{I}, \hat{\calV}, p)$.
	By $\alpha_{i^*+1}^+ \in \calU(\hat{M}, \hat{I})$, we have $\alpha_{i^*+1}^+ \in \calT(\hat{M}, \hat{I}, \hat{\calV}, p)$.
	Therefore $\hat{\calR}$ satisfies (A2).
	Since $\deg_{\hat{M}}(\alpha_{i^*+1}) = 1$ and $\alpha_{i^*+1} \beta_{i^*+1}$ is a $+$-edge,
	we have (A5).

	To prove (A1)$'$ and (A4),
	we show the following claim:
	\begin{claim*}
		$\hat{\calR}$ is a path in $\calG(\hat{\calV}, p)$.
		Moreover,
		for any $\beta^\sigma \alpha^{\sigma'} \in \hat{\calR}$
		such that $\alpha$ is incident to an edge in $I$,
		there is no edge between
		$\beta^\sigma$ and $\alpha^{\overline{\sigma'}}$ in $\calE(\hat{\calV}, p)$.
	\end{claim*}
	\begin{proof}[Proof of Claim]
		The difference between $\hat{\calV}$ and $\calV$ can only be in the $+$-space of $\alpha_i$ and the $-$-space of $\beta_i$
		for $i^*+1 \leq i \leq k$.
		Since $\calP_m$ meet $\alpha_{i^*+1}^+ \beta_{i^*+1}^+, \dots, \alpha_k^+ \beta_k^+$ and $\calR$ is a path in $\calG(\calV, p)$,
		$\hat{\calR}$ does not have any of $\alpha_{i^*+1}^+ \beta_{i^*+1}^+, \dots, \alpha_k^+ \beta_k^+$.
		Therefore, for each $\beta^\sigma \alpha^{\sigma'} \in \hat{\calR}$, we obtain the following observations.
		\begin{itemize}
			\item
			The $\sigma'$-space of $\alpha$ with respect to $\hat{\calV}$ is $U_{\alpha}^{\sigma'}$.
			\item
			If $\beta$ is different from $\beta_{i^*+1}, \dots, \beta_k$,
			then the $\sigma$-space of $\beta$ with respect to $\hat{\calV}$ is $V_{\beta}^{\sigma}$.
			\item
			If $\beta$ is one of $\beta_{i^*+1}, \dots, \beta_k$, say, $\beta_i$,
			then $\beta^{\sigma} = \beta_i^-$, and
			either
			$p(V_{\beta_i}^-) > p(V_{\beta_i}^+)$ or
			the $-$-space of $\beta_i$ with respect to $\hat{\calV}$ is $V_{\beta_i}^-$.
		\end{itemize}
		In particular,
		the third follows from the BP-invariance of $\beta_{i}^-$.
		Thus every $\beta^\sigma \alpha^{\sigma'} \in \hat{\calR}$ belongs to $\calE(\hat{\calV}, p)$
		by \cref{lem:>}~(3) and $\beta^\sigma \alpha^{\sigma'} \in \calE(\calV, p)$.
		This implies that $\hat{\calR}$ is a path in $\calG(\hat{\calV}, p)$.
		If $\alpha$ is incident to an edge in $\hat{I}$,
		then $\alpha$ is different from $\alpha_{i^*+1}, \dots, \alpha_k$.
		Hence the $\overline{\sigma'}$-space of $\alpha$ with respect to $\hat{\calV}$ is equal to $U_{\alpha}^{\overline{\sigma'}}$.
		If $\alpha$ is incident to an edge in $I \setminus \hat{I}$,
		then $\alpha$ is one of $\alpha_{i^*+1}, \dots, \alpha_k$, say, $\alpha_i$.
		In this case, $\alpha^{\sigma'} = \alpha_i^-$ and either $p(U_{\alpha_i}^+) > p(V_{\beta_i}^-)$ or the $+$-space of $\alpha_i$ with respect to $\hat{\calV}$ is $U_{\alpha_i}^+$;
		the latter follows from the BP-invariance of $\beta_i^-$.
		By \cref{lem:>}~(3) and $\beta^\sigma \alpha^{\sigma'} \notin \calE(\calV, p)$,
		we obtain $\beta^{\sigma} \alpha^{\overline{\sigma'}} \notin \calE(\hat{\calV}, p)$.
	\end{proof}
	
	Claim immediately verifies that every $\calQ_l$ in $\calR$ forms an inner path for $(\hat{I}, \hat{\calV}, p)$
	and every $\calP_l$ intersecting with none of $\alpha_{i^*+1}^- \beta_{i^*+1}^-, \dots, \alpha_k^- \beta_k^-$
	forms an outer path for $(\hat{I}, \hat{\calV}, p)$.
	
	Suppose that $\calP_l$ meets some of $\alpha_{i^*+1}^- \beta_{i^*+1}^-, \dots, \alpha_k^- \beta_k^-$.
	Then we see that $\calP_l \circ \calQ_{l+1}$ is representable as the concatenation of several pseudo outer paths and inner paths for $(\hat{I}, \hat{\calV}, p)$,
	which implies that $\hat{\calR}$ satisfies (A1)$'$.
	By Claim, all edges in $\calP_l$ belong to $\calE(\hat{\calV}, p)$.
	Let $\calP^- \defeq (\alpha_{i^*+1}^- \beta_{i^*+1}^-, \beta_{i^*+1}^- \alpha_{i^*+2}^-, \dots, \beta_k^- \alpha_{k+1}^-)$.
	The intersection of $\calP_l$ and $\calP^-$
	is the disjoint union of the subpath of $\calP^-$
	having the form $\calP_l[ \alpha_i^-, \beta_j^- ]$ with $i^*+1 \leq i \leq j \leq k$ or the form $\calP_l[\alpha_i^-)$ with $i^*+1 \leq i \leq k$.
	That is,
	$\calP_l$ can be represented as
	\begin{align*}
	\calP_l(\alpha_{i_1}^-] \circ \calP^-[ \alpha_{i_1}^-, \beta_{i_2}^- ] \circ \calP_l[ \beta_{i_2}^-, \alpha_{i_3}^- ]
	\circ \cdots \circ \calP^-[ \alpha_{i_p}^-, \beta_{i_{p+1}}^- ] \circ \calP_l[ \beta_{i_{p+1}}^- )
	\end{align*}
	or
	\begin{align*}
	\calP_l(\alpha_{i_1}^-] \circ \calP^-[ \alpha_{i_1}^-, \beta_{i_2}^- ] \circ \calP_l[ \beta_{i_2}^-, \alpha_{i_3}^- ]
	\circ \cdots \circ \calP_l[ \beta_{i_p}^-, \alpha_{i_{p+1}}^- ] \circ \calP^-[ \alpha_{i_{p+1}}^- ),
	\end{align*}
	where $i^*+1 \leq i_q \leq i_{q+1} \leq k$ for each $q$.
	
	Since $\alpha_i \beta_i$ is a $+$-edge in $\hat{M}$ for $i^*+1 \leq i \leq k$,
	the subpath $\calP_l[ \alpha_{i_q}^-, \beta_{i_q}^- ]$ forms an inner path for $(\hat{M}, \hat{I}, \hat{\calV}, p)$.
	Moreover, $\calP_l[\alpha_{i_{p+1}}^-) \circ \calQ_{l+1}$ is also an inner path for $(\hat{M}, \hat{I}, \hat{\calV}, p)$ in the latter case.
	By Claim,
	if $\calP_l$ has an edge $\beta^\sigma \alpha^{\sigma'}$ such that $\alpha$ is incident to an edge in $\hat{I}$,
	there is no edge between $\beta^\sigma$ and $\alpha^{\overline{\sigma'}}$ in $\calE(\hat{\calV}, p)$.
	This implies that
	the remaining parts of $\calP_l$,
	which are
	$\calP_l(\alpha_{i_1}^-]$, $\calP_l[ \beta_{i_q}^-, \alpha_{i_{q+1}}^- ]$, and $\calP_l[ \beta_{i_{p+1}}^- )$,
	are pseudo outer paths for $(\hat{M}, \hat{I}, \hat{\calV}, p)$.
	Thus $\calP_l \circ \calQ_{l+1}$ can be viewed as the concatenation of several pseudo outer paths and inner paths for $(\hat{M}, \hat{I}, \hat{\calV}, p)$.
	
	We finally show that $\calP$ satisfies (A4).
	Since $\calC(\alpha_{i^*+1}^+)$ is the union of $\calC(\alpha_{k+1}^+)$ in $\calG(\calV, p)$ and $\calP_m[\alpha_{i^*+1}^+)$,
	it suffices to see that,
	if $\calP_l$ has $\beta'^\sigma \alpha_i^-, \alpha_i^- \beta_i^-$ for some $i^*+1 \leq i \leq k$,
	then there is no edge between $\beta'^\sigma$ and $\alpha_i^+$ in $\calG(\hat{\calV}, p)$;
	this follows from Claim.
	Hence $\calP$ satisfies (A4).
	
	This completes the proof.
\end{proof}

By this update, $\theta$ does not increase and $\phi$ strictly decreases.
\begin{lemma}\label{lem:preprocessing theta varphi}
	$\theta(\hat{M}, \hat{I}, \hat{\calR}) \leq \theta(M, I, \calR)$ and $\phi(\hat{M}, \hat{I}, \hat{\calR}) < \phi(M, I, \calV)$.
\end{lemma}
\begin{proof}
	The difference of $\theta$ follows from the fact that
    $N_{\rm S}$ increases by $|\calP_m[\alpha_{i^*+1}^+)|$
	and the number of edges in the union of outer paths decreases at least by $|\calP_m[\alpha_{i^*+1}^+])|$.
 	Thus we have $\theta(\hat{M}, \hat{I}, \hat{\calR}) \leq \theta(M, I, \calR)$.
	Clearly $N_0$ does not change and $|\hat{\calR}| = |\calR| - |\calP_m[\alpha_{i^*+1}^+)|$,
	which implies $\phi(\hat{M}, \hat{I}, \hat{\calR}) = \phi(M, I, \calV) - |\calP_m[ \alpha_{i^*+1} )| < \phi(M, I, \calV)$.
\end{proof}

We update
\begin{align*}
M \leftarrow \hat{M}, \qquad
I \leftarrow \hat{I}, \qquad
\calV \leftarrow \hat{\calV}(\calR).
\end{align*}
This update can be done in $O(|\calR|) = O(\min \set{\mu, \nu})$ time.
By \cref{prop:pseudo augmenting path,prop:Nouter},
the resulting $(M, I)$ is a matching-pair of size $k$, $\calV$ a valid labeling for $M$,
$p$ an $(M, I, \calV)$-compatible $c$-potential,
and
$\calR$ an augmenting path for $(M, I, \calV, p)$ satisfying (A3)--(A5).
Moreover,
$p$ satisfies (Zero)$'$ for $\calR$ by $\alpha_{i^*+1}^+ \in \calU(M, I) \not\ni \alpha_{k+1}^+$ (in \cref{prop:Nouter}).

In \cref{sec:base,,sec:violate,,sec:non violate},
we appropriately execute this modification of the last outer path as the preprocessing to simplify the update in each phase.

\section{Base case: $\calR = \calP_0$ and $\calP_0$ is simple}\label{sec:base}
Suppose that $\calR$ consists only of a single outer path $\calP_0$.
Then $\calR$ clearly satisfies ($\No$).
By applying the simplification in \cref{subsec:preprocessing} to $\calP_0$,
we can assume that $\calP_0$ consists of a single edge in $\calE(\calV, p)$.
We first redefine $+$- and $-$-edges of $M$
and $+$- and $-$-spaces of $\calV$
so that
\begin{align*}
\calP_0 = (\beta_0^+ \alpha_1^+).
\end{align*}
By (A5), $\deg_M(\alpha_1) \leq 1$
and
an edge incident to $\alpha_{k+1}$ in $M$ is a $+$-edge if it exists.
Moreover,
we modify $(M, I)$ so that $\deg_M(\beta_0) \leq 1$
and an edge incident to $\beta_0$ in $M$ is a $+$-edge if it exists;
it can be done by the same procedure as in \cref{subsec:modifications} for (A5).
By this modification and (Zero)$'$,
we have
\begin{align}\label{eq:except}
    p(U_\alpha^\sigma) = p(V_\beta^{\sigma'}) = 0
\end{align}
for any unmatched spaces $U_\alpha^\sigma, V_\beta^{\sigma'}$ by $(M, I)$ except $U_{\alpha_1}^+$ and $V_{\beta_0}^+$.

Let $C_{\beta_0}$ be the connected components of $M \setminus I$
containing $\beta_0$,
and $P_{\beta_0}$ be the maximal rank-2 path in $C_{\beta_0}$ 
which start with $\beta_0$.
We let $P_{\beta_0} \defeq \{ \beta_0 \}$ 
if $\beta_0$ is incident to no edge or a rank-1 edge.
For a vector space $Y \in \calM_{\beta_{0}}$ with $Y \neq V_{\beta_0}^+$,
let $\calV(C_{\beta_0}; Y)$ be the labeling obtained from $\calV$
by replacing $V_{\beta_0}^-$
with $Y$
and by setting the $-$-space of $\beta$
and the $+$-space of $\alpha$ for $\alpha, \beta$ belonging to $C_{\beta_0}$
so that \eqref{eq:+-} and \eqref{eq:++ --} hold.
Clearly we have $\calV(C_{\beta_0}; V_{\beta_0}^-) = \calV$.
One can observe that $\calV$ and $\calV(C_{\beta_0}; Y)$ can be different only in the $-$-space of $\beta$
and the $+$-space of $\alpha$ for $\alpha, \beta \in V(P_{\beta_0})$,
and that $\calV(C_{\beta_0}; Y)$ is a valid labeling for $M$.
For $\alpha_{1}$,
we similarly define $C_{\alpha_{1}}$, $P_{\alpha_{1}}$,
and $\calV(C_{\alpha_{1}}; X)$ for $X \in \calM_{\alpha_{1}}$ with $X \neq U_{\alpha_1}^+$.
\begin{lemma}\label{lem:single-counted component}
	\begin{itemize}
		\item[{\rm (1)}]
		If $p(V_{\beta_0}^-) \geq p(V_{\beta_0}^+)$ (resp. $p(V_{\beta_0}^-) > p(V_{\beta_0}^+)$),
		then we have $p(V_{\beta}^-) \geq p(V_{\beta}^+)$ and $p(U_{\alpha}^+) \geq p(U_{\alpha}^-)$
		(resp. $p(V_{\beta}^-) > p(V_{\beta}^+)$ and $p(U_{\alpha}^+) > p(U_{\alpha}^-)$)
		for each $\alpha, \beta \in V(P_{\beta_0})$.
		Also if $p(U_{\alpha_1}^-) \geq p(U_{\alpha_1}^+)$ (resp. $p(U_{\alpha_1}^-) > p(U_{\alpha_1}^+)$),
		then we have $p(U_{\alpha}^-) \geq p(U_{\alpha}^+)$ and $p(V_{\beta}^+) \geq p(V_{\beta}^-)$ (resp. $p(U_{\alpha}^-) > p(U_{\alpha}^+)$ and $p(V_{\beta}^+) > p(V_{\beta}^-)$)
		for each $\alpha, \beta \in V(P_{\alpha_1})$.
		\item[{\rm (2)}]
	If $p(V_{\beta_0}^-) \geq p(V_{\beta_0}^+)$,
	then $\calV(C_{\beta_0}; Y) \simeq_p \calV$
	for each $Y \in \calM_{\beta_{0}}$ with $Y \neq V_{\beta_0}^+$.
	Also if $p(U_{\alpha_{1}}^-) \geq p(U_{\alpha_{1}}^+)$,
	then $\calV(C_{\alpha_{1}}; X) \simeq_p \calV$
	for each $X \in \calM_{\alpha_{1}}$ with $X \neq U_{\alpha_{1}}^-$.
	\end{itemize}
\end{lemma}
\begin{proof}
	We only show the case for $\beta_0$; the case for $\alpha_1$ is similar.
	
	(1).
	If $P_{\beta_0} = \{\beta_0\}$,
	we are done.
	Suppose that $P_{\beta_0}$ has of the form $(\beta_{0} \alpha_{0}, \alpha_{0} \beta_{-1}, \dots)$.
	Since $\beta_{0} \alpha_{0}$ is rank-2,
	we have $A_{\alpha_{0} \beta_{0}}(U_{\alpha_{0}}^+, V_{\beta_{0}}^+) \neq \set{0}$.
	Since $p$ is a $c$-potential satisfying (Tight) and $\beta_{0} \alpha_{0}$ is a $+$-edge,
	we have $p(U_{\alpha_{0}}^+) + p(V_{\beta_{0}}^+) + c \geq d_{\alpha_{0} \beta_{0}} = p(U_{\alpha_{0}}^-) + p(V_{\beta_{0}}^-) + c$.
	By combining it with $p(V_{\beta_0}^-) \geq p(V_{\beta_0}^+)$,
	we have $p(U_{\alpha_{0}}^+) \geq p(U_{\alpha_{0}}^-)$.
	By a similar argument,
	we obtain $p(V_\beta^-) \geq p(V_\beta^+)$ and $p(U_\alpha^+) \geq p(U_\alpha^-)$ for each $\alpha, \beta \in V(P_{\beta_{0}})$.
	By replacing $\geq$ with $>$ in the above argument,
	we obtain $p(V_\beta^-) > p(V_\beta^+)$ and $p(U_\alpha^+) > p(U_\alpha^-)$ for each $\alpha, \beta \in V(P_{\beta_{0}})$.
	
	(2).
	Recall that $\calV$ and $\calV(C_{\beta_0}; Y)$ can be different only in the $-$-space of $\beta$
	and the $+$-space of $\alpha$ for $\alpha, \beta \in V(P_{\beta_0})$.
	Suppose $p(V_{\beta_0}^-) \geq p(V_{\beta_0}^+)$.
	Then, by (1),
	$p(V_{\beta}^-) \geq p(V_{\beta}^+)$ and $p(U_{\alpha}^+) \geq p(U_{\alpha}^-)$ hold
	for each $\alpha, \beta \in V(P_{\beta_0})$.
	Since $\calV(C_{\beta_0}; Y)$ is a valid labeling for $M$,
	for each $\beta \in V(P_{\beta_0})$
	the $-$-space of $\beta$ with respect to $\calV(C_{\beta_0}; Y)$, denoted by $Y'$,
	is different from the $+$-space of $\beta$ with respect to $\calV(C_{\beta_0}; Y)$,
	which is $V_{\beta}^+$.
	Thus, by (Reg),
	we have $p(Y') = p(V_\beta^-) \geq p(V_{\beta}^+)$.
	Similarly,
	for $\alpha \in V(P_{\beta_0})$,
	we have $p(X') = p(U_\alpha^+) \geq p(U_\alpha^-)$,
	where $X'$ is the $+$-space of $\alpha$ with respect to $\calV(C_{\beta_0}; Y)$.
	Thus $\calV(C_{\beta_0}; Y) \simeq_p \calV$.
\end{proof}

Let us define $M^*$ and $\calV^*$ by
\begin{align*}
M^* \defeq M \cup \{ \beta_0 \alpha_1\}, \qquad \calV^* \defeq \calV \left(C_{\alpha_1}; (V_{\beta_0}^+)^{\perp_{\beta_0 \alpha_1}}\right)\left(C_{\beta_0}; (U_{\alpha_1}^+)^{\perp_{\alpha_1 \beta_0}}\right).
\end{align*}
Then the following holds:
\begin{proposition}\label{prop:base case}
	$(M^*, I)$ is a matching-pair of size $k+1$, $\calV^*$ is a valid labeling for $M^*$,
	and
	$p$ is an optimal $(M^*, I, \calV^*)$-compatible $c$-potential.
\end{proposition}
\begin{proof}
	By $\deg_I(\beta_0) \leq 1$ and $\deg_I(\alpha_1) \leq 1$,
	the resulting $M^*$ satisfies (Deg).
	Let $C^*$ be the union of $C_{\beta_0}$, $C_{\alpha_1}$, and $\beta_0 \alpha_1$,
	which is a connected component of $M^* \setminus I$.
	Suppose, to the contrary,
	that $M^*$ does not satisfy (Cycle),
	or equivalently,
	$C^*$ is a cycle consisting of rank-2 edges.
	It can happen only when
	$C_{\beta_0}$ and $C_{\alpha_{1}}$ coincide,
	$C_{\beta_0}$ forms a rank-2 path in $M$,
	and $\beta_0 \alpha_1$ is rank-2.
	In this case,
	we have $A_{\alpha \beta}(U_{\alpha}^+, V_{\beta}^+) \neq \set{0} \neq A_{\alpha \beta}(U_{\alpha}^-, V_{\beta}^-)$ for each $\alpha \beta \in C_{\beta_0}$.

	For each $+$-edge $\alpha \beta \in C_{\beta_0}$,
	there is $\alpha^- \beta^- \in \calE(\calV, p)$.
	Hence we have $p(U_{\alpha}^-) + p(V_{\beta}^-) \leq p(U_{\alpha}^+) + p(V_{\beta}^+)$.
	In addition,
	since $C_{\beta_0}$ is not rearrangeable,
	$C_{\beta_0}$ has at least one single-tight $+$-edge.
	For such an edge,
	the above inequality is strict.
	Thus
	we have
	\begin{align}\label{eq:strict}
	\sum_{\alpha, \beta \in V(C_{\beta_0})} \left( p(U_{\alpha}^-) + p(V_{\beta}^-) \right) < \sum_{\alpha, \beta \in V(C_{\beta_0})} \left( p(U_{\alpha}^+) + p(V_{\beta}^+) \right).
	\end{align}
	Similarly, for each $-$-edge $\alpha \beta \in C_{\beta_0}$,
	there is $\alpha^+ \beta^+ \in \calE(\calV, p)$.
	Moreover, the edge $\beta_0^+ \alpha_1^+$ exists in $\calE(\calV, p)$.
	Hence
	\begin{align*}
	\sum_{\alpha, \beta \in V(C_{\beta_0})} \left( p(U_{\alpha}^-) + p(V_{\beta}^-) \right) \geq \sum_{\alpha, \beta \in V(C_{\beta_0})} \left( p(U_{\alpha}^+) + p(V_{\beta}^+) \right)
	\end{align*}
	also holds,
	which contradicts~\eqref{eq:strict}.
	
	We then show that $\calV^*$ is a valid labeling for $M^*$
	and $p$ is an $(M^*, I, \calV^*)$-compatible $c$-potential.
	In particular, the former implies that $(M^*, I)$ is a matching-pair.

	Suppose that $C^*$ is a path component of $M^* \setminus I$
	or $C^*$ is a cycle component such that $C_{\beta_0} (= C_{\alpha_{k+1}})$
	has a rank-1 edge.
	Then $P_{\beta_0}$ and $P_{\alpha_1}$ are disjoint,
	particularly,
	$P_{\beta_0}$ does not contain $\alpha_1$
	and $P_{\alpha_1}$ does not contain $\beta_0$.
	Hence the $+$-spaces of $\beta_0$ and $\alpha_1$ with respect to $\calV^*$
	are $V_{\beta_0}^+$ and $U_{\alpha_1}^+$,
	respectively.
	By \cref{lem:edge} and $\beta_0^+ \alpha_1^+ \in \calE(\calV, p)$,
	we have $(V_{\beta_0}^+)^{\perp_{\beta_0 \alpha_1}} \neq U_{\alpha_1}^+$
	and $(U_{\alpha_1}^+)^{\perp_{\alpha_1 \beta_0}} \neq V_{\beta_0}^+$.
	Hence $\calV^*$
	is a valid labeling for $M$.
	In addition, $p(U_{\alpha_1}^-) \geq p(U_{\alpha_1}^+)$
	if $(V_{\beta_0}^+)^{\perp_{\beta_0 \alpha_1}} \neq U_{\alpha_1}^-$,
	and $p(V_{\beta_0}^-) \geq p(V_{\beta_0}^+)$
	if $(U_{\alpha_1}^+)^{\perp_{\alpha_1 \beta_0}} \neq V_{\beta_0}^-$.
	Thus we have $\hat{\calV} \simeq_p \calV$ by \cref{lem:single-counted component}~(2).
	Therefore, by \cref{lem:>}~(1),
	$p$ is an $(M, I, \calV^*)$-compatible potential.
	In particular, $p$ satisfies (Reg) for $\calV^*$
	and (Tight) for all edges in $M$.

    The remaining is to show that the conditions~\eqref{eq:+-} and~\eqref{eq:++ --} hold for the edge $\beta_0 \alpha_1$ and that (Tight) holds for $\beta_0 \alpha_1$,
    which implies that $\calV^*$ is a valid labeling for $M^*$
    and $p$ is an $(M^*, I, \calV^*)$-compatible potential.
    The former follows from the fact that
    $(V_{\beta_0}^+)^{\perp_{\beta_0 \alpha_1}}$ and $(U_{\alpha_1}^+)^{\perp_{\alpha_1 \beta_0}}$ are the $-$-spaces of $\alpha_1$ and of $\beta_0$ with respect to $\calV^*$,
    respectively,
    and
    $A_{\alpha_1 \beta_0}((V_{\beta_0}^+)^{\perp_{\beta_0 \alpha_1}}, V_{\beta_0}^+) = A_{\alpha_1 \beta_0}(U_{\alpha_1}^+, (U_{\alpha_1}^+)^{\perp_{\alpha_1 \beta_0}}) = \set{0}$.
    Since the edge $\beta_0 \alpha_1$ is a $-$-edge in $M^*$
	and $\beta_0^+ \alpha_1^+$ exists in $\calE(\calV^*, p)$ by \cref{lem:>}~(3),
	the latter holds.

	Suppose that $C^*$ is a cycle component such that $C_{\beta_0}$
	consists of rank-2 edges.
	In this case, $P_{\beta_0}$ and $P_{\alpha_1}$ coincide;
	they are the same as $C_{\beta_0}$.
	Since $M^*$ satisfies (Cycle),
	$\beta_0 \alpha_1$ must be rank-1.
	
	If $p(U_{\alpha_1}^-) < p(U_{\alpha_1}^+)$,
	then $(V_{\beta_0}^+)^{\perp_{\beta_0 \alpha_1}} = U_{\alpha_1}^-$ by \cref{lem:edge}.
	Hence $\calV^* = \calV(C_{\beta_0}; (U_{\alpha_1}^+)^{\perp_{\alpha_1 \beta_0}})$,
	implying that the $+$-space of $\beta$ with respect to $\calV^*$ is $V_\beta^+$.
	\cref{lem:edge,lem:single-counted component}~(2) assert that
	$\calV^*$ is a valid labeling for $M$
	and $p$ is $(M, I, \calV^*)$-compatible.
	Furthermore,
	since $\beta_0 \alpha_1$ must be rank-1,
	we have $(V_{\beta_0}^+)^{\perp_{\beta_0 \alpha_1}} (= U_{\alpha_1}^-) = \kerL(A_{\alpha_1 \beta_0})$
	and $(U_{\alpha_1}^+)^{\perp_{\alpha_1 \beta_0}} = \kerR(A_{\alpha_1 \beta_0})$.
	Thus $A_{\alpha_1 \beta_0}((V_{\beta_0}^+)^{\perp_{\beta_0 \alpha_1}}, V_{\beta_0}^+) = A_{\alpha_1 \beta_0}(U_{\alpha_1}^+, (U_{\alpha_1}^+)^{\perp_{\alpha_1 \beta_0}}) = \set{0}$ holds,
	implying that $\calV^*$ is also a valid labeling for $M^*$.
	By \cref{lem:>}~(3),
	the edge $\beta_0^+ \alpha_1^+$ exists in $\calE(\calV^*, p)$.
	Therefore $p$ is an $(M^*, I, \calV^*)$-compatible potential.
	
	If $p(U_{\alpha_1}^-) \geq p(U_{\alpha_1}^+)$,
	then $p(V_{\beta_0}^-) \leq p(V_{\beta_0}^+)$ by \cref{lem:single-counted component}~(1).
	Moreover, since $C_{\beta_0}$ has a single-tight edge,
	the above inequality is strict, i.e., $p(V_{\beta_0}^-) < p(V_{\beta_0}^+)$.
	Thus $(U_{\alpha_1}^+)^{\perp_{\alpha_1 \beta_0}} = V_{\beta_0}^-$ by \cref{lem:edge}.
	By a similar argument to the case of $p(U_{\alpha_1}^-) < p(U_{\alpha_1}^+)$,
	$\calV^*$ is also a valid labeling for $M^*$
	and $p$ is an $(M^*, I, \calV^*)$-compatible potential.

    Clearly $(M^*, I)$ is a matching-pair of size $k+1$ such that $\calU(M^*, I) = \calU(M, I) \setminus \{ \beta_0^+, \alpha_1^+ \}$.
    By the condition~\eqref{eq:except},
    $p$ satisfies (Zero) for $(M^*, I, \calV^*)$.
    Thus, by \cref{lem:compatible},
    $p$ is optimal.
    
	This completes the proof.
\end{proof}

This update can be done in $O(|M|) = O(\min \set{\mu, \nu})$ time.
We output the resulting $(M^*, I, \calV^*, p)$ and stop the augmentation procedure.

\begin{remark}\label{rmk:rearrangeable}
    In the proof of \cref{prop:base case}, the condition that $C_{\beta_0}$ is not rearrangeable is used only when the union $C^*$ of $C_{\beta_0}$, $C_{\alpha_1}$, and $\beta_0 \alpha_1$ forms a cycle component in $M^*$.
    Hence, in the case where $C^*$ is a path component in $M^*$,
    then the statement of \cref{prop:base case} holds even if
    $C_{\beta_0}$ is rearrangeable.
    \qqed
\end{remark}

\begin{remark}\label{rmk:cycle}
    Let $C^*$ be the union of $C_{\beta_0}$, $C_{\alpha_1}$, and $\beta_0 \alpha_1$ as in \cref{rmk:rearrangeable}.
    The proof of \cref{prop:base case} implies that
    for each $\alpha \in V(C^*)$,
    the $+$-space or $-$-space of $\alpha$ with respect to $\calV^*$ is $U_\alpha^+$ or $U_\alpha^-$,
    respectively,
    even when $C^*$ forms a cycle.
    The same holds for $\beta \in V(C^*)$.
    \qqed
\end{remark}

\section{$\calR$ violates ($\No$) or ($\Ni$)}\label{sec:violate}
In this section, we consider the case where $\calR$ violates ($\No$) or ($\Ni$).
Let $\calR = \calP_0 \circ \calQ_1 \circ \calP_1 \circ \cdots \circ \calQ_m \circ \calP_m$ be an augmenting path for $(M, I, \calV, p)$,
in which $\calP_m = (\beta_0^{\sigma_0} \alpha_1^{\sigma_1}, \alpha_1^{\sigma_1} \beta_1^{\sigma_1}, \dots, \beta_k^{\sigma_k} \alpha_{k+1}^{\sigma_{k+1}})$.

\subsection{$\calR$ violates ($\No$)}\label{subsec:Nouter}
Suppose that $\calR$ violates ($\No$),
i.e.,
$\calP_m$ is not simple
or $\calR$ meets some $\beta_i^{\overline{\sigma_i}}$
that is not BP-invariant with respect to $\calP_m$.

Let $i^*$ be the minimum index such that
$P_m[\alpha_{i^*+1})$ forms a path in $G$.
That is, if $\calP_m$ is simple
then $i^* = 0$,
and
otherwise
$P_m$ is of the form
\begin{align*}
P_m = (\beta_0 \alpha_1, \alpha_1 \beta_1, \dots, \alpha_{i^*} \beta_{i^*}, \beta_{i^*} \alpha_{i^* +1}, \dots, \alpha_{j^*} \beta_{j^*}, \beta_{j^*} \alpha_{j^*+1}, \dots, \beta_k \alpha_{k+1}),
\end{align*}
where $\alpha_{i^*}\beta_{i^*} = \alpha_{j^*} \beta_{j^*}$
and $\beta_{i^*} \alpha_{i^*+1}, \alpha_{i^*+1} \beta_{i^*+1}, \dots, \beta_k \alpha_{k+1}$
are distinct.
Note that, in the latter case,
we have
$\alpha_{i^*}^{\sigma_{i^*}} \beta_{i^*}^{\sigma_{i^*}} = \alpha_{j^*}^{\overline{\sigma_{j^*}}} \beta_{j^*}^{\overline{\sigma_{j^*}}}$.

There are two cases:
(Case~1) $\calR$ meets some $\beta_i^{\overline{\sigma_i}}$ with $i > i^*$
that is not BP-invariant with respect to $\calP_m$
and (Case~2) all $\beta_i^{\overline{\sigma_i}}$ with $i > i^*$ which $\calR$ meets are BP-invariant with respect to $\calP_m$.

\paragraph{(Case~1) $\calR$ meets some $\beta_i^{\overline{\sigma_i}}$ with $i > i^*$ that is not BP-invariant with respect to $\calP_m$.}
Choose such $\beta_i^{\overline{\sigma_i}}$ with the maximum index $i$.
By $i > i^*$ and the choice of $\beta_i^{\overline{\sigma_i}}$,
$\calP_m[\beta_i^{\sigma_i})$ is a simple outer path and
every $\beta_j^{\overline{\sigma_j}}$ with $j > i$ belonging to $\calR$ is BP-invariant with respect to $\calP_m$.
Hence, by executing the simplification in \cref{subsec:preprocessing} to $\calP_m[\beta_i^{\sigma_i})$,
we can assume that
\begin{align*}
\calP_m = (\beta_0^{\sigma_0} \alpha_1^{\sigma_1}, \alpha_1^{\sigma_1} \beta_1^{\sigma_1}, \dots, \alpha_i^{\sigma_i} \beta_i^{\sigma_i}, \beta_i^{\sigma_i} \alpha_{i+1}^+).
\end{align*}

Since $\beta_i^{\overline{\sigma_i}}$ is not BP-invariant
and the space at $\beta_i^{\sigma_i}$ of the back-propagation of $\calP_m$
is $(U_{\alpha_{i+1}}^+)^{\perp_{\alpha_{i+1} \beta_i}}$,
we have
$p(V_{\beta_i}^+) = p(V_{\beta_i}^-)$ and $V_{\beta_i}^{\overline{\sigma_i}}$
is different from $(U_{\alpha_{i+1}}^+)^{\perp_{\alpha_{i+1} \beta_i}}$.
This implies $p(U_{\alpha_{i+1}}^+) + p(V_{\beta_i}^{\overline{\sigma_i}}) + c = d_{\alpha_{i+1} \beta_i}$
and $A_{\alpha_{i+1} \beta_i}(U_{\alpha_{i+1}}^+, V_{\beta_i}^{\overline{\sigma_i}}) \neq \set{0}$.
Thus the edge $\beta_i^{\overline{\sigma_i}} \alpha_{i+1}^+$ exists in $\calE(\calV, p)$.

Let $\calP_l$ be the outer path in $\calR$ that meets $\beta_i^{\overline{\sigma_i}}$.
We update
\begin{align*}
\calR \leftarrow \calP_0 \circ \calQ_1 \circ \calP_1 \circ \cdots \circ \calP_l(\beta_i^{\overline{\sigma_i}}] \circ (\beta_i^{\overline{\sigma_i}} \alpha_{i+1}^+).
\end{align*}
It is clear that the resulting $\calR$ is an augmenting path for $(M, I, \calV, p)$.
Since the number of edges in the union of outer paths decreases,
so does $\theta$.
This update can be done in $O(|\calR|) = O(\min \set{\mu, \nu})$ time.
Return to the initial stage (\cref{sec:initial}).

\paragraph{(Case~2) All $\beta_i^{\overline{\sigma_i}}$ with $i > i^*$ which $\calR$ meets are BP-invariant with respect to $\calP_m$.}
Since $\calR$ fails ($\No$),
$\calP_m$ is not simple,
i.e., $i^* > 0$ and there is $j^* > i^*$ such that $\beta_{i^*}^{\sigma_{i^*}} = \beta_{j^*}^{\overline{\sigma_{j^*}}}$.
The assumption says that
$\calP_m[ \beta_{j^*}^{\sigma_{j^*}} )$ is simple,
$\beta_{j^*}^{\overline{\sigma_{j^*}}}$ is BP-invariant with respect to $\calP_m$,
and 
every $\beta_j^{\overline{\sigma_j}}$ with $j > {j^*}$ belonging to $\calR$ is also BP-invariant.
Hence
we may assume that
\begin{align*}
\calP_m = (\beta_0^{\sigma_0} \alpha_1^{\sigma_1}, \alpha_1^{\sigma_1} \beta_1^{\sigma_1}, \dots, \alpha_{i^*}^{\sigma_{i^*}} \beta_{i^*}^{\sigma_{i^*}}, \beta_{i^*}^{\sigma_{i^*}} \alpha_{i^*+1}^{\sigma_{i^*+1}},
\dots, \alpha_{j^*}^{\sigma_{j^*}} \beta_{j^*}^{\sigma_{j^*}}, \beta_{j^*}^{\sigma_{j^*}} \alpha_{j^*+1}^+)
\end{align*}
by executing the simplification in \cref{subsec:preprocessing} to $\calP_m[ \beta_{j^*}^{\sigma_{j^*}} )$.

We first delete $\alpha_{i^*} \beta_{i^*} (=\alpha_{j^*} \beta_{j^*})$ from $M$ and $I$;
the resulting edge sets are denoted by $M'$ and $I'$.
We then redefine $+$- and $-$-edges of $M'$
and $+$- and $-$-spaces of $\calV$
so that
all $\alpha_{i^*+1} \beta_{i^*+1}, \alpha_{i^*+2} \beta_{i^*+2}, \dots, \alpha_{j^*-1} \beta_{j^*-1}$ are $-$-edges
and
\begin{align*}
\calP_m = (\dots, \alpha_{i^*}^+ \beta_{i^*}^-, \beta_{i^*}^- \alpha_{i^*+1}^-, \dots, \beta_{j^*-1}^- \alpha_{j^*}^-, \alpha_{j^*}^- \beta_{j^*}^+, \beta_{j^*}^+ \alpha_{j^*+1}^+).
\end{align*}
Note here that
the resulting $\calV$ is no longer a valid labeling for $M$,
since $A_{\alpha_{i^*} \beta_{i^*}}(U_{\alpha_{i^*}}^-, V_{\beta_{i^*}}^+) \neq \set{0} \neq A_{\alpha_{i^*} \beta_{i^*}}(U_{\alpha_{i^*}}^+, V_{\beta_{i^*}}^-)$ (corresponding to the edges $\alpha_{i^*}^+ \beta_{i^*}^-$ and $\alpha_{j^*}^- \beta_{j^*}^+$ in $\calP_m$).
On the other hand,
$\calV$ is a valid labeling for $M'$,
since it does not have $\alpha_{i^*}\beta_{i^*}$.
Here $(M', I')$ is a matching-pair of size $k-2$.

Let $C_{\alpha_{j^*+1}}$
be the connected component of $M' \setminus I'$
containing $\alpha_{j^*+1}$,
and $Y$ be the $+$-space of $\beta_{j^*}$ with respect to $\calV(\calP_m[\beta_{i^*}^-, \alpha_{j^*}^-]^{-1})$.
We define $\hat{M}$, $\hat{I}$, and $\hat{\calV}$ by
\begin{align*}
\hat{M} &\defeq M' \cup P_m[\beta_{i^*}, \alpha_{j^*}] \cup \{ \beta_{j^*} \alpha_{j^*+1} \},\\
\hat{I} &\defeq I' \setminus \set{\alpha_{i^*+1} \beta_{i^*+1}, \alpha_{i^*+2} \beta_{i^*+2}, \dots, \alpha_{j^*-1} \beta_{j^*-1}},\\
\hat{\calV} &\defeq \calV(\calP_m[\beta_{i^*}^-, \alpha_{j^*}^-]^{-1})(P_m[\beta_{i^*}, \alpha_{j^*}]; (U_{\alpha_{j^*+1}}^+)^{\perp_{\alpha_{j^*+1} \beta_{j^*}}})(C_{\alpha_{j^*+1}}; Y^{\perp_{\beta_{j^*} \alpha_{j^*+1}}}),
\end{align*}
where $P_m[\beta_{i^*}, \alpha_{j^*}]$ denotes the subpath of $P_m$ from $\beta_{i^*}$ to $\alpha_{j^*}$;
see Figure~\ref{fig:Nouter}.
\begin{figure}
	\centering
	\begin{tikzpicture}[
	node/.style={
		fill=black, circle, minimum height=5pt, inner sep=0pt,
	},
	M/.style={
		line width = 3pt
	},
	AM/.style={
		line width = 3pt,
		red
	},
	DM/.style={
		line width = 3pt,
		dashed
	},
	DI/.style={
		line width = 3pt,
		blue
	},
	P/.style={
		very thick,
	},
	PD/.style={
		very thick, dashed,
	},
	calE/.style={
		thick
	}
	]
	
	\def\mu{a0, a1, a2, a3, a4, a5, a6, a7, a8}
	\def\nu{b0, b1, b2, b3, b4, b5, b6, b7, b8}
	\def\size{1.7cm}
	\def\hight{4cm}
	\def\side{5pt}
	\def\up{12pt}

	\nodecounter{\nu}
	\coordinate (pos);
	\foreach \currentnode in \nu {
		\node[node, below=0 of pos, anchor=center] (\currentnode) {};
		\coordinate (pos) at ($(pos)+(-\size, 0)$);
	}
	\nodecounter{\mu}
	\coordinate (pos) at ($(b0) + (0, -\size)$);
	\foreach \currentnode in \mu {
		\node[node, below=0 of pos, anchor=center] (\currentnode) {};
		\coordinate (pos) at ($(pos)+(-\size, 0)$);
	}
	
	\foreach \i in {6, 7, 8} {
		\draw[M] (b\i) -- (a\i);
	}
	\foreach \i in {1, 2, 3, 4} {
		\draw[DI] (b\i) -- (a\i);
	}
	\foreach \i/\j in {6/7, 7/8} {
		\draw[M] (b\i) -- (a\j);
	}
	\foreach \i in {0, 5} {
		\draw[DM] (b\i) -- (a\i);
	}
	\foreach \i/\j in {0/1, 1/2, 2/3, 3/4, 4/5, 5/6} {
		\draw[AM] (b\i) -- (a\j);
	}
	
	\foreach \i in \nu {
		\coordinate [left = \side] (l\i) at (\i);
		\coordinate [right = \side] (r\i) at (\i);
	}
	
	\foreach \i in \mu {
		\coordinate [left = \side] (l\i) at (\i);
		\coordinate [right = \side] (r\i) at (\i);
	}

	\draw (lb0) node[above=5pt] {$-$};
	\draw (rb0) node[above=5pt] {$+$};
	\draw (la0) node[below=5pt] {$+$};
	\draw (ra0) node[below=5pt] {$-$};
	\coordinate [above=15pt] (b0n) at (b0);
	\draw (b0n) node[above] {$\beta_{i^*}$};
	
	\draw (lb5) node[above=5pt] {$+$};
	\draw (rb5) node[above=5pt] {$-$};
	\draw (la5) node[below=5pt] {$-$};
	\draw (ra5) node[below=5pt] {$+$};
	\coordinate [above=15pt] (b5n) at (b5);
	\draw (b5n) node[above] {$\beta_{j^*}$};
	
	\coordinate (Pls) at ($(ra1)+(0, -1cm)$);
	\coordinate (Plt) at ($(rb2)+(0, 1cm)$);
	\draw [P, -{Latex[length=3mm]}] (Pls) -- (ra1) -- (rb1) -- (ra2) -- (rb2) -- (Plt);
    \coordinate [label=below:{$\calP_l$}] () at (Pls);
	
	\coordinate (Pms) at ($(ra4)+(0, -1.5cm)$);
	\draw [P, -{Latex[length=3mm]}] (Pms) -- ($(Pms)!0.6!(ra4)$);
	\draw [P] ($(Pms)!0.4!(ra4)$) -- (ra4);
	\draw [PD, -{Latex[length=3mm]}] (ra4) -- (rb4) -- (ra5) -- (rb5);
	\coordinate [label=below:{$\calP_m$}] () at (Pms);
	
	\coordinate (Rs) at ($(ra7)+(0, -1cm)$);
	\coordinate (Rt) at ($(rb7)+(0, 1cm)$);
	\draw[P, -{Latex[length=3mm]}] (Rs) -- (Rt);
    \coordinate [label=below:{$\hat{\calR}$}] () at (Rs);
	
	\coordinate (R2s) at ($(la8)+(0, -1cm)$);
	\coordinate (R2t) at ($(lb7)+(0, 1cm)$);
	\draw[P, -{Latex[length=3mm]}] (R2s) -- (la8) -- (lb7) -- (R2t);
    \coordinate [label=below:{$\hat{\calR}$}] () at (R2s);
	
	\foreach \i in {1, 2, 3, 4} {
		\draw (lb\i) node[above=.3333em-1pt, fill=white, inner sep=1pt] {$-$};
		\draw (rb\i) node[above=.3333em-1pt, fill=white, inner sep=1pt] {$+$};
		\draw (la\i) node[below=.3333em-1pt, fill=white, inner sep=1pt] {$-$};
		\draw (ra\i) node[below=.3333em-1pt, fill=white, inner sep=1pt] {$+$};
	}
	\foreach \i in {6, 7, 8} {
		\draw (lb\i) node[above=.3333em-1pt, fill=white, inner sep=1pt] {$+$};
		\draw (rb\i) node[above=.3333em-1pt, fill=white, inner sep=1pt] {$-$};
		\draw (la\i) node[below=.3333em-1pt, fill=white, inner sep=1pt] {$+$};
		\draw (ra\i) node[below=.3333em-1pt, fill=white, inner sep=1pt] {$-$};
	}
	
	\coordinate [above right=5pt and 3pt] (arb0) at (rb0);
	\coordinate [above left=5pt and 3pt] (alb0) at (lb0);
	\coordinate [below right=5pt and 3pt] (bra0) at (ra0);
	\coordinate [below left=5pt and 3pt] (bla0) at (la0);
	\draw[dotted, very thick] (arb0) -- (alb0) -- (bla0) -- (bra0) -- (arb0);
	
	\coordinate [above right=5pt and 3pt] (arb5) at (rb5);
	\coordinate [above left=5pt and 3pt] (alb5) at (lb5);
	\coordinate [below right=5pt and 3pt] (bra5) at (ra5);
	\coordinate [below left=5pt and 3pt] (bla5) at (la5);
	\draw[dotted, very thick] (arb5) -- (alb5) -- (bla5) -- (bra5) -- (arb5);
	
	\coordinate [below = 15pt] (a0n) at (a0);
	\draw (a0n) node[below] {$\alpha_{i^*}$};
	\coordinate [below = 15pt] (a5n) at (a5);
	\draw (a5n) node[below] {$\alpha_{j^*}$};
	\coordinate [below = \up] (a6n) at (a6);
	\draw (a6n) node[below] {$\alpha_{j^*+1}$};
	\coordinate [below left = \up and 5pt] (a4n) at (a4);
	\draw (a4n) node[below] {$\alpha_{k^*}$};

	\draw[PD, -{Latex[length=3mm]}] (lb0) -- node[above left = 0 and -4pt]{$\calP_m$} ($(lb0)!0.6!(la1)$);
	\draw[PD] ($(lb0)!0.4!(la1)$) -- (la1);
	\foreach \i/\j in {1/1, 1/2, 2/2, 2/3, 3/3, 3/4, 4/4, 4/5, 5/5, 5/6} {
		\draw[PD] (lb\i) -- (la\j);
	}
	
	\foreach \i in {3} {
		\draw [calE] (rb\i) -- (ra\i);
	}

	\coordinate[label=left:$\cdots$] () at ($(b8)!0.5!(a8) + (-0.2cm, 0)$);
	
	\end{tikzpicture}
	\caption{
		Modification in Case~2 in \cref{subsec:Nouter};
		the definitions of all solid thick lines and paths are the same as in Figure~\ref{fig:preprocessing}.
		The dashed thick lines represent the edges deleted from $M$ by the modification, i.e., the edges in $M \setminus \hat{M}$.
	}
	\label{fig:Nouter}
\end{figure}
Then the following holds.
\begin{proposition}\label{prop:not simple}
	$(\hat{M}, \hat{I})$ is a matching-pair of size $k$
	such that $\calU(\hat{M}, \hat{I}) = \calU(M, I) \setminus \{ \alpha_{j^*+1}^+ \} \cup \{ \alpha_{j^*}^+ \}$,
	$\hat{\calV}$ is a valid labeling for $\hat{M}$,
	and
	$p$ is an $(\hat{M}, \hat{I}, \hat{\calV})$-compatible potential.
\end{proposition}
\begin{proof}
	Let $M'' \defeq M' \cup P_m[\beta_{i^*}, \alpha_{j^*}]$.
	We first show that $(M'', \hat{I})$ is a matching-pair of size $k-1$,
	$\calV(\calP_m[\beta_{i^*}^-, \alpha_{j^*}^-]^{-1})$ is a valid labeling for $M''$,
	and
	$p$ is an $(M'', \hat{I}, \calV(\calP_m[\beta_{i^*}^-, \alpha_{j^*}^-]^{-1}))$-compatible $c$-potential.
	
	Since both $\beta_{i^*}$ and $\alpha_{j^*}$ are incident to no edge in $M'$,
	$\beta_{i^*}^-$ and $\alpha_{j^*}^-$ belong to $\calS(M', I', \calV, p)$ and $\calT(M', I', \calV, p)$,
	respectively.
	Hence $\calP_m[ \beta_{i^*}^-, \alpha_{j^*}^- ]$ forms an augmenting path for $(M', I', \calV, p)$,
	which is in the base case (\cref{sec:base}).
	One can easily see that $(M'', \hat{I})$ and $\calV(\calP_m[\beta_{i^*}^-, \alpha_{j^*}^-]^{-1})$ are the resulting pair and labeling,
	respectively, by the augmentation of $(M', I')$ via $\calP_m[ \beta_{i^*}^-, \alpha_{j^*}^- ]$.
	Thus $(M'', \hat{I})$ is a matching-pair of size $k-1$, $\calV(\calP_m[\beta_{i^*}^-, \alpha_{j^*}^-]^{-1})$ is a valid labeling for $M''$,
	and $p$ is an $(M'', \hat{I}, \calV(\calP_m[\beta_{i^*}^-, \alpha_{j^*}^-]^{-1}))$-compatible potential
	by \cref{prop:base case}.
	
	We then show that the edge $\beta_{j^*}^+ \alpha_{j^*+1}^+$ exists in $\calG(\calV(\calP_m[\beta_{i^*}^-, \alpha_{j^*}^-]^{-1}), p)$.
	Recall here that $Y$ is the $+$-space of $\beta_{j^*}$ with respect to $\calV(\calP_m[\beta_{i^*}^-, \alpha_{j^*}^-]^{-1})$.
	By $p(Y) = p(V_{\beta_{j^*}}^+)$ and $\beta_{j^*}^+ \alpha_{j^*+1}^+ \in \calE(\calV, p)$,
	we have $p(U_{\alpha_{j^*+1}}^+) + p(Y) + c = d_{\alpha_{j^*+1} \beta_{j^*}}$.
	Thus it suffices to show that $A_{\alpha_{j^*+1} \beta_{j^*}}(U_{\alpha_{j^*+1}}^+, Y) \neq \set{0}$.
	
	If $p(V_{\beta_{j^*}}^+) < p(V_{\beta_{j^*}}^-)$,
	then \cref{lem:outer path} asserts $Y = V_{\beta_{j^*}}^+$.
	Thus we obtain $A_{\alpha_{j^*+1} \beta_{j^*}}(U_{\alpha_{j^*+1}}^+, Y) = A_{\alpha_{j^*+1} \beta_{j^*}}(U_{\alpha_{j^*+1}}^+, V_{\beta_{j^*}}^+) \neq \set{0}$.
	If $p(V_{\beta_{j^*}}^+) \geq p(V_{\beta_{j^*}}^-)$,
	then it follows from the BP-invariance of $\beta_{j^*}^-$ that
	$V_{\beta_{j^*}}^-$ is the space at $\beta_{j^*}^-$ of the back-propagation of $\calP_m$.
	Hence $A_{\alpha_{j^*+1} \beta_{j^*}}(U_{\alpha_{j^*+1}}^+, V_{\beta_{j^*}}^-) = \set{0}$.
	Since $Y$ is different from $V_{\beta_{j^*}}^-$,
	we obtain $A_{\alpha_{j^*+1} \beta_{j^*}}(U_{\alpha_{j^*+1}}^+, Y) \neq \set{0}$.
	
	We finally show the statement of \cref{prop:not simple}.
	Clearly, we have $\calU(\hat{M}, \hat{I}) = \calU(M, I) \setminus \{ \alpha_{j^*+1}^+ \} \cup \{ \alpha_{j^*}^+ \}$.
	We can see that $(\beta_{j^*}^+ \alpha_{j^*+1}^+)$ forms an augmenting path for $(M'', \hat{I}, \calV(\calP_m[\beta_{i^*}^-, \alpha_{j^*}^-]^{-1}), p)$,
	which is in the base case.
	Furthermore, the connected component of $\hat{M} \setminus \hat{I}$ containing $\beta_{j^*} \alpha_{j^*+1}$ forms a path.
	By \cref{prop:base case} and \cref{rmk:rearrangeable},
	we can argument $(M'', \hat{I})$ via $(\beta_{j^*}^+ \alpha_{j^*+1}^+)$,
	and
	$(\hat{M}, \hat{I})$ and $\hat{\calV}$ can be seen as the resulting pair and labeling,
	respectively.
	Thus $(\hat{M}, \hat{I})$ is a matching-pair of size $k$,
	$\hat{\calV}$ is a valid labeling for $\hat{M}$,
	and
	$p$ is an $(\hat{M}, \hat{I}, \hat{\calV})$-compatible $c$-potential.
\end{proof}

Let $k^*$ be the minimum index with $i^*+1 \leq k^* \leq j^*$
such that $(\alpha_{k^*}^+ \beta_{k^*}^+, \beta_{k^*}^+ \alpha_{k^*+1}^+, \dots, \beta_{j^*-1}^+ \alpha_{j^*}^+)$ forms a subpath of $\calP_m$.
We define
\begin{align*}
\hat{\calR} \defeq \calP_0 \circ \calQ_1 \circ \calP_1 \circ \cdots \circ \calQ_m \circ \calP_m(\alpha_{k^*}^+].
\end{align*}
Figure~\ref{fig:Nouter} also describes $\hat{\calR}$.
The following holds;
the proof is given at the end of this section.
\begin{proposition}\label{prop:not simple pseudo}
	$\hat{\calR}$ is a pseudo augmenting path for $(\hat{M}, \hat{I}, \hat{\calV}, p)$.
\end{proposition}

In this case, $\theta$ strictly decreases.
\begin{lemma}\label{lem:theta case 2}
    $\theta(\hat{M}, \hat{I}, \hat{\calR}) < \theta(M, I, \calR)$.
\end{lemma}
\begin{proof}
	The number of edges in the union of outer paths decreases by $|\calP_m[\alpha_{k^*}^+)| \geq |\calP_m[\beta_{i^*}^-, \alpha_{j^*}^-]| + 2 $,
	and $N_{\rm S}$ increases at most by $|\calP_m[\beta_{i^*}^-, \alpha_{j^*}^-]| + 1$.
	Hence $\theta$ decreases at least by $1$.
\end{proof}

We update
\begin{align*}
M \leftarrow \hat{I}, \qquad
I \leftarrow \hat{I}, \qquad
\calV \leftarrow \hat{\calV}(\hat{\calR}),\qquad
\calR \leftarrow \hat{\calR}.
\end{align*}
This update can be done in $O(|\calR|) = O(\min \set{\mu, \nu})$ time.
By \cref{prop:pseudo augmenting path,prop:not simple,prop:not simple pseudo},
the resulting $(M, I)$ is a matching-pair of size $k$, $\calV$ a valid labeling for $M$,
$p$ an $(M, I, \calV)$-compatible $c$-potential,
and
$\calR$ is an augmenting path for $(M, I, \calV, p)$.
Moreover, since $\alpha_{j^*+1}^+$ is deleted from $\calU(M, I)$ and $\alpha_{j^*}^+$ is added to $\calU(M, I)$ in this update (by \cref{prop:not simple})
and $\alpha_{k^*}^+ \in \calC(\alpha_{j^*}^+)$,
$p$ satisfies (Zero)$'$ for $\calR$.
Return to the initial stage (\cref{sec:initial}).

The proof of \cref{prop:not simple pseudo} requires \cref{lem:C_alpha} below;
\cref{lem:C_alpha}~(1) is also used in the proofs of \cref{prop:cycle =,prop:path:pseudo augmenting path}.
Let $P_{\alpha_{j^*+1}}$ be the maximal rank-2 path in $C_{\alpha_{j^*+1}}$ that starts with $\alpha_{j^*+1}$,
where $P_{\alpha_{j^*+1}} := \set{\alpha_{j^*+1}}$ if $\alpha_{j^*+1}$ is incident to no edge or a rank-1 edge.
\begin{lemma}\label{lem:C_alpha}
	\begin{itemize}
		\item[{\rm (1)}]
		For each $\beta \in V(P_{\alpha_{j^*+1}})$,
		it holds that $p(V_{\beta}^+) > p(V_{\beta}^-)$, the $+$-space of $\beta$ with respect to $\hat{\calV}$
		is $V_{\beta}^+$,
		or $p(V_{\beta}^+) = p(V_{\beta}^-)$ and $\beta^+ \in \calC(\alpha_{j^*+1}^+)$.
        Similarly,
		for each $\alpha \in V(P_{\alpha_{j^*+1}})$,
		it holds that $p(U_{\alpha}^-) > p(U_{\alpha}^+)$,
		the $-$-space of $\alpha$ with respect to $\hat{\calV}$ is $U_{\alpha}^-$,
		or
		$p(U_{\alpha}^-) = p(U_{\alpha}^+)$ and $\alpha^+ \in \calC(\alpha_{j^*+1}^+)$.
		\item[{\rm (2)}]
		If $\hat{\calR}$ meets $\alpha_i^+$ for some $i$ with $i^* < i \leq j^*$,
		then $p(U_{\alpha_i}^+) > p(U_{\alpha_i}^-)$ or the $+$-space of $\alpha_i$ with respect to $\hat{\calV}$
		is $U_{\alpha_i}^+$.
		Similarly,
		if $\hat{\calR}$ meets $\beta_i^+$ for some $i$ with $i^* < i < j^*$,
		then $p(V_{\beta_i}^+) > p(V_{\beta_i}^-)$ or the $+$-space of $\beta_i$ with respect to $\hat{\calV}$
		is $V_{\beta_i}^+$.
	\end{itemize}
\end{lemma}
\begin{proof}
	(1).
	If $p(U_{\alpha_{j^*+1}}^-) < p(U_{\alpha_{j^*+1}}^+)$,
	then it follows from $\beta_{j^*}^+\alpha_{j^*+1}^+ \in \calE(\hat{\calV}, p)$ and \cref{lem:edge} that $Y^{\perp_{\beta_{j^*} \alpha_{j^*+1}}} = U_{\alpha_{j^*+1}}^-$.
	Hence, for each $\alpha, \beta \in V(P_{\alpha_{j^*+1}})$, the $-$-space of $\alpha$ and the $+$-space of $\beta$ with respect to $\hat{\calV}$
	coincide with those with respect to $\calV$,
	i.e., $U_{\alpha}^-$ and $V_{\beta}^+$,
	respectively.
	
	Suppose $p(U_{\alpha_{j^*+1}}^-) \geq p(U_{\alpha_{j^*+1}}^+)$.
	By \cref{lem:single-counted component}~(1),
	we have $p(U_{\alpha}^-) \geq p(U_{\alpha}^+)$ and $p(V_{\beta}^+) \geq p(V_{\beta}^-)$
	for each $\alpha, \beta \in V(P_{\alpha_{j^*+1}})$.
	Moreover, if $p(V_{\beta}^+) = p(V_{\beta}^-)$ (resp. $p(U_{\alpha}^+) = p(U_{\alpha}^-)$)
	then $p(U_{\alpha_{j^*+1}}^-) = p(U_{\alpha_{j^*+1}}^+)$
	and the subpath of $P_{\alpha_{j^*+1}}$ from $\alpha_{j^*+1}$ to $\beta$ (resp. $\alpha$)
	consists of double-tight edges.
	Hence there is an $\alpha_{j^*+1}^+$-$\beta^+$ $+$-path (resp. an $\alpha_{j^*+1}^+$-$\alpha^+$ $+$-path) in $\calG(\calV, p)|_M$,
	which implies that
	$\beta^+$ (resp. $\alpha^+$) belongs to $\calC(\alpha_{j^*+1}^+)$.

	(2).
	There are two cases: (i) $p(U_{\alpha_{j^*}}^+) \geq p(U_{\alpha_{j^*}}^-)$ and (ii) $p(U_{\alpha_{j^*}}^+) < p(U_{\alpha_{j^*}}^-)$.
	
	(i) $p(U_{\alpha_{j^*}}^+) \geq p(U_{\alpha_{j^*}}^-)$.
	By $\beta_{j^*} \alpha_{j^*} \in I$ and (Tight),
	we have $p(V_{\beta_{j^*}}^+) \geq p(V_{\beta_{j^*}}^-)$.
	By the BP-invariance of $\beta_{j^*}^-$,
	$V_{\beta_{j^*}}^-$ is the space at $\beta_{j^*}^+$ of the back-propagation of $\calP_m$,
	i.e.,
	$V_{\beta_{j^*}}^- = (U_{\alpha_{j^*+1}}^+)^{\perp_{\alpha_{j^*+1} \beta_{j^*}}}$.
	Thus we obtain
	\begin{align*}
	\hat{\calV} &= \calV(\calP_m[\beta_{i^*}^-, \alpha_{j^*}^-]^{-1})(P_m[\beta_{i^*}, \alpha_{j^*}]; (U_{\alpha_{j^*+1}}^+)^{\perp_{\alpha_{j^*+1} \beta_{j^*}}})(C_{\alpha_{j^*+1}}; Y^{\perp_{\beta_{j^*} \alpha_{j^*+1}}})\\
	&= \calV(\calP_m[\beta_{i^*}^-, \alpha_{j^*}^-]^{-1})(C_{\alpha_{j^*+1}}; Y^{\perp_{\beta_{j^*} \alpha_{j^*+1}}}).
	\end{align*}
	That is, for each $j = i^*, i^*+1,\dots,j^*-1$,
	the $-$-space of $\beta_{j}$ and the $+$-space of $\alpha_{j+1}$ with respect to $\hat{\calV}$
	are $V_{\beta_{j}}^-$ and $U_{\alpha_{j+1}}^+$,
	respectively.
	In particular,
	the $+$-space of $\alpha_{j^*}$ with respect to $\hat{\calV}$
	is $U_{\alpha_{j^*}}^+$,
	and if $\hat{\calR}$ meets $\alpha_i^+$ for some $i$ with $i^* < i \leq j$,
	then the $+$-spaces of $\alpha_i$ with respect to $\hat{\calV}$ is $U_{\alpha_i}^+$.

	Since $V_{\beta_{j^*}}^-$ is the space at $\beta_{j^*}^+$ of the back-propagation of $\calP_m$ and $U_{\alpha_{j^*}}^- = (V_{\beta_{j^*}}^-)^{\perp_{\beta_{j^*} \alpha_{j^*}}}$,
	$U_{\alpha_{j^*}}^-$ is the space of $\alpha_{j^*}^-$ of the back-propagation of $\calP_m$.
	Therefore,
	the back-propagation of $\calP_m[\beta_{i^*}^-, \alpha_{j^*}^-]$
	coincides with the restriction of the back-propagation of $\calP_m$ to $\calP_m[\beta_{i^*}^-, \alpha_{j^*}^-]$.
	That is,
	for $j = i^*, i^*+1,\dots,j^*-1$,
	the $-$-space of $\alpha_{j+1}$ and the $+$-space of $\beta_j$ with respect to $\hat{\calV}$
	are the spaces at $\alpha_{j+1}^-$ and at $\beta_j^-$ of the back-propagation of $\calP_m$,
	respectively.
	Suppose that $\hat{\calR}$ meets $\beta_i^+$ for some $i$ with $i^* < i < j$.
	Then, by the BP-invariance of $\beta_{i}^+$,
	we obtain that $p(V_{\beta_i}^+) > p(V_{\beta_i}^-)$,
	or that the $+$-space of $\beta_i$ with respect to $\hat{\calV}$
	is $V_{\beta_i}^+$.

	(ii) $p(U_{\alpha_{j^*}}^+) < p(U_{\alpha_{j^*}}^-)$.
	In this case,
	we have $p(V_{\beta_{j^*}}^+) < p(V_{\beta_{j^*}}^-)$.
	It suffices to show that
	\begin{itemize}
		\item
		if $\hat{\calR}$ meets $\alpha_i^+$ for some $i$ with $i^* < i \leq j^*$ and $p(U_{\alpha_i}^+) \leq p(U_{\alpha_i}^-)$,
		then the $+$-space of $\alpha_i$ with respect to $\hat{\calV}$
		is $U_{\alpha_i}^+$, and 
		\item 
		if $\hat{\calR}$ meets $\beta_i^+$ for some $i$ with $i^* < i < j^*$ and $p(V_{\beta_i}^+) \leq p(V_{\beta_i}^-)$,
		then the $+$-space of $\beta_i$ with respect to $\hat{\calV}$
		is the space at $\beta_i^-$ of the back-propagation of $\calP_m$.
	\end{itemize}
	
	We show the former bullet.
	Suppose that $\hat{\calR}$ meets $\alpha_i^+$ for some $i$ with $i^* < i \leq j$ and $p(U_{\alpha_i}^+) \leq p(U_{\alpha_i}^-)$.
	Then the subpath $P_m[ \beta_{i^*}, \alpha_i ]$ of $P_m[ \beta_{i^*}, \alpha_{j^*} ]$ has a rank-1 edge;
	otherwise $p(U_{\alpha_i}^+) > p(U_{\alpha_i}^-)$ holds by $p(V_{\beta_{j^*}}^+) < p(V_{\beta_{j^*}}^-)$ and \cref{lem:single-counted component}~(1), a contradiction.
	Thus the $+$-space of $\alpha_i$ with respect to $\hat{\calV}$
	coincides with that with respect to $\calV$,
	i.e., $U_{\alpha_i}^+$.
	
	We then prove the latter bullet.
	Suppose that $\hat{\calR}$ meets $\beta_i^+$ for some $i$ with $i^* < i < j$ and $p(V_{\beta_i}^+) \leq p(V_{\beta_i}^-)$.
	Then the subpath $P_m[ \beta_{i}, \alpha_{j^*} ]$ of $P_m[ \beta_{i^*}, \alpha_{j^*} ]$ has a rank-1 edge;
	otherwise $p(V_{\beta_i}^+) > p(V_{\beta_i}^-)$ holds by $p(U_{\alpha_{j^*}}^+) < p(U_{\alpha_{j^*}}^-)$ and \cref{lem:single-counted component}~(1),
	a contradiction.
	Thus the $+$-space of $\beta_i$ with respect to $\hat{\calV}$
	coincides with the space at $\beta_i^-$ of the back-propagation of $\calP_m$.
	
	This completes the proof.
\end{proof}

We are ready to prove \cref{prop:not simple pseudo}.
\begin{proof}[Proof of \cref{prop:not simple pseudo}]
	Clearly the initial node $\beta(\calP_0)$ of $\hat{\calR}$ belongs to $\calS(\hat{M}, \hat{I}, \hat{\calV}, p)$.
	We see that the last node $\alpha_{k^*}^+$ of $\hat{\calR}$ belongs to $\calT(\hat{M}, \hat{I}, \hat{\calV}, p)$ as follows,
	which implies that $\hat{\calR}$ satisfies (A2).
	Since $\beta_{j^*-1} \alpha_{j^*}$ is a $+$-edge in $\hat{M}$,
	we have $\alpha_{j^*}^+ \in \calU(\hat{M}, \hat{I})$.
	Since the path $\calP_m[ \alpha_{k^*}^+, \alpha_{j^*}^+ ]$ and $\calP_m[\alpha_{k^*}^-, \alpha_{j^*}^-]$
	form subpaths of the outer path $\calP_m$ in $\calG(\calV, p)$,
	we have $A_{\alpha_{i+1} \beta_i}(U_{\alpha_{i+1}}^-, V_{\beta_i}^+) = A_{\alpha_{i+1} \beta_i}(U_{\alpha_{i+1}}^+, V_{\beta_i}^-) = \set{0}$
	for each $i = k^*, \dots, j^*-1$ by (O2).
	Thus $\alpha_{k^*} \beta_{k^*}, \beta_{k^*} \alpha_{k^*+1}, \dots, \beta_{j^*-1} \alpha_{j^*}$ are rank-2.
	This implies that $\alpha_{k^*} \beta_{k^*}, \beta_{k^*} \alpha_{k^*+1}, \dots, \beta_{j^*-1} \alpha_{j^*}$ are double-tight with respect to $(\hat{\calV}, p)$,
	and hence,
	the $+$-path from $\alpha_{j^*}^+$ to $\alpha_{k^*}^+$ exists in $\calE(\hat{\calV}, p)|_{\hat{M}}$.
	Thus we have $\alpha_{k^*}^+ \in \calC(\alpha_{j^*}^+) \subseteq \calT(\hat{M}, \hat{I}, \hat{\calV}, p)$.

	The proof strategy for (A1)$'$ is similar to the proof of \cref{prop:Nouter}.
	Since we do not require that $\hat{R}$ satisfies (A3)--(A5),
	several arguments in the proof of \cref{prop:Nouter} can be omitted.
	We first show the following claim.
	\begin{claim*}
		$\hat{\calR}$ is a path in $\calG(\hat{\calV}, p)$.
		Moreover,
		for any $\beta^\sigma \alpha^{\sigma'} \in \hat{\calR}$
		such that $\alpha$ is incident to an edge in $\hat{I}$,
		there is no edge between
		$\beta^\sigma$ and $\alpha^{\overline{\sigma'}}$ in $\calE(\hat{\calV}, p)$.
	\end{claim*}
	\begin{proof}[Proof of Claim]
	Take any $\beta^\sigma \alpha^{\sigma'} \in \hat{\calR}$.
	Then we can see that $p(V_\beta^\sigma) > p(V_\beta^{\overline{\sigma}})$
	or the $\sigma$-space of $\beta$ with respect to $\hat{\calV}$ is $V_\beta^\sigma$.
	Indeed,
	if $\beta \notin V(P_{\alpha_j^* + 1}) \cup V(P_m[ \beta_{i^*}, \alpha_{j^*} ])$, or $\beta \in V(P_{\alpha_j^* + 1})$ and $\sigma = -$,
	then the $\sigma$-space of $\beta$ with respect to $\hat{\calV}$ is $V_\beta^\sigma$.
	If $\beta \in V(P_{\alpha_j^* + 1})$ and $\sigma = +$,
	then by \cref{lem:C_alpha}~(1) and (A4),
	$p(V_\beta^\sigma) > p(V_\beta^{\overline{\sigma}})$
	or the $\sigma$-space of $\beta$ with respect to $\hat{\calV}$ is $V_\beta^\sigma$.
	If $\beta \in V(P_m[ \beta_{i^*}, \alpha_{j^*} ])$,
	then $\sigma = +$, and hence by \cref{lem:C_alpha}~(2),
	$p(V_\beta^\sigma) > p(V_\beta^{\overline{\sigma}})$
	or the $\sigma$-space of $\beta$ with respect to $\hat{\calV}$ is $V_\beta^\sigma$.
	Note that $\hat{\calR}$ does not have any of $\beta_{i^*}^+$ and $\beta_{i^*}^-$.

	Suppose that $\alpha$ exits $P_{\alpha_{j^*} +1}$ and $P_m[ \beta_{i^*}, \alpha_{j^*} ]$;
	note that
	this includes the case where $\alpha$ is incident to an edge in $\hat{I}$.
	Then the $+$-space and $-$-space of $\alpha$ with respect to $\hat{\calV}$ are $U_\alpha^+$ and $U_\alpha^-$,
	respectively.
	By \cref{lem:>}~(3),
	the edge $\beta^\sigma \alpha^{\sigma'}$ exists in $\calE(\hat{\calV}, p)$.
	In particular,
	if $\alpha$ is incident to an edge in $\hat{I}$,
	then it is also incident to an edge in $I$
	and $\beta^\sigma \alpha^{\overline{\sigma'}} \notin \calE(\calV, p)$.
	In this case,
	we also obtain $\beta^\sigma \alpha^{\overline{\sigma'}} \notin \calE(\hat{\calV}, p)$ by \cref{lem:>}~(3).
	
	If $\alpha$ belongs to $P_m[ \beta_{i^*}, \alpha_{j^*} ]$,
	then $\sigma' = +$.
	Hence by \cref{lem:C_alpha}~(2),
	$p(U_{\alpha_i}^+) > p(U_{\alpha_i}^+)$ or the $+$-space of $\alpha_i$ with respect to $\hat{\calV}$
	is $U_{\alpha_i}^+$.
	\cref{lem:>}~(3)
	asserts that the edge $\beta^\sigma \alpha^{\sigma'}$ exists in $\calE(\hat{\calV}, p)$.
	
	Suppose that $\alpha$ belongs to $P_{\alpha_{j^*} +1}$.
	Note that the $+$-space of $\alpha$ with respect to $\hat{\calV}$ is $U_{\alpha}^+$.
	Hence,
	if $\sigma' = +$,
	then we have $\beta^\sigma \alpha^{\sigma'} \in \calE(\hat{\calV}, p)$ by \cref{lem:>}~(3).
	If $\sigma' = -$,
	then $p(U_{\alpha}^-) > p(U_{\alpha}^+)$,
	the $-$-space of $\alpha$ with respect to $\hat{\calV}$ is $U_{\alpha}^-$,
	or
	$p(U_{\alpha}^-) = p(U_{\alpha}^+)$ and $\alpha^+ \in \calC(\alpha_{j^*+1}^+)$.
	In the first or second case,
	\cref{lem:>}~(3)
	verifies $\beta^\sigma \alpha^- \in \calE(\hat{\calV}, p)$.
	
	In the last case,
	there is no edge between $\beta^\sigma$ and $\alpha^+$ in $\calE(\calV, p)$
	by the condition (A4).
	By $p(U_{\alpha}^-) = p(U_{\alpha}^+)$ and $\beta^\sigma \alpha^- \in \calE(\calV, p)$,
	we obtain $p(U_\alpha^+) + p(V_\beta^\sigma) + c = d_{\alpha \beta}$,
	which implies $A_{\alpha \beta}(U_\alpha^+, V_\beta^\sigma) = \set{0}$.
	Moreover, by $\beta^\sigma \alpha^- \in \calE(\calV, p)$,
	we have
	$U_\alpha^- \not\subseteq \kerL(A_{\alpha \beta})$ and $V_\beta^\sigma \not\subseteq \kerR(A_{\alpha \beta})$.
	Let $X$ and $Y$ denote the $-$-space of $\alpha$ and the $\sigma$-space of $\beta$ with respect to $\hat{\calV}$,
	respectively.
	Note that $X$ and $Y$ are different from $U_\alpha^+$ and $V_\beta^{\overline{\sigma}}$,
	respectively.
	By \cref{lem:>}~(3),
	we have
	$\beta^\sigma \alpha^+ \notin \calE(\hat{\calV}, p)$,
	implying
	$A_{\alpha \beta}(U_\alpha^+, Y) = \set{0}$.
	If $p(V_\beta^\sigma) \leq p(V_\beta^{\overline{\sigma}})$,
	then the $\sigma$-space of $\beta$ with respect to $\hat{\calV}$ is $V_\beta^\sigma$,
	i.e., $Y = V_{\beta}^\sigma$.
	Hence we obtain $A_{\alpha \beta}(X, V_{\beta}^\sigma) \neq \set{0}$ by $X \neq U_\alpha^+$.
	If $p(V_\beta^\sigma) > p(V_\beta^{\overline{\sigma}})$,
	then one can see that $\alpha \beta$ is rank-1,
	$U_\alpha^+ = \kerL(A_{\alpha \beta})$,
	and $V_{\beta}^{\overline{\sigma}} = \kerR(A_{\alpha \beta})$.
	By $Y \neq V_{\beta}^{\overline{\sigma}}$ and $X \neq U_\alpha^+$,
	we obtain $A_{\alpha \beta}(X, Y) \neq \set{0}$.
	Thus the edge $\beta^\sigma \alpha^{\sigma'}$ exists in $\calE(\hat{\calV}, p)$.
	\end{proof}
	
	By Claim,
	we have $\calQ_l \subseteq \calE(\hat{\calV}, p)$ and $\calP_l \subseteq \calE(\hat{\calV}, p)$ for each $l$.
	The former immediately implies that $\calQ_l$ forms an inner path for $(\hat{M}, \hat{I}, \hat{\calV}, p)$.
	By $\beta^\sigma \alpha^{\overline{\sigma'}} \notin \calE(\hat{\calV}, p)$
	if $\beta^\sigma \alpha^{\sigma'} \in \calP_l$ and $\alpha$ is incident to an edge in $\hat{I}$,
	$\calP_l$ that does not meet $\alpha_{i^*+1}^+ \beta_{i^*+1}^+, \dots, \alpha_{j^*-1}^+ \beta_{j^*-1}^+$
	forms a pseudo outer path for $(\hat{M}, \hat{I}, \hat{\calV}, p)$.
	Also, if $\calP_l$ meets some of $\alpha_{i^*+1}^+ \beta_{i^*+1}^+, \dots, \alpha_{j^*-1}^+ \beta_{j^*-1}^+$,
	where we recall that $\alpha_i \beta_i$ is a $-$-edge in $\hat{M}$ for $i^*+1 \leq i \leq j^*-1$,
	then $\calP_l$ is the concatenation of several pseudo outer paths and inner paths for $(\hat{M}, \hat{I}, \hat{\calV}, p)$
	as seen in the proof of \cref{prop:Nouter}.
	Thus $\hat{\calR}$ satisfies (A1)$'$.
	This completes the proof.
\end{proof}

\subsection{$\calR$ satisfies ($\No$) but violates ($\Ni$)}\label{subsec:Ninner}
Suppose that $\calR$ satisfies ($\No$) but violates ($\Ni$).
We can assume that $\calQ_m$ is an inner $+$-path,
i.e., the last node of $\calQ_m$ is $\beta_0^+$.
Since $\calR$ satisfies ($\No$),
we can additionally assume that $\calP_m = (\beta_0^+ \alpha_1^{\sigma_1})$ by executing the simplification in \cref{subsec:preprocessing} to $\calP_m$.
Recall that $\calQ^-$ is the maximal inner $-$-path
such that $\beta(\calQ^-) = \beta_0^-$.
Since $\calR$ violates ($\Ni$),
$\beta_0^-$ is not BP-invariant with respect to $\calP_m$
and
there is $\calP_l$ with $l \leq m-2$ 
such that the last node of $P_l$ belongs to $Q^-$
and $\alpha(\calP_l)$ is a $-$-vertex or $\calP_l$ is not proper.

Consider such $\calP_l$ with the minimum index $l$.
We denote by $\beta^\sigma \alpha^{\sigma'}$ the last edge of $\calP_l$.
By the assumption,
the edge $\beta^\sigma \alpha^-$ exists in $\calE(\calV, p)$
and $\alpha^-$ belongs to $\calQ^-$.
Let $\calP_l'$ be the path obtained from $\calP_l$ by replacing $\beta^\sigma \alpha^{\sigma'}$ with $\beta^\sigma \alpha^-$
in $\calG(\calV, p)$.

Since $\beta_0^-$ is not BP-invariant with respect to $\calP_m$,
we have $p(V_{\beta_0}^+) = p(V_{\beta_0}^-)$
and $V_{\beta_0}^-$ is different from the space at $\beta_0^+$ of the back-propagation of $\calP_m$;
the latter implies
$A_{\alpha_{1} \beta_0}(U_{\alpha_1}^{\sigma_1}, V_{\beta_0}^-) \neq \set{0}$.
Hence there exists $\beta_0^- \alpha_{1}^+$ in $\calE(\calV, p)$.
We update $\calR$ as
\begin{align}\label{eq:Ninner R}
\hat{\calR} \defeq \calP_0 \circ \calQ_1 \circ \cdots \circ \calQ_l \circ \calP_l' \circ \calQ^-[ \alpha^-, \beta_0^- ] \circ (\beta_0^- \alpha_1^{\sigma_1}).
\end{align}
This update can be done in $O(\min \set{\mu, \nu})$ time.
By the choice of $\calP_l$,
no outer path $\calP_{l'}$ with $l' < l$ meets $\calQ^-[ \alpha^-, \beta_0^- ]$.
Hence $\calR$ forms a path in $\calG(\calV, p)$.
It is clear that $\calP_l'$ and $(\beta_0^- \alpha_{1}^{\sigma_1})$ are outer paths for $(M, I, \calV, p)$,
and that $\calQ^-[ \alpha^-, \beta_0^-]$ is an inner path for $(M, I, \calV, p)$.
Thus
the resulting $\calR$ is an augmenting path for $(M, I, \calV, p)$.
Clearly $p$ satisfies (Zero)$'$ for $\calR$ and $\theta$ strictly decreases.
Return to the initial stage (\cref{sec:initial}).

\section{$\calR$ satisfies both ($\No$) and ($\Ni$) but is not in the base case}\label{sec:non violate}
Suppose that $\calR$ satisfies both ($\No$) and ($\Ni$) but is not in the base case.
We can assume that $\calQ_m$ forms an inner $-$-path,
where the last node $\beta(\calQ_m)$ of $\calQ_m$ is $\beta_0^-$.
Let $C$ be the connected component of $M \setminus I$ containing $Q_m$.
We consider two cases:
\begin{itemize}
	\item
	$C$ is a cycle component.
	\item
	$C$ is a path component.
\end{itemize}
The first and second are discussed in \cref{subsec:cycle -,subsec:path},
respectively.

\subsection{$C$ is a cycle component}\label{subsec:cycle -}
We denote by $\alpha_{-1}^- \beta_0^-$ the last edge of $\calQ_m$.
Let
\begin{align*}
M' \defeq M \setminus \set{ \alpha_{-1} \beta_0 }.
\end{align*}
We can easily see that
$(M', I)$ is a matching-pair of size $k-1$,
$\calV$ is a valid labeling for $M'$,
and
for the resulting $(M', I, \calV)$,
$p$ is an $(M', I, \calV)$-compatible $c$-potential.

Since $C$ is a cycle component,
the last node of $P_m$ exits $C$.
Hence we can assume that $\calP_m = (\beta_0^- \alpha_1^-)$
by executing the simplification in \cref{subsec:preprocessing} for $\calP_m$.
Let $C_{\beta_0}$ and $C_{\alpha_1}$ be the connected component of $M' \setminus I$
containing $\beta_0$ and $\alpha_1$,
respectively;
note that $C_{\beta_0}$ is equal to $C \setminus \set{ \alpha_{-1} \beta_0 }$.

We define $\hat{M}$, $\hat{\calV}$, and $\hat{\calR}$ by
\begin{align*}
\hat{M} &\defeq M' \cup \set{ \beta_0 \alpha_1 },\\ 
\hat{\calV} &\defeq
\calV(C_{\alpha_1}; (V_{\beta_0}^-)^{\perp_{\beta_0 \alpha_1}})(C_{\beta_0}; (U_{\alpha_1}^-)^{\perp_{\alpha_1 \beta_0}}),\\
\hat{\calR} &\defeq
\calP_0 \circ \calQ_1 \circ \calP_1 \circ \cdots \circ \calP_{m-1}.
\end{align*}

\begin{proposition}\label{prop:cycle =}
	$(\hat{M}, I)$ is a matching-pair of size $k$
	such that $\calU(\hat{M}, I) = \calU(M, I) \setminus \set{\alpha_1^-} \cup \set{\alpha_{-1}^-}$,
$\hat{\calV}$ is a valid labeling for $\hat{M}$,
$p$ is an $(\hat{M}, I, \hat{\calV})$-compatible $c$-potential,
and $\hat{\calR}$ is a pseudo augmenting path for $(\hat{M}, I, \hat{\calV}, p)$.
\end{proposition}

The proof of \cref{prop:cycle =} requires the following lemma,
which is also used in the proof of \cref{prop:path:pseudo augmenting path} in \cref{subsec:path}.
Let $P_{\beta_0}$ be the maximal rank-2 path in $C_{\beta_0}$ starting from $\beta_0$,
where $P_{\beta_0} \defeq \set{\beta_0}$ if $\beta_0$ is incident to a rank-1 edge.
\begin{lemma}\label{lem:C_beta}
	If $\hat{\calR}$ meets $\alpha^+ \beta^+$ for some $\alpha \beta \in P_{\beta_0}$,
	then $p(V_{\beta}^+) > p(V_{\beta}^-)$
	or
	the $+$-space of $\beta$ with respect to $\hat{\calV}$ is $V_{\beta}^+$.
	If $\hat{\calR}$ meets $\alpha^- \beta^-$ for some $\alpha \beta \in P_{\beta_0}$,
	then $p(U_{\alpha}^-) > p(U_{\alpha}^+)$, the $-$-space of $\alpha$ with respect to $\hat{\calV}$
	is $U_{\alpha}^-$,
	or $p(U_\alpha^+) = p(U_\alpha^-)$ and $\alpha^+$ belongs to $\calQ^+$.
\end{lemma}
\begin{proof}
	If $(U_{\alpha_1}^-)^{\perp_{\alpha_1 \beta_0}} = V_{\beta_0}^+$,
	then we have $\hat{\calV} = \calV(C_{\alpha_1}; (V_{\beta_0}^-)^{\perp_{\beta_0 \alpha_1}})$.
	That is,
	for each $\alpha, \beta$ belonging to $P_{\beta_0}$, the $+$-space of $\beta$ and the $-$-space of $\alpha$ with respect to $\hat{\calV}$
	coincide with those with respect to $\calV$,
	i.e., $V_{\beta}^+$ and $U_{\alpha}^-$,
	respectively.

	In the following,
	we assume that $(U_{\alpha_1}^-)^{\perp_{\alpha_1 \beta_0}} \neq V_{\beta_0}^+$.
	Then it follows from $\beta_{0}^-\alpha_{1}^- \in \calE(\calV, p)$ and \cref{lem:edge} that
	$p(V_{\beta_0}^+) \geq p(V_{\beta_0}^-)$.
	By \cref{lem:single-counted component}~(1),
	we obtain $p(V_{\beta}^+) \geq p(V_{\beta}^-)$ and $p(U_{\alpha}^-) \geq p(U_{\alpha}^+)$
	for each $\alpha, \beta$ belonging to $P_{\beta_0}$.
	
	If $\hat{\calR}$ meets $\alpha^+ \beta^+$ for some $\alpha \beta \in V(P_{\beta_0})$,
	then we have $p(V_{\beta}^+) > p(V_{\beta}^-)$.
	Indeed,
	suppose to the contrary
	that $p(V_{\beta}^+) = p(V_{\beta}^-)$.
	Then $p(V_{\beta_0}^+) = p(V_{\beta_0}^-)$
	and the subpath of $P_{\beta_0}$ from $\beta_0$ to $\beta$
	consists of double-tight edges.
	The former with $(U_{\alpha_1}^-)^{\perp_{\alpha_1 \beta_0}} \neq V_{\beta_0}^+$ implies that $\beta_0^+$ is not BP-invariant with respect to $\calP_m$.
	The latter implies that there is a $+$-path from $\beta_0^+$ to $\alpha^+$ in $\calG(\calV, p)|_M$,
	i.e.,
	$\alpha^+$ belongs to $\calQ^+$.
	This contradicts that $\calR$ satisfies ($\Ni$).

	By a similar argument,
	if $p(U_{\alpha}^+) = p(U_{\alpha}^-)$,
	then $\alpha^+$ belongs to $\calQ^+$.
	Thus, if $\hat{\calR}$ meets $\alpha^- \beta^-$ for some $\alpha \beta \in P_{\beta_0}$,
	then $p(U_{\alpha}^-) > p(U_{\alpha}^+)$
	or $p(U_\alpha^+) = p(U_\alpha^-)$ and $\alpha^+$ belongs to $\calQ^+$.
\end{proof}

We are ready to show \cref{prop:cycle =}.
\begin{proof}[Proof of \cref{prop:cycle =}]
    Since $\beta_0^-$ and $\alpha_1^-$ belong to $\calS(M', I, \calV, p)$ and $\calT(M', I, \calV, p)$,
    respectively,
    $(\beta_0^- \alpha_1^-)$ forms an augmenting path for $(M', I, \calV, p)$.
    Furthermore, the connected component of $\hat{M}$ containing $\beta_0 \alpha_1$ forms a path.
    Thus, by \cref{prop:base case} and \cref{rmk:rearrangeable},
    $(\hat{M}, I)$ is a matching-pair of size $k$
	such that $\calU(\hat{M}, I) = \calU(M, I) \setminus \set{\alpha_1^-} \cup \set{\alpha_{-1}^-}$,
$\hat{\calV}$ is a valid labeling for $\hat{M}$,
and
$p$ is an $(\hat{M}, I, \hat{\calV})$-compatible $c$-potential
(even if $C_{\beta_0}$ is rearrangeable in $M'$).

    We then show that $\hat{\calR}$ is a pseudo augmenting path for $(\hat{M}, I, \hat{\calV}, p)$.
	Since $C$ is a cycle component in $M$,
	the initial node of $P_0$ does not belong to $C$.
	Therefore the initial node $\beta(\calP_0)$ of $\calP_0$
	belongs to $\calS(\hat{M}, I, \hat{\calV}, p)$.
	We can see that 
	the last node of $\hat{\calR}$,
	which is $\alpha(\calP_{m-1})$,
	belongs to $\calT(\hat{M}, I, \hat{\calV}, p)$
	as follows.
	Clearly $\deg_{\hat{M}}(\alpha_{-1}) = 1$
    and $\alpha_{-1}$ is incident to a $-$-edge.
    Moreover, by \cref{lem:>}~(2),
	the $-$-path from $\alpha(\calP_{m-1})$ to $\alpha_{-1}^-$ (the reverse of $\calQ_m \setminus \set{ \alpha_{-1}^- \beta_0^+ }$) exists in $\calG(\hat{\calV}, p)|_{\hat{M}}$.
	Thus we obtain $\alpha(\calP_{m-1}) \in \calC(\alpha_{-1}^-) \subseteq \calT(\hat{M}, I, \hat{\calV}, p)$.
	Hence $\hat{\calR}$ satisfies (A2).
	
	Take any $\beta^\sigma \alpha^{\sigma'} \in \hat{\calR}$.
	By the same argument as in the proof of \cref{prop:not simple pseudo} (or \cref{prop:Nouter}),
	it suffices to show that $\beta^\sigma \alpha^{\sigma'} \in \calE(\hat{\calV}, p)$
	and
	that
	if $\alpha$ is incident to an edge in $I$,
	there is no edge between
	$\beta^\sigma$ and $\alpha^{\overline{\sigma'}}$ in $\calE(\hat{\calV}, p)$.
	By \cref{lem:C_alpha}~(1) and \cref{lem:C_beta},
	$p(V_\beta^\sigma) > p(V_\beta^{\overline{\sigma}})$
	or the $\sigma$-space of $\beta$ with respect to $\hat{\calV}$ is $V_\beta^\sigma$.
	
	Suppose $\alpha \notin V(P_{\alpha_1}) \cup V(P_{\beta_0})$.
	Then the $+$-space and $-$-space of $\alpha$ with respect to $\hat{\calV}$ are $U_\alpha^+$ and $U_\alpha^-$,
	respectively.
	By \cref{lem:>}~(3) and $\beta^\sigma \alpha^{\sigma'} \in \calE(\calV, p)$,
	the edge $\beta^\sigma \alpha^{\sigma'}$ exists in $\calE(\hat{\calV}, p)$.
	In particular,
	if $\alpha$ is incident to an edge in $I$,
	then $\beta^\sigma \alpha^{\overline{\sigma'}} \notin \calE(\calV, p)$.
	In this case,
	we also obtain $\beta^\sigma \alpha^{\overline{\sigma'}} \notin \calE(\hat{\calV}, p)$ by \cref{lem:>}~(3).
	
	Suppose that $\alpha$ belongs to $P_{\beta_0}$.
	If $\sigma' = +$,
	then the $+$-space of $\alpha$ with respect to $\hat{\calV}$ is $U_{\alpha}^+$.
	Hence \cref{lem:>}~(3)
	asserts that the edge $\beta^\sigma \alpha^+$ exists in $\calE(\hat{\calV}, p)$.
	If $\sigma' = -$,
	then by \cref{lem:C_beta}, it holds that $p(U_{\alpha}^-) > p(U_{\alpha}^+)$, the $-$-space of $\alpha$ with respect to $\hat{\calV}$
	is $U_{\alpha}^-$,
	or $p(U_\alpha^+) = p(U_\alpha^-)$ and $\alpha^+$ belongs to $\calQ^+$.
	In the first or second case,
	\cref{lem:>}~(3)
	verifies $\beta^\sigma \alpha^- \in \calE(\hat{\calV}, p)$.

	Otherwise the $-$-space of $\alpha$ with respect to $\hat{\calV}$ is different from $U_{\alpha}^-$,
	$p(U_\alpha^+) = p(U_\alpha^-)$,
	and $\alpha^+$ belongs to $\calQ^+$.
	Then one can see that $\beta_0^+$ is not BP-invariant with respect to $\calP_m$.
	Hence,
	by ($\Ni$),
	there is no edge between $\beta^\sigma$ and $\alpha^+$ in $\calE(\calV, p)$.
	By $p(U_{\alpha}^+) = p(U_{\alpha}^-)$ and $\beta^\sigma \alpha^- \in \calE(\calV, p)$,
	we obtain $p(U_\alpha^+) + p(V_\beta^\sigma) + c = d_{\alpha \beta}$,
	which implies $A_{\alpha \beta}(U_\alpha^+, V_\beta^\sigma) = \set{0}$.
	Moreover, by $\beta^\sigma \alpha^- \in \calE(\calV, p)$,
	we have
	$U_\alpha^- \not\subseteq \kerL(A_{\alpha \beta})$ and $V_\beta^\sigma \not\subseteq \kerR(A_{\alpha \beta})$.
	Let $X$ and $Y$ denote the $-$-space of $\alpha$ and the $\sigma$-space of $\beta$ with respect to $\hat{\calV}$,
	respectively.
	Note here that the $+$-space of $\alpha$ with respect to $\hat{\calV}$ is $U_{\alpha}^+$.
	By \cref{lem:>}~(3),
	we have
	$\beta^\sigma \alpha^+ \notin \calE(\hat{\calV}, p)$,
	implying
	$A_{\alpha \beta}(U_\alpha^+, Y) = \set{0}$.
	If $p(V_\beta^\sigma) \leq p(V_\beta^{\overline{\sigma}})$,
	then the $\sigma$-space of $\beta$ with respect to $\hat{\calV}$ is $V_\beta^\sigma$,
	i.e., $Y = V_{\beta}^\sigma$.
	Hence we obtain $A_{\alpha \beta}(X, V_{\beta}^\sigma) \neq \set{0}$ by $X \neq U_\alpha^+$.
	If $p(V_\beta^\sigma) > p(V_\beta^{\overline{\sigma}})$,
	then one can see that $\alpha \beta$ is rank-1,
	$U_\alpha^+ = \kerL(A_{\alpha \beta})$,
	and $V_{\beta}^{\overline{\sigma}} = \kerR(A_{\alpha \beta})$.
	By $X \neq U_\alpha^+$ and $Y \neq V_{\beta}^{\overline{\sigma}}$,
	we obtain $A_{\alpha \beta}(X, Y) \neq \set{0}$.
	Thus the edge $\beta^\sigma \alpha^-$ exists in $\calE(\hat{\calV}, p)$.

	If $\alpha$ belongs to $P_{\alpha_1}$,
	then by the same argument as in the proof of \cref{prop:not simple pseudo} and \cref{lem:C_alpha}~(1),
	we obtain $\beta^\sigma \alpha^{\sigma'} \in \calE(\hat{\calV}, p)$.
	This completes the proof.
\end{proof}

By this update, $\theta$ does not increase and $\phi$ strictly decreases.
\begin{lemma}\label{lem:cycle theta varphi}
    $\theta(\hat{M}, \hat{I}, \hat{\calR}) \leq \theta(M, I, \calR)$
	and $\phi(\hat{M}, \hat{I}, \hat{\calR}) < \phi(M, I, \calR)$.
\end{lemma}
\begin{proof}
    By $\hat{M} = M \setminus \set{\alpha_{-1} \beta_0} \cup \set{\beta_0 \alpha_1}$,
    $N_{\rm S}$ does not change.
    The edge $\beta_0^- \alpha_1^-$ is removed from the union of the resulting outer paths,
    and only $\beta_0^- \alpha_1^-$ can be newly added to the union of the resulting outer paths,
    implying that the number of edges in the union of outer paths does not increase.
    Thus $\theta(\hat{M}, \hat{I}, \hat{\calR}) \leq \theta(M, I, \calR)$ holds.
    
    We obtain $|\hat{\calR}| = |\calR| - |\calQ_m| - |\calP_m|$ and $N_0(\hat{M}, I, \hat{\calR}) = N_0(M, I, \calR)$.
	Thus $\phi(\hat{M}, I, \hat{\calR}) < \phi(M, I, \calR)$ holds.
\end{proof}

We update
\begin{align*}
M \leftarrow \hat{M}, \qquad
\calV \leftarrow \hat{\calV}(\calR),\qquad
\calR \leftarrow \hat{\calR}.
\end{align*}
This update can be done in $O(\min\set{\mu, \nu})$ time.
By \cref{prop:pseudo augmenting path,prop:cycle =},
the resulting $(M, I)$ is a matching-pair of size $k$, $\calV$ a valid labeling for $M$,
$p$ an $(M, I, \calV)$-compatible $c$-potential,
and
$\calR$ is an augmenting path for $(M, I, \calV, p)$.
Moreover, since $\alpha_1^-$ is deleted from $\calU(M, I)$ and $\alpha_{-1}^-$ is added to $\calU(M, I)$ in this update (by \cref{prop:cycle =})
and $\alpha(\calP_{m-1}) \in \calC(\alpha_{-1}^-)$,
$p$ satisfies (Zero)$'$ for $\calR$.
Return to the initial stage (\cref{sec:initial}).

\subsection{$C$ is a path component}\label{subsec:path}
We denote by $\alpha_{-1}^- \beta_0^-$ the last edge of $\calQ_m$.
Let
\begin{align*}
	M' \defeq M \setminus \set{ \alpha_{-1} \beta_0 }.
\end{align*}
We then redefine $+$- and $-$-edges of $M'$
and $+$- and $-$-spaces of $\calV$
so that all edges in $M' \cap P_m$ are $+$-edges,
$\beta_0$ is incident to a $+$-edge if $\deg_{M'}(\beta_0) = 1$,
and
$\calP_m$ forms a $+$-path.
Note that the last edge of $\calQ_m$ (or the last edge of $\calQ^-$) becomes $\alpha_{-1}^- \beta_0^+$.
The resulting $\calV$ is no longer a valid labeling for $M$,
since $A_{\alpha_{-1} \beta_{0}}(U_{\alpha_{-1}}^-, V_{\beta_{0}}^+) \neq \set{0}$ (corresponding to the edge $\alpha_{-1}^- \beta_{0}^+$ in $\calQ_m$).
On the other hand,
$\calV$ is a valid labeling for $M'$,
since $M'$ does not have $\alpha_{-1}\beta_{0}$.
Here $(M', I)$ is a matching-pair of size $k-1$.
We can easily see that,
for the resulting $(M', I, \calV)$,
$p$ is an $(M', I, \calV)$-compatible $c$-potential.

By applying the simplification in \cref{subsec:preprocessing} to $\calP_m$,
we can assume that $\calP_m = (\beta_0^+ \alpha_1^+)$.
Moreover, $\beta_0^+$ and $\alpha_1^+$ belong to $\calS(M', I, \calV, p)$ and $\calT(M', I, \calV, p)$,
respectively.
Hence $(\beta_0^+ \alpha_1^+)$ forms an augmenting path for $(M', I, \calV, p)$.
Let $C_{\alpha_1}$ be the connected component of $M' \setminus I$
containing $\alpha_1$.
We define $\hat{M}$ and $\hat{\calV}$ by
\begin{align*}
\hat{M} &\defeq M' \cup \set{ \beta_0 \alpha_1 },\\
\hat{\calV} &\defeq
\calV(C_{\alpha_1}; (V_{\beta_0}^+)^{\perp_{\beta_0 \alpha_1}})(C_{\beta_0}; (U_{\alpha_1}^+)^{\perp_{\alpha_1 \beta_0}});
\end{align*}
see Figure~\ref{fig:path}.
\begin{figure}
	\centering
	\begin{tikzpicture}[
	node/.style={
		fill=black, circle, minimum height=5pt, inner sep=0pt,
	},
	M/.style={
		line width = 3pt
	},
	AM/.style={
		line width = 3pt,
		red
	},
	DM/.style={
		line width = 3pt,
		dashed
	},
	DI/.style={
		line width = 3pt,
		blue
	},
	P/.style={
		very thick,
	},
	PD/.style={
		very thick, dashed,
	},
	calE/.style={
		thick
	}
	]
	
	\def\mu{a1, a2, a3}
	\def\nu{b0, b1, b2, b3}
	\def\bmu{ba1, ba2, ba3}
	\def\bnu{bb1, bb2, bb3}
	\def\brmu{ba0, ba-1, ba-2}
	\def\brnu{bb0, bb-1, bb-2, bb-3}
	\def\size{1.7cm}
	\def\hight{4cm}
	\def\side{5pt}
	\def\up{12pt}
	
	\nodecounter{\nu}
	\coordinate (pos);
	\foreach \currentnode in \nu {
		\node[node, below=0 of pos, anchor=center] (\currentnode) {};
		\coordinate (pos) at ($(pos)+(-\size, 0)$);
	}
	\nodecounter{\mu}
	\coordinate (pos) at ($(b1) + (0, -\size)$);
	\foreach \currentnode in \mu {
		\node[node, below=0 of pos, anchor=center] (\currentnode) {};
		\coordinate (pos) at ($(pos)+(-\size, 0)$);
	}

	\nodecounter{\brnu}
	\coordinate (pos) at ($(b0) + (0, -\hight)$);
	\foreach \currentnode in \brnu {
		\node[node, below=0 of pos, anchor=center] (\currentnode) {};
		\coordinate (pos) at ($(pos)+(\size, 0)$);
	}
	\nodecounter{\brmu}
	\coordinate (pos) at ($(bb0) + (0, -\size)$);
	\foreach \currentnode in \brmu {
		\node[node, below=0 of pos, anchor=center] (\currentnode) {};
		\coordinate (pos) at ($(pos)+(\size, 0)$);
	}
	\nodecounter{\bnu}
	\coordinate (pos) at ($(bb0) + (-\size, 0)$);
	\foreach \currentnode in \bnu {
		\node[node, below=0 of pos, anchor=center] (\currentnode) {};
		\coordinate (pos) at ($(pos)+(-\size, 0)$);
	}
	\nodecounter{\bmu}
	\coordinate (pos) at ($(bb0) + (-\size, -\size)$);
	\foreach \currentnode in \bmu {
		\node[node, below=0 of pos, anchor=center] (\currentnode) {};
		\coordinate (pos) at ($(pos)+(-\size, 0)$);
	}

	\foreach \i/\j in {a1/b1, b1/a2, a2/b2, b2/a3, a3/b3, ba1/bb1, bb1/ba2, ba2/bb2, bb2/ba3, ba3/bb3, bb0/ba0, ba0/bb-1, bb-1/ba-1, ba-1/bb-2, bb-2/ba-2, ba-2/bb-3} {
		\draw[M] (\i) -- (\j);
	}
	
	\draw[AM] (b0) -- (a1);
	\draw[DM] (bb0) -- (ba1);

	\foreach \i in \nu {
		\coordinate [left = \side] (l\i) at (\i);
		\coordinate [right = \side] (r\i) at (\i);
	}
	\foreach \i in \mu {
		\coordinate [left = \side] (l\i) at (\i);
		\coordinate [right = \side] (r\i) at (\i);
	}
	\foreach \i in \bnu {
		\coordinate [left = \side] (l\i) at (\i);
		\coordinate [right = \side] (r\i) at (\i);
	}
	\foreach \i in \bmu {
		\coordinate [left = \side] (l\i) at (\i);
		\coordinate [right = \side] (r\i) at (\i);
	}
	\foreach \i in \brnu {
		\coordinate [left = \side] (l\i) at (\i);
		\coordinate [right = \side] (r\i) at (\i);
	}
	\foreach \i in \brmu {
		\coordinate [left = \side] (l\i) at (\i);
		\coordinate [right = \side] (r\i) at (\i);
	}
	\coordinate [above = \up] (b0n) at (b0);
	\coordinate [above = \up] (bb0n) at (bb0);
	\draw[dotted, semithick] (b0) -- (bb0);
	\draw (b0n) node[above] {$\beta_0$};
	\draw (bb0n) node[above, fill=white, inner sep=2pt] {$\beta_0$};

	\draw[PD, -{Latex[length=3mm]}] (lb0) -- node[above left = 0 and -4pt]{$\calP_m$} ($(lb0)!0.6!(la1)$);
	\draw[PD] ($(lb0)!0.4!(la1)$) -- (la1);
	
	\draw[PD, -{Latex[length=3mm]}] (lba2) -- ($(lba2)!0.6!(lbb1)$);
	\draw[PD] ($(lba2)!0.4!(lbb1)$) --node[above left = 0 and -7pt]{$\calQ_m$} (lbb1);
	\draw[PD] (lbb1) -- (lba1) -- (lbb0);
	
	\coordinate (Pm-1s) at ($(lba2)+(0, -1cm)$);
	\draw[P, -{Latex[length=3mm]}] (Pm-1s) -- ($(Pm-1s)!0.6!(lba2)$);
	\draw[P] ($(Pm-1s)!0.4!(lba2)$) -- (lba2);
    \coordinate [label=below:{$\calP_{m-1}$}] () at (Pm-1s);
	
	\coordinate (Rs) at ($(ra2)+(0, -1cm)$);
	\coordinate (Rt) at ($(rb2)+(0, 1cm)$);
	\draw[P, -{Latex[length=3mm]}] (Rs) -- (Rt);
    \coordinate [label=below:{$\hat{\calR}$}] () at (Rs);
	
	\coordinate (R2s) at ($(la3)+(0, -1cm)$);
	\coordinate (R2t) at ($(lb2)+(0, 1cm)$);
	\draw[P, -{Latex[length=3mm]}] (R2s) -- (la3) -- (lb2) -- (R2t);
    \coordinate [label=below:{$\hat{\calR}$}] () at (R2s);
	
	\coordinate (R3s) at ($(lba0)+(0, -1cm)$);
	\coordinate (R3t) at ($(lbb-1)+(0, 1cm)$);
	\draw[P, -{Latex[length=3mm]}] (R3s) -- (lba0) -- (lbb-1) -- (R3t);
    \coordinate [label=below:{$\hat{\calR}$}] () at (R3s);
	
	\coordinate (R4s) at ($(rba-2)+(0, -1cm)$);
	\coordinate (R4t) at ($(rbb-1)+(0, 1cm)$);
	\draw[P, -{Latex[length=3mm]}] (R4s) -- (rba-2) -- (rbb-2) -- (rba-1) -- (rbb-1) -- (R4t);
    \coordinate [label=below:{$\hat{\calR}$}] () at (R4s);
	
	\coordinate [below = \up] (a1n) at (a1);
	\draw (a1n) node[below] {$\alpha_1$};
	\coordinate [below = \up] (ba1n) at (ba1);
	\draw (ba1n) node[below] {$\alpha_{-1}$};
	
	\coordinate[label=left:$\cdots$] () at ($(b3)!0.5!(a3) + (-0.2cm, 0)$);
	\coordinate[label=left:$\cdots$] () at ($(bb3)!0.5!(ba3) + (-0.2cm, 0)$);
	\coordinate[label=right:$\cdots$] () at ($(bb-3)!0.5!(ba-2) + (0.5cm, 0)$);

	\foreach \i in {1, 3} {
		\draw[calE] (rb\i) -- (ra\i);
	}
	\foreach \i in {0, 1, 2, 3} {
		\draw[calE] (rbb\i) -- (rba\i);
	}
	\foreach \i/\j in {1/2} {
		\draw[calE] (lb\i) -- (la\j);
	}
	\foreach \i/\j in {-3/-2, -2/-1, 2/3} {
		\draw[calE] (lbb\i) -- (lba\j);
	}
	\foreach \i in {1} {
		\draw[calE] (lb\i) -- (la\i);
	}
	\foreach \i in {2} {
		\draw[calE] (lbb\i) -- (lba\i);
	}
	\foreach \i/\j in {0/1, 1/2, 2/3} {
		\draw[calE] (rbb\i) -- (rba\j);
	}

	\coordinate [left = 5pt] (Q-s) at (lba3);
	\coordinate [left = 5pt] (Q-st) at (lbb2);
	\draw[[-{Latex[length=2mm]}, dotted, thick] (Q-s) -- ($(Q-s)!0.5!(Q-st)$);
	\coordinate[label=left:$\calQ^-$] () at ($(Q-s) + (3pt, 0)$);
	\coordinate [left = 5pt] (Q-ts) at (lba1);
	\coordinate [left = 5pt] (Q-t) at (lbb0);
	\draw[[-{Latex[length=2mm]}, dotted, thick] ($(Q-ts)!0.6!(Q-t)$) --node[above left = 0pt and -10pt] {$\calQ^-$} (Q-t);
	
	\coordinate [right = 5pt] (Qs) at (rba1);
	\coordinate [right = 5pt] (Qst) at (rbb1);
	\draw[[-{Latex[length=2mm]}, dotted, thick] (Qs) -- ($(Qs)!0.6!(Qst)$);
	\coordinate[label=right:$\calQ$] () at ($(Qs) + (-2pt, 0)$);
	\coordinate [right = 5pt] (Qts) at (rba3);
	\coordinate [right = 5pt] (Qt) at (rbb3);
	\draw[[-{Latex[length=2mm]}, dotted, thick] ($(Qts)!0.6!(Qt)$) --node[above right = 0pt and -2pt] {$\calQ$} (Qt);

	\foreach \i in \nu {
		\draw (l\i) node[above=.3333em-1pt, fill=white, inner sep=1pt] {$+$};
		\draw (r\i) node[above=.3333em-1pt, fill=white, inner sep=1pt] {$-$};
	}
	\foreach \i in \mu {
		\draw (l\i) node[below=.3333em-1pt, fill=white, inner sep=1pt] {$+$};
		\draw (r\i) node[below=.3333em-1pt, fill=white, inner sep=1pt] {$-$};
	}
	\foreach \i in \bnu {
		\draw (l\i) node[above=.3333em-1pt, fill=white, inner sep=1pt] {$-$};
		\draw (r\i) node[above=.3333em-1pt, fill=white, inner sep=1pt] {$+$};
	}
	\foreach \i in \bmu {
		\draw (l\i) node[below=.3333em-1pt, fill=white, inner sep=1pt] {$-$};
		\draw (r\i) node[below=.3333em-1pt, fill=white, inner sep=1pt] {$+$};
	}
	\foreach \i in \brnu {
		\draw (l\i) node[above=.3333em-1pt, fill=white, inner sep=1pt] {$+$};
		\draw (r\i) node[above=.3333em-1pt, fill=white, inner sep=1pt] {$-$};
	}
	\foreach \i in \brmu {
		\draw (l\i) node[below=.3333em-1pt, fill=white, inner sep=1pt] {$+$};
		\draw (r\i) node[below=.3333em-1pt, fill=white, inner sep=1pt] {$-$};
	}
	
	\coordinate [above = \up] (b0n) at (b0);
	\coordinate [above = \up] (bb0n) at (bb0);
	\draw[dotted, semithick] (b0) -- (bb0);
	\draw (b0n) node[above] {$\beta_0$};
	\draw (bb0n) node[above, fill=white, inner sep=2pt] {$\beta_0$};
	
	\end{tikzpicture}
	\caption{
		Modification in \cref{subsec:path};
		the definitions of all lines and paths are the same as in Figures~\ref{fig:preprocessing} and~\ref{fig:Nouter}.
	}
	\label{fig:path}
\end{figure}
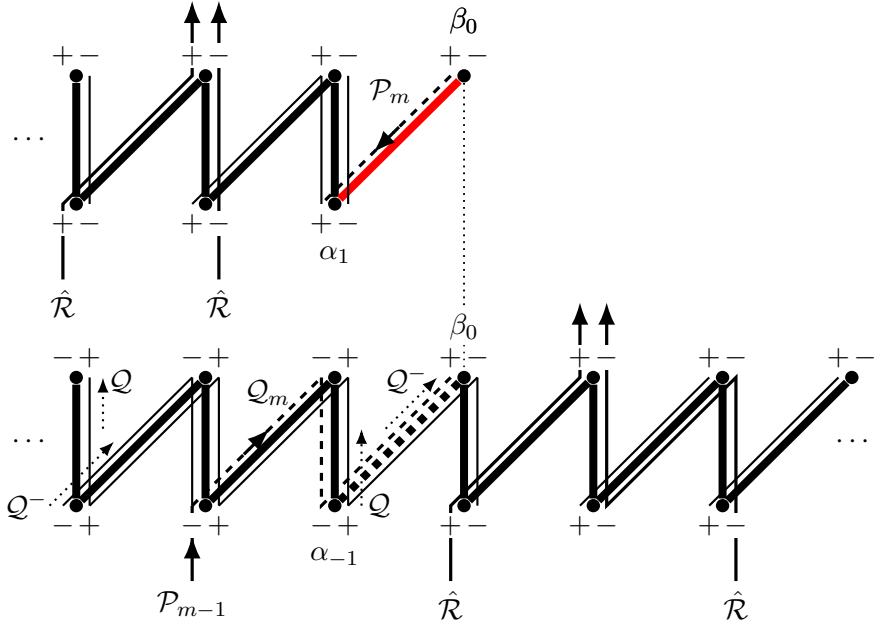

If the connected component of $\hat{M}$ containing $\beta_0 \alpha_1$ forms a cycle,
i.e., $\alpha_1$ belongs to $C_{\beta_0}$,
then $C_{\beta_0}$ is not rearrangeable in $M'$.
Indeed, otherwise we have $\alpha(\calP_{m-1}) \in \calC(\alpha_1^+)$,
which contradicts (A4).
Thus,
by \cref{prop:base case} and \cref{rmk:rearrangeable},
we obtain the following:
\begin{proposition}\label{prop:=}
		$(\hat{M}, I)$ is a matching-pair of size $k$
		such that $\calU(\hat{M}, I) = \calU(M, I) \setminus \set{\alpha_1^+} \cup \set{ \alpha_{-1}^- }$,
		$\hat{\calV}$ is a valid labeling for $\hat{M}$,
		and
		$p$ is an $(\hat{M}, I, \hat{\calV})$-compatible $c$-potential.
\end{proposition}

By deleting $\alpha_{-1} \beta_0$ from $M$,
it can happen that $\beta(\calP_0) \notin \calS(\hat{M}, \hat{I}, \hat{\calV}, p)$.
The following states that,
in such a case,
$\beta_0^-$ belongs to $\calS(\hat{M}, I, \hat{\calV}, p)$
and there is a path from $\beta_0^-$ to $\beta(\calP_0)$ in $\calG(\hat{\calV}, p)$
that is the concatenation of an outer path and an inner path for $(\hat{M}, I, \hat{\calV}, p)$.
Here we denote by $\calQ$ the maximal inner $+$-path in $\calG(\hat{\calV}, p)|_{\hat{M}}$ that starts with $\alpha_{-1}^+$;
Figure~\ref{fig:path} also describes $\calQ$.
\begin{lemma}\label{lem:S}
	If $\beta(\calP_0) \not\in \calS(\hat{M}, I, \hat{\calV}, p)$,
	then $\beta(\calP_0)$ belongs to $\calQ$, $\beta_0^- \in \calS(\hat{M}, I, \hat{\calV}, p)$,
	and the path $(\beta_0^- \alpha_{-1}^+) \circ \calQ(\beta(\calP_0)]$ exists in $\calG(\hat{\calV}, p)$.
\end{lemma}
\begin{proof}
	Assume that $\beta(\calP_0) = \hat{\beta}^\sigma$ does not belong to $\calS(\hat{M}, I, \hat{\calV}, p)$.
	It follows from $\hat{M} = M \setminus \set{ \beta_0 \alpha_{-1} } \cup \set{ \beta_0 \alpha_1 }$
	that $\deg_M(\beta) = \deg_{\hat{M}}(\beta)$ for each $\beta$.
	Since $\hat{\beta}^\sigma$ belongs to $\calS(M, I, \calV, p)$,
	we have $\deg_M(\hat{\beta}) \geq 1$ and $\hat{\beta}^{\sigma} \in \calC_{\beta^*}$,
	where $\beta_*$ is an end node of some connected component of $M \setminus I$ incident to a $\sigma$-edge.
	Indeed, otherwise $\deg_M(\hat{\beta}) = \deg_{\hat{M}}(\hat{\beta}) = 0$,
	which implies that $\hat{\beta}^\sigma$ belongs to $\calS(\hat{M}, I, \hat{\calV}, p)$.
	This is a contradiction to the assumption.
	We denote by $\calP$ the $\sigma$-path in $\calG(\calV, p)|_M$ from $\beta_*^\sigma$ to $\hat{\beta}^{\sigma}$.
	
	We can easily observe $\beta_*^\sigma \in \calU(\hat{M}, I)$.
    Thus
	the assumption $\hat{\beta}^\sigma \notin \calS(\hat{M}, I, \hat{\calV}, p)$
	implies that
	there is an edge in $\calP$ not belonging to $\calE(\hat{\calV}, p)|_{\hat{M}}$.
	We then show that such an edge is $\beta_0^- \alpha_{-1}^+$.
	By \cref{prop:=},
	$\hat{\calV}$ is a valid labeling for $\hat{M}$ and
	$p$ is an $(\hat{M}, I, \hat{\calV})$-compatible $c$-potential.
	Hence
	$\hat{\calV}$ is also a valid labeling for $M' = \hat{M} \setminus \set{ \beta_0 \alpha_1 }$ and
	$p$ is also an $(M', I, \hat{\calV})$-compatible $c$-potential.
	By \cref{lem:>}~(2),
	we obtain $\calE(\calV, p)|_{M'} = \calE(\hat{\calV}, p)|_{M'}$.
	By $M \setminus \hat{M} = M \setminus M' = \set{ \beta_0 \alpha_{-1} }$,
	we have
	\begin{align}\label{eq:E}
	\calE(\calV, p)|_{M} \setminus 
	\calE(\hat{\calV}, p)|_{\hat{M}} =
	\begin{cases}
	\set{ \beta_0^+ \alpha_{-1}^- } & \text{if $\beta_0 \alpha_{-1}$ is single-tight},\\
	\set{ \beta_0^+ \alpha_{-1}^-, \beta_0^- \alpha_{-1}^+ } & \text{if $\beta_0 \alpha_{-1}$ is double-tight}.
	\end{cases}
	\end{align}
	Since $\beta_0^+ \notin \calS(M, I, \calV, p)$ by (A3),
	the edge in $\calP$ not belonging to $\calE(\hat{\calV}, p)|_{\hat{M}}$ must be $\beta_0^- \alpha_{-1}^+$
	(and $\beta_0 \alpha_{-1}$ must be double-tight).
	
	It follows from $\beta_0^- \alpha_{-1}^+ \in \calP$
	that
	$\beta_*$ is an end node of $C_{\beta_0}$
	incident to a $-$-edge,
	$\sigma$ is equal to $-$,
	and $\calP$ is the concatenation of the $-$-path from $\beta_*^-$ to $\beta_0^-$, $\beta_0^- \alpha_{-1}^+$,
	and $\calQ(\beta(\calP_0)]$.
	Hence $\beta(\calP_0)$ belongs to $\calQ$.
	By $\calE(\calV, p)|_{M'} = \calE(\hat{\calV}, p)|_{M'}$,
	$\calG(\hat{\calV}, p)$ contains the $-$-path from $\beta_*^-$ to $\beta_0^-$ and $\calQ(\beta(\calP_0)]$.
	In particular, we have $\beta_0^- \in \calS(\hat{M}, I, \hat{\calV}, p)$ by the existence of the $-$-path.
    Moreover we have $\beta_0^- \alpha_{-1}^+ \in \calE(\hat{\calV}, p)$.
	Indeed, the $-$-space $(U_{\alpha_1}^+)^{\perp_{\alpha_1 \beta_0}}$ of $\beta_0$ with respect to $\hat{\calV}$
	is different from $V_{\beta_0}^+$ by \cref{lem:edge},
	and
	the $+$-space of $\alpha_{-1}$ with respect to $\hat{\calV}$ is $U_{\alpha_{-1}}^+$.
	Since $A_{\alpha_{-1} \beta_0}(U_{\alpha_{-1}}^+, V_{\beta_0}^+) = \set{ 0 }$, $\beta_0 \alpha_{-1}$ is rank-2, and $(U_{\alpha_1}^+)^{\perp_{\alpha_1 \beta_0}} \neq V_{\beta_0}^+$,
	we have $A_{\alpha_{-1} \beta_0}(U_{\alpha_{-1}}^+, (U_{\alpha_1}^+)^{\perp_{\alpha_1 \beta_0}}) \neq \set{ 0 }$,
	implying $\beta_0^- \alpha_{-1}^+ \in \calE(\hat{\calV}, p)$.
	Hence the path $(\beta_0^- \alpha_{-1}^+) \circ \calQ(\beta(\calP_0)]$ exists in $\calG(\hat{\calV}, p)$.
\end{proof}

We define $\hat{\calR}$ by
\begin{align*}
\hat{\calR} &\defeq
\begin{cases}
\calP_0 \circ \calQ_1 \circ \calP_1 \circ \cdots \circ \calP_{m-1} & \text{if $\beta(\calP_0) \in \calS(\hat{M}, I, \hat{\calV}, p)$},\\
(\beta_0^- \alpha_{-1}^+) \circ \calQ(\beta(\calP_0)] \circ \calP_0 \circ \calQ_1 \circ \calP_1 \circ \cdots \circ \calP_{m-1} & \text{if $\beta(\calP_0) \notin \calS(\hat{M}, I, \hat{\calV}, p)$}.
\end{cases}
\end{align*}
\begin{proposition}\label{prop:path:pseudo augmenting path}
	$\hat{\calR}$ is a pseudo augmenting path for $(\hat{M}, I, \hat{\calV}, p)$.
\end{proposition}
\begin{proof}
    This follows from \cref{lem:C_alpha}~(1) and \cref{lem:C_beta} as in the proof of \cref{prop:cycle =}.
    In particular, even when the connected component of $\hat{M}$ containing $\beta_0 \alpha_1$ forms a cycle,
    i.e., $C_{\beta_0} = C_{\alpha_1}$,
    \cref{rmk:cycle} enables us to follow the same argument in the proof of \cref{prop:cycle =}.
\end{proof}

By this update, $\theta$ does not increase and $\phi$ strictly decreases.
\begin{lemma}\label{lem:path theta varphi}
$\theta(\hat{M}, I, \hat{\calR}) \leq \theta(M, I, \calR)$ and $\phi(\hat{M}, I, \hat{\calR}) < \phi(M, I, \calR)$.
\end{lemma}
\begin{proof}
By the same argument as in the proof of \cref{lem:cycle theta varphi},
we obtain $\theta(\hat{M}, I, \hat{\calR}) \leq \theta(M, I, \calR)$,
and $\phi(\hat{M}, I, \hat{\calR}) < \phi(M, I, \calR)$ if $\beta(\calP_0) \in \calS(\hat{M}, I, \hat{\calV}, p)$.

Suppose $\beta(\calP_0) \notin \calS(\hat{M}, I, \hat{\calV}, p)$.
Then $|\hat{\calR}| = |\calR| - |\calQ_m| - |\calP_m| + |(\beta_0^- \alpha_{-1}^+) \circ \calQ(\beta(\calP_0)]|$ and $N_0(\hat{M}, I, \hat{\calR}) = N_0(M, I, \calR) - |(\beta_0^- \alpha_{-1}^+) \circ \calQ(\beta(\calP_0)]|$ by \cref{lem:S}.
Thus we obtain $\phi(\hat{M}, I, \hat{\calR}) < \phi(M, I, \calR) - |\calQ_m| < \phi(M, I, \calR)$.
\end{proof}

We update
\begin{align*}
M \leftarrow \hat{M}, \qquad
\calV \leftarrow \hat{\calV}(\calR),\qquad
\calR \leftarrow \hat{\calR}.
\end{align*}
This update can be done in $O(\min\set{\mu, \nu})$ time.
By \cref{prop:pseudo augmenting path,prop:=,prop:path:pseudo augmenting path},
the resulting $(M, I)$ is a matching-pair of size $k$, $\calV$ a valid labeling for $M$,
$p$ an $(M, I, \calV)$-compatible $c$-potential,
and
$\calR$ is an augmenting path for $(M, I, \calV, p)$.
Moreover, since $\alpha_1^+$ is deleted from $\calU(M, I)$ and $\alpha_{-1}^-$ is added to $\calU(M, I)$ in this update (by \cref{prop:cycle =})
and $\alpha(\calP_{m-1}) \in \calC(\alpha_{-1}^-)$,
$p$ satisfies (Zero)$'$ for $\calR$.
Return to the initial stage (\cref{sec:initial}).

\section*{Acknowledgments}
The author thanks Hiroshi Hirai for careful reading and helpful comments.
This work was supported by JSPS KAKENHI Grant Numbers JP17K00029, 20K23323, 20H05795, Japan.

\end{document}